%% file: MANOVA_bulk.tex
\documentclass[11pt,reqno,oneside]{amsart}
\usepackage[margin=1in]{geometry}
\usepackage{amsthm,amsmath,amssymb}
\usepackage{graphicx,bbm,booktabs,xr,enumerate}
\usepackage{hyperref}

\numberwithin{equation}{section}

\allowdisplaybreaks

\newtheorem{theorem}{Theorem}[section]
\newtheorem{proposition}[theorem]{Proposition}
\newtheorem{lemma}[theorem]{Lemma}
\newtheorem{corollary}[theorem]{Corollary}
\theoremstyle{definition}
\newtheorem*{definition*}{Definition}
\newtheorem{definition}[theorem]{Definition}

\newtheorem*{remark*}{Remark}
\newtheorem{remark}[theorem]{Remark}
\newtheorem{example}[theorem]{Example}

\newcommand{\R}{\mathbb{R}}
\newcommand{\C}{\mathbb{C}}
\newcommand{\Nor}{\mathcal{N}}
\newcommand{\Id}{\operatorname{Id}}
\newcommand{\1}{\mathbbm{1}}
\newcommand{\Tr}{\operatorname{Tr}}
\newcommand{\col}{\operatorname{col}}
\newcommand{\A}{\mathcal{A}}
\newcommand{\D}{\mathcal{D}}
\newcommand{\F}{\mathbf{F}}
\newcommand{\cR}{\mathcal{R}}
\newcommand{\E}{\mathbb{E}}
\renewcommand{\H}{\mathcal{H}}
\newcommand{\G}{\mathcal{G}}
\newcommand{\B}{\mathcal{B}}
\newcommand{\cC}{\mathcal{C}}
\newcommand{\diag}{\operatorname{diag}}
\newcommand{\e}{\mathbf{1}}
\newcommand{\J}{\mathbf{J}}
\newcommand{\I}{\mathcal{I}}
\newcommand{\N}{\mathbb{N}}

\newcommand{\Q}{\mathcal{Q}}
\renewcommand{\S}{\mathcal{S}}
\newcommand{\eps}{\varepsilon}
\newcommand{\rank}{\operatorname{rank}}

\renewcommand{\a}{\alpha}
\newcommand{\cI}{\mathcal{I}}
\renewcommand{\i}{\mathbf{i}}
\renewcommand{\j}{\mathbf{j}}

\renewcommand{\SS}{\mathrm{SS}}
\newcommand{\MS}{\mathrm{MS}}
\newcommand{\ii}{\mathfrak{q}}
\newcommand{\jj}{\mathfrak{q'}}
\newcommand{\cD}{\mathcal{D}}

\title[Eigenvalue Distributions of MANOVA Estimators]
{Eigenvalue distributions of variance components estimators in
high-dimensional random effects models}

\author{Zhou Fan}
\author{Iain Johnstone}
\address{Department of Statistics, Stanford University}
\email{zhoufan@stanford.edu, imj@stanford.edu}
\thanks{ZF is supported by a Hertz Foundation Fellowship and an NDSEG
Fellowship (DoD, Air Force Office of Scientific Research, 32 CFR
168a). ZF and IMJ are also supported in part by NIH grant R01 EB001988.}

\begin{document}
\maketitle

\begin{abstract}
We study the spectra of MANOVA estimators for
variance component covariance matrices in multivariate random effects models.
When the dimensionality of the observations is large and
comparable to the number of realizations of
each random effect, we show that the empirical spectra of such
estimators are well-approximated by deterministic laws. The Stieltjes transforms
of these laws are characterized by systems of fixed-point equations, which
are numerically solvable by a simple iterative procedure. Our proof uses
operator-valued free probability theory, and we establish a general asymptotic
freeness result for families
of rectangular orthogonally-invariant random matrices, which is of independent
interest. Our work is motivated by the estimation of components of covariance
between multiple phenotypic traits in quantitative genetics, and we specialize
our results to common experimental designs that arise in this application.
\end{abstract}

\section{Introduction}
Since the work of R.\ A.\ Fisher, random effects
linear models have played a foundational role in quantitative
genetics. Fisher described the decomposition of the variance of a
quantitative trait in a population into components, which
may be estimated by observing these traits in individuals of different
relations \cite{fisher}. One important motivation for estimating these
components is in predicting the
evolutionary response of the population to natural or artificial selection.
If an episode of selection changes the mean value of a trait in this
generation by $S$, the change $\Delta \mu$
inherited by the next generation is predicted by the breeders' equation
$\Delta \mu=\sigma_A^2(\sigma_z^2)^{-1}S$,
where $\sigma_z^2$ is the total population variance and
$\sigma_A^2$ is its additive genetic component. A common method of estimating
$\sigma_A^2$ is using a random effects
model with a suitable experimental design \cite{lynchwalsh}.

In reality, selection acting on a trait rarely only induces a response
in that single trait, but instead also affects genetically correlated traits
\cite{landearnold,phillipsarnold,blows}. Most of this correlation is likely due
to pleiotropy, the influence of a single gene on multiple traits,
and there is evidence that pleiotropic effects are widespread across the phenome
\cite{barton,walshblows,mcguiganetal,blowsmcguigan,blowsetal}. 
Letting $S \in \R^p$ denote the changes in mean values of $p$ traits in this
generation due to selection, the changes inherited by the next
generation are predicted by the multivariate breeders' equation
$\Delta \mu=GP^{-1}S$, where $P \in \R^{p \times p}$ is the total phenotypic
covariance of the traits and $G \in \R^{p \times p}$ is its additive
genetic component \cite{lande}. The response to selection may be understood via
the principal eigenvectors of $G$ and the alignment of the
``selection gradient'' $P^{-1}S$ with these eigenvectors.
Hence, there is significant interest in understanding the spectral
structure of $G$ \cite{kirkpatrick,walshblows,hineetal,blowsmcguigan}.
Analogously to the univariate setting, $G$ may be estimated by variance
components in multivariate random effects models.

Gene expression microarrays have enabled the measurements of thousands of
quantitative phenotypic traits in a single experimental study,
providing an opportunity to better understand the nature and extent of
pleiotropy and the effective dimensionality of possible evolutionary response
in the entire phenome of an
organism \cite{mcguiganetal,blowsetal}. However, the theory of
large random matrices \cite{paulaue} as well as numerical simulations
\cite{blowsmcguigan} both suggest that variance component matrices estimated in
these settings may exhibit significant spectral noise due to their high
dimensionality. In this work, we derive a characterization of the spectra
of such estimates.

We study the general multivariate random effects model
\begin{equation}\label{eq:mixedmodel}
Y=X\beta+\sum_{r=1}^k U_r\alpha_r,
\hspace{0.25in}\alpha_r \sim \Nor(0, \Id_{I_r} \otimes \Sigma_r).
\end{equation}
$Y \in \R^{n \times p}$ represents $n$ observations of $p$ traits,
modeled as a sum of fixed effects $X\beta$ and $k$ random effects
$U_1\alpha_1,\ldots,U_k\alpha_k$. (It is common to add a residual error term
$\eps$; for notational convenience we incorporate $\eps$ by allowing $U_k=\Id$
and $\alpha_k=\eps$.)
$X \in \R^{n \times m}$ and
$U_r \in \R^{n \times I_r}$ are known design and incidence matrices.
Each $\alpha_r \in \R^{I_r \times p}$ is an unobserved random matrix with
i.i.d.\ rows distributed as $\Nor(0,\Sigma_r)$, representing $I_r$ independent
realizations of the $r^\text{th}$ effect. The regression
coefficients $\beta \in \R^{m \times p}$ and variance components
$\Sigma_r \in \R^{p \times p}$ are unknown parameters. We study
estimators of $\Sigma_r$ that are quadratic in $Y$ and invariant to $\beta$,
i.e.\ estimators of the form
\begin{equation}\label{eq:quadraticestimator}
\hat{\Sigma}_r=Y^TB_rY, \hspace{0.25in} (B_rX=0)
\end{equation}
for symmetric matrices $B_r \in \R^{n \times n}$. In particular, model
(\ref{eq:mixedmodel}) encompasses nested and 
crossed classification designs, and (\ref{eq:quadraticestimator}) encompasses
MANOVA estimators and MINQUEs. We discuss examples
in Section \ref{sec:applications}. We consider the asymptotic regime where
$n,I_1,\ldots,I_k$ grow proportionally. For classification designs,
this means that the number of groups at the highest level of division scales
proportionally with $n$, and all further sub-divisions remain bounded in size.
This is the relevant regime for experiments that estimate components
of phenotypic covariance, from considerations of both experimental practicality 
and optimal design \cite{robertsona,robertsonb}.

Our main result shows that when $p$ is also comparable to $n$,
the spectra of estimators (\ref{eq:quadraticestimator}) are
accurately predicted by deterministic laws which
depend on the true variance components $\Sigma_1,\ldots,\Sigma_k$. We
characterize these laws by systems of fixed-point equations in their Stieltjes
transforms, which generalize the Marcenko-Pastur equation for the usual 
sample covariance matrix \cite{marcenkopastur}. These equations
may be solved numerically to approximate the probability density functions of
these laws.

For sample covariance matrices, the Marcenko-Pastur
equation underpins many procedures for inferring the
population spectrum \cite{elkaroui,mestre,raoetal,baichenyao,ledoitwolf} and
debiasing sample eigenvalues in ``spiked'' covariance models
\cite{baiksilverstein,paul,benaychgeorgesnadakuditi,baiyao}. Similar inferential
questions are of interest in variance components applications, and we hope 
that our result will enable the study of such questions.

\subsection{Main result}\label{sec:result}
To present an analogy, we review the Marcenko-Pastur equation for
sample covariance matrices: Given $Y \in \R^{n \times p}$ consisting
of $n$ i.i.d.\ rows with distribution $\Nor(0,\Sigma)$, consider the sample
covariance $\hat{\Sigma}=n^{-1}Y^TY$. When $\Sigma=\Id$, the spectrum
of $\hat{\Sigma}$ is well-approximated by the Marcenko-Pastur law
\cite{marcenkopastur,silversteinbai}. More
generally, for any $\Sigma$, the spectrum of
$\hat{\Sigma}$ is predicted by the Marcenko-Pastur equation:
\begin{theorem}[\cite{marcenkopastur,silversteinbai}]\label{thm:mp}
Let $\mu_{\hat{\Sigma}}=p^{-1}\sum_{i=1}^p \delta_{\lambda_i(\hat{\Sigma})}$ 
denote the empirical spectral measure of $\hat{\Sigma}$.
Suppose $n,p \to \infty$ such that $c<p/n<C$ and
$\|\Sigma\|<C$ for some constants $C,c>0$. Then for each
$z \in \C^+$, there exists a unique value $m_0(z) \in \C^+$ satisfying
\begin{equation}\label{eq:mpequation}
m_0(z)=\frac{1}{p}\Tr\left[\left(\left(1-\frac{p}{n}
-\frac{p}{n}zm_0(z)\right)\Sigma-z\Id_p\right)^{-1}\right],
\end{equation}
and $m_0:\C^+ \to \C^+$
defines the Stieltjes transform of a ($n,p,\Sigma$-dependent) probability
measure $\mu_0$ on $\R$ such that $\mu_{\hat{\Sigma}}-\mu_0 \to 0$ weakly
almost surely.
\end{theorem}
\noindent The Stieltjes transform $m_0$ determines $\mu_0$ via the
Stieltjes inversion formula.

Theorem \ref{thm:mp} is usually stated in an alternative form, assuming
convergence of $p/n$ to $\gamma \in (0,\infty)$ and of the spectrum of
$\Sigma$ to a weak limit $\mu^*$. In this case $\mu_{\hat{\Sigma}}$ converges
to a fixed weak limit $\mu_0$ depending on $\gamma$ and $\mu^*$. We have stated
this theorem instead in a ``deterministic equivalent'' form
\cite{hachemetal,couilletetal}, where $\mu_0$ is defined by the finite-sample
quantities $p/n$ and $\Sigma$. This form is arguably more closely tied to
applications, since one typically computes the analytic
prediction for $\mu_{\hat{\Sigma}}$ directly from these finite-sample
quantities, rather than first passing to an abstract limit. (See also
the discussion in \cite{elkaroui}.)

The main result of our paper is the following extension of Theorem
\ref{thm:mp} to the setting of model (\ref{eq:mixedmodel}).
Consider $\hat{\Sigma}=Y^TBY$
for symmetric $B \in \R^{n \times n}$ satisfying $BX=0$.
Define $I_+=\sum_{r=1}^k I_r$,
\[U=\begin{pmatrix}
\sqrt{I_1}U_1 \mid \sqrt{I_2}U_2 \mid \cdots \mid \sqrt{I_k}U_k\end{pmatrix}
\in \R^{n \times I_+},\qquad
F=U^TBU \in \R^{I_+ \times I_+}.\]
For any $F \in \C^{I_+ \times I_+}$, let $\Tr_r F$ denote the trace of its
$(r,r)$ block in the $k \times k$ block decomposition corresponding to
$\C^{I_+}=\C^{I_1} \oplus \cdots \oplus \C^{I_k}$.
For $a=(a_1,\ldots,a_k)$ and $b=(b_1,\ldots,b_k)$, define
\[D(a)=\diag(a_1\Id_{I_1},\ldots,a_k\Id_{I_k}) \in \C^{I_+ \times I_+},
\quad b \cdot \Sigma=b_1 \Sigma_1+\ldots+b_k \Sigma_k.\]
We state our result also in deterministic equivalent form, which
avoids imposing ``joint convergence'' assumptions on
$\Sigma_1,\ldots,\Sigma_k$:

\begin{theorem}\label{thm:bulkdistribution}
Suppose $n,p,I_1,\ldots,I_k \to \infty$ such that $c<p/n<C$, $c<I_r/n<C$,
$n\|B\|<C$, $\|\Sigma_r\|<C$, and $\|U_r\|<C$ for each $r=1,\ldots,k$ and some
constants $C,c>0$.
Then for each $z \in \C^+$, there exist unique $z$-dependent values
$a_1,\ldots,a_k \in \C^+ \cup \{0\}$ and
$b_1,\ldots,b_k \in \overline{\C^+}$
that satisfy, for $r=1,\ldots,k$, the equations
\begin{align}
a_r&=-\tfrac{1}{I_r}\Tr\left((z\Id_p+b \cdot \Sigma)^{-1}\Sigma_r\right)
\label{eq:arecursion},\\
b_r&=-\tfrac{1}{I_r}\Tr_r \left([\Id_{I_+}+FD(a)]^{-1}F\right).
\label{eq:brecursion}
\end{align}
The function $m_0:\C^+ \to \C^+$ given by
\begin{equation}\label{eq:m0}
m_0(z)=-\tfrac{1}{p}\Tr\left((z\Id_p+b \cdot \Sigma)^{-1}\right)
\end{equation}
defines the Stieltjes transform of a probability
measure $\mu_0$ on $\R$ such that $\mu_{\hat{\Sigma}}-\mu_0 \to 0$ weakly
almost surely.
\end{theorem}

Note that $\mu_0$ is a deterministic measure defined by
$\Sigma_1,\ldots,\Sigma_k$ and the structure of the model, whereas
$\mu_{\hat{\Sigma}}$ is random and depends on the data $Y$.

\begin{remark}\label{remark:MP}
When $Y$ has $n$ i.i.d.\ rows, the sample covariance $\hat{\Sigma}=n^{-1}Y^TY$
corresponds to the special case of (\ref{eq:mixedmodel}) with $k=1$,
$U_1=\Id$, $\Sigma_1=\Sigma$, and $B=n^{-1}\Id_n$. In this case, equations
(\ref{eq:arecursion}--\ref{eq:m0}) reduce to
\begin{equation}\label{eq:MPab}
a_1=-\tfrac{1}{n}
\Tr\left((z\Id_p+b_1\Sigma)^{-1}\Sigma\right),\qquad
b_1=-1/(1+a_1),
\end{equation}
\begin{equation}\label{eq:MPm0}
m_0(z)=-\tfrac{1}{p}\Tr\left((z\Id_p+b_1\Sigma)^{-1}\right),
\end{equation}
which imply (by the identity $A^{-1}-(A+B)^{-1}=A^{-1}B(A+B)^{-1}$)
\begin{align*}
-1-\frac{1}{b_1}=a_1
&=-\frac{z}{nb_1}\Tr\left((z\Id_p)^{-1}
-(z\Id_p+b_1\Sigma)^{-1}\right)\\
&=-\frac{p}{nb_1}+\frac{pzm_0(z)}{nb_1}.
\end{align*}
Hence $b_1=-1+(p/n)+(p/n)zm_0(z)$. Together with the
above expression for $m_0(z)$, this recovers the Marcenko-Pastur equation
(\ref{eq:mpequation}).
\end{remark}

In most cases, (\ref{eq:arecursion}--\ref{eq:m0}) do not admit a
closed-form solution in $a_1,\ldots,a_k$, $b_1,\ldots,b_k$, and $m_0(z)$.
However, these equations may be solved numerically:
\begin{theorem}\label{thm:algorithm}
For each $z \in \C^+$, the values $a_r$ and $b_r$
in Theorem \ref{thm:bulkdistribution} are the limits, as $t \to \infty$, of
the iterative procedure which arbitrarily initializes
$b_1^{(0)},\ldots,b_k^{(0)} \in \overline{\C^+}$
and iteratively computes (for $t=0,1,2,\ldots$) $a_r^{(t)}$ from $b_r^{(t)}$
using (\ref{eq:arecursion}) and $b_r^{(t+1)}$ from $a_r^{(t)}$ using
(\ref{eq:brecursion}).
\end{theorem}
This yields a method for computing the
density of $\mu_0$ in Theorem \ref{thm:bulkdistribution}:
By the Stieltjes inversion formula, the density at $x \in \R$ is
approximately $\pi^{-1}\Im m_0(x+i\eps)$ for small $\eps$, which we may
compute from $b_1,\ldots,b_k$ using the above procedure. A software
implementation is available upon request.

Theorems \ref{thm:bulkdistribution} and \ref{thm:algorithm} are inspired by the
study of similar models for wireless communication channels.
In particular, \cite{couilletetal} establishes analogous results for the matrix
\[S+\sum_{r=1}^k \Sigma_r^{1/2}G_r^*B_rG_r\Sigma_r^{1/2}\]
where $B_r \in \C^{n_r \times n_r}$ are positive semidefinite and diagonal.
Earlier work of \cite[Theorem 1.2.1]{zhang} considers $k=1$, $S=0$, and
arbitrary Hermitian $B_1$. For $S=0$, this model is encompassed by our
Theorem \ref{thm:Wdistribution}; however, we remark
that these works do not require Gaussian $G_r$.
In \cite{dupuyloubaton} and the earlier work of \cite{moustakassimon}
using the replica method, the authors study the model
\[\sum_{r,s=1}^k \Sigma_r^{1/2}G_r^*T_r^{1/2}T_s^{1/2}
G_s\Sigma_s^{1/2},\]
where $\Sigma_r,T_r$ are positive semidefinite and $G_r$ are complex
Gaussian. This model is similar to ours, and we recover their result
in Theorem \ref{thm:Wdistribution} using a
different proof. We note that \cite{dupuyloubaton} proves only mean convergence,
whereas we also control the variance and prove convergence a.s.
We use a free probability approach, which may be easier to generalize to
other models.

\subsection{Overview of proof}\label{subsec:proofsketch}
We use the tools of operator-valued free probability theory, in particular
rectangular
probability spaces and their connection to operator-valued freeness developed 
in \cite{benaychgeorges} and the free deterministic equivalents approach
of \cite{speichervargas}.

Let us write $\alpha_r$ in (\ref{eq:mixedmodel}) as
$\alpha_r=G_r\Sigma_r^{1/2}$, where $G_r \in \R^{I_r \times p}$ has
i.i.d.\ $\Nor(0,1)$ entries. Then $\hat{\Sigma}=Y^TBY$ takes the form
\[\hat{\Sigma}=\sum_{r,s=1}^k
\Sigma_r^{1/2}G_r^TU_r^TBU_sG_s\Sigma_s^{1/2}.\]
We observe the following:
If $O_0,O_1,\ldots,O_k \in \R^{p \times p}$ and $O_{k+r} \in
\R^{I_r \times I_r}$ for each $r=1,\ldots,k$ are real orthogonal matrices,
then by rotational invariance of $G_r$, $\mu_{\hat{\Sigma}}$
remains invariant in law under the transformations
\[\Sigma_r^{1/2} \mapsto H_r:=O_r^T\Sigma_r^{1/2}O_0,\;\;
U_r^TBU_s \mapsto F_{rs}:=O_{k+r}^TU_r^TBU_sO_{k+s}.\]
Hence we may equivalently consider the matrix
\begin{equation}\label{eq:equivalentSigmahat}
W=\sum_{r,s=1}^k H_r^TG_r^TF_{rs}G_sH_s
\end{equation}
for $O_0,\ldots,O_{2k}$ independent and Haar-distributed.
The families $\{F_{rs}\}$, $\{G_r\}$, $\{H_r\}$ are
independent of each other, with each family satisfying a
certain joint orthogonal invariance in law (formalized in
Section \ref{sec:freeapprox}).

Following \cite{benaychgeorges}, we embed the matrices $\{F_{rs}\}$, $\{G_r\}$,
$\{H_r\}$ into a square matrix space $\C^{N \times N}$. We then consider
deterministic elements $\{f_{rs}\}$, $\{g_r\}$, $\{h_r\}$ in a
von Neumann algebra $\A$ with tracial state $\tau$, such that
these elements model the embedded matrices, and
$\{f_{rs}\}$, $\{g_r\}$, and $\{h_r\}$ are free with amalgamation
over a diagonal sub-algebra of projections in $\A$. We follow the deterministic
equivalents approach of \cite{speichervargas} and allow $(\A,\tau)$ and
$\{f_{rs}\},\{g_r\},\{h_r\}$ to also depend on $n$ and $p$.

Our proof of Theorem \ref{thm:bulkdistribution} consists of two steps:
\begin{enumerate}[1.]
\item For independent, jointly orthogonally-invariant
families of random matrices, we formalize the notion of a free
deterministic equivalent and prove an asymptotic freeness result
establishing validity of this approximation.
\item For our specific model of interest, we show that the Stieltjes transform
of $w:=\sum_{r,s} h_r^*g_r^*f_{rs}g_sh_s$ in the free model satisfies 
the equations (\ref{eq:arecursion}--\ref{eq:m0}).
\end{enumerate}
\noindent We establish separately the existence and uniqueness of the fixed
point to (\ref{eq:arecursion}--\ref{eq:brecursion}) using a contractive
mapping argument and uniqueness of analytic continuation. This implies that
the Stieltjes transform of $w$ in step 2 is uniquely determined by
(\ref{eq:arecursion}--\ref{eq:m0}), which implies by step 1 that
(\ref{eq:arecursion}--\ref{eq:m0}) asymptotically determine
the Stieltjes transform of $W$.

An advantage of this approach is that the approximation is separated
from the computation of the approximating
measure $\mu_0$. The approximation in step 1 is
general---it may be applied to other matrix models arising in statistics
and engineering, and it follows a line of work establishing asymptotic
freeness of random matrices
\cite{voiculescuinv,dykema,voiculescustrengthened,hiaipetz,collins,
collinssnaidy,benaychgeorges,speichervargas}.
In the computation in step 2,
the Stieltjes transform of $w$ is exactly (rather than
approximately) described by (\ref{eq:arecursion}--\ref{eq:m0}).
The computation is thus entirely algebraic, using free cumulant tools of
\cite{nicaetal,speichervargas}, and it does not require analytic approximation
arguments or bounds.

\subsection{Outline of paper}
Section \ref{sec:applications} specializes Theorem \ref{thm:bulkdistribution} to
several classification designs that arise in applications.
Section \ref{sec:freeapprox} reviews
free probability theory and states the asymptotic freeness
result. Section \ref{sec:modelcomputation} performs the computation in 
the free model. The remainder of the proof and other details are
deferred to the supplementary appendices.

\subsection{Notation}
$\|\cdot\|$ denotes the $l_2$ norm for vectors and the $l_2 \to
l_2$ operator norm for matrices. $M^T$, $M^*$, and $\Tr M=\sum_i M_{ii}$
denote the transpose, conjugate-transpose, and trace of $M$.
$\Id_n$ denotes the identity matrix of size $n$. $\diag(A_1,\ldots,A_k)$
denotes the block-diagonal matrix with blocks $A_1,\ldots,A_k$.
$\C^+=\{z \in \C:\Im z>0\}$ and $\overline{\C^+}=\{z \in \C:\Im z \geq 0\}$
denote the open and closed half-planes.

For a $*$-algebra $\A$ and elements
$(a_i)_{i \in \I}$ of $\A$, $\langle a_i:i \in \I \rangle$ denotes the
sub-$*$-algebra generated by $(a_i)_{i \in \I}$. We write $\langle \{a_i\}
\rangle$ if the index set $\I$ is clear from context. If $\A$ is a von Neumann
algebra, $\langle \{a_i\} \rangle_{W^*}$ denotes the generated von Neumann
sub-algebra, i.e.\ the ultraweak closure of $\langle \{a_i\} \rangle$,
and $\|a_i\|$ denotes the $C^*$-norm.

\subsection*{Acknowledgments}
We thank Mark Blows for introducing us to this problem and for much help
in guiding us through understanding the quantitative genetics applications.

\section{Specialization to classification designs}\label{sec:applications}
\input{Section_2_fix}

\section{Operator-valued free probability}\label{sec:freeapprox}
\input{Section_3_fix}

\section{Computation in the free model}\label{sec:modelcomputation}
\input{Section_4_fix}

\appendix

\section{Details for classification designs}\label{appendix:applications}
\input{Appendix_A_fix}

\section{Proof of asymptotic freeness}\label{appendix:freeapprox}
\input{Appendix_B_fix}

\section{Analysis of fixed-point equations}\label{appendix:fixedpoint}
\input{Appendix_C_fix}

\section{Free probability constructions}\label{appendix:freeprobdetails}
\input{Appendix_D_fix}

\section{Marcenko-Pastur case}\label{appendix:march-past-case}
\input{Appendix_E}

\bibliography{references}{}
\bibliographystyle{alpha}
\end{document}

%% file: Section_2_fix.tex
The form (\ref{eq:quadraticestimator}) encompasses MANOVA estimators, which
solve for $\Sigma_1,\ldots,\Sigma_k$ in the system of equations
$Y^TM_rY=\E[Y^TM_rY]$ for a certain choice of symmetric matrices
$M_1,\ldots,M_k \in \R^{n \times n}$ \cite[Chapter 5.2]{searleetal}. From
(\ref{eq:mixedmodel}), the identity
$\E[\alpha_s^T M\alpha_s]=(\Tr M)\Sigma_s$ for any matrix $M$, and independence
of $\alpha_r$, we get
\[\E[Y^TM_rY]=\sum_{s=1}^k
\E[\alpha_s^TU_s^TM_rU_s\alpha_s]
=\sum_{s=1}^k \Tr(U_s^TM_rU_s)\Sigma_s.\]
Hence each MANOVA estimate $\hat{\Sigma}_r$ takes the form
(\ref{eq:quadraticestimator}), where
$B_r$ is a linear combination of $M_1,\ldots,M_k$.

In balanced or fully-nested classification designs, standard choices
for $M_1,\ldots,M_k$ project onto subspaces of $\R^n$ such that each
$Y^TM_rY$ corresponds to a ``sum-of-squares''. We may simplify
(\ref{eq:brecursion}) in such settings by analytically computing the matrix
inverse and block trace. We provide several examples below, 
deferring matrix algebra details and a more general procedure
for obtaining such simplifications to Appendix \ref{appendix:applications}.

For more general designs and models, $M_1,\ldots,M_k$ may be ad-hoc, although
Theorem \ref{thm:bulkdistribution} still applies to such estimators. 
The theorem also applies to MINQUEs \cite{rao,lamotte} in these
settings, which prescribe a specific form for $B \in \R^{n \times n}$ based on a
variance minimization criterion.

\subsection{One-way classification}\label{subsec:oneway}
$\{Y_{i,j} \in \R^p:1 \leq i \leq I,1 \leq j \leq J_i\}$ represent
observations of $p$ traits across $n=\sum_{i=1}^I J_i$
samples, belonging to $I$ groups of sizes $J_1,\ldots,J_I$.
The data are modeled as
\begin{equation}\label{eq:onewaymodel}
Y_{i,j}=\mu+\alpha_i+\eps_{i,j},
\end{equation}
where $\mu \in \R^p$ is a vector of population mean values,
$\alpha_i \sim \Nor(0,\Sigma_1)$ are i.i.d.\ random group effects,
and $\eps_{i,j} \sim \Nor(0,\Sigma_2)$ are i.i.d.\ residual errors.
In quantitative genetics, this is the model for the half-sib experimental
design and also for the standard twin study, where groups correspond to
half-siblings or twin pairs \cite{lynchwalsh}.

Defining the sums-of-squares
\[\SS_1=\sum_{i=1}^I J_i(\bar{Y}_i-\bar{Y})(\bar{Y}_i-\bar{Y})^T,
\qquad \SS_2=\sum_{i=1}^I \sum_{j=1}^{J_i} (Y_{i,j}-\bar{Y}_i)
(Y_{i,j}-\bar{Y}_i)^T,\]
where $\bar{Y}_i \in \R^p$ and $\bar{Y} \in \R^p$ denote the mean in the
$i^\text{th}$ group and of all samples, respectively,
the standard MANOVA estimators are given \cite[Chapter 3.6]{searleetal} by
\begin{equation}\label{eq:onewayestimators}
\hat{\Sigma}_1=\frac{1}{K}
\left(\frac{1}{I-1}\SS_1-\frac{1}{n-I}\SS_2\right),\hspace{0.1in}
\hat{\Sigma}_2=\frac{1}{n-I}\SS_2,
\end{equation}
where $K=(n-\frac{1}{n} \sum_{i=1}^I J_i^2)/(I-1)$.
The balanced case corresponds to $J_1=\ldots=J_I=K$.
Theorem \ref{thm:bulkdistribution} yields the following
corollary:
\begin{corollary}\label{cor:oneway}
Assume $p,n,I \to \infty$ such that $c<p/n<C$, $I/n>c$, $(n-I)/n>c$,
$\max_{i=1}^I J_i<C$, $\|\Sigma_1\|<C$, and $\|\Sigma_2\|<C$ for some $C,c>0$.
Denote $I_1=I$ and $I_2=n$. Then:
\begin{enumerate}[(a)]
\item For $\hat{\Sigma}_1$, $\mu_{\hat{\Sigma}_1}-\mu_0 \to 0$
weakly a.s.\ where $\mu_0$ has Stieltjes transform $m_0(z)$ determined by
\begin{align*}
a_s&=-\tfrac{1}{I_s}
\Tr\left((z\Id+b_1\Sigma_1+b_2\Sigma_2)^{-1}\Sigma_s\right)
\qquad \text{for } s=1,2,\\
b_1&=-{\textstyle \sum}_{i=1}^I \tfrac{J_i}{KI+IJ_ia_1+na_2},\quad
b_2=\tfrac{n-I}{K(n-I)-na_2}-{\textstyle \sum}_{i=1}^I
\tfrac{1}{KI+IJ_ia_1+na_2},\\
m_0(z)&=-\tfrac{1}{p}\Tr\left((z\Id+b_1\Sigma_1+b_2\Sigma_2)^{-1}
\right).
\end{align*}
\item For $\hat{\Sigma}_2$, $\mu_{\hat{\Sigma}_2}-\mu_0 \to 0$
weakly a.s.\ where $\mu_0$ has Stieltjes transform $m_0(z)$ determined by
\[a_2=-\tfrac{1}{n}
\Tr\left((z\Id+b_2\Sigma_2)^{-1}\Sigma_2\right),\qquad
b_2=-\tfrac{n-I}{n-I+na_2},\]
\[m_0(z)=-\tfrac{1}{p}\Tr\left((z\Id+b_2\Sigma_2)^{-1}\right).\]
\end{enumerate}
\end{corollary}

\noindent ``Determined by'' is in the sense of
Theorem \ref{thm:bulkdistribution},
i.e.\ for each $z \in \C^+$ there exists a unique solution to these equations
with $a_s \in \C^+ \cup \{0\}$,
$b_s \in \overline{\C^+}$, and $m_0(z) \in \C^+$. This system may
be solved by the procedure of Theorem \ref{thm:algorithm}.

Figure \ref{fig:oneway} displays the simulated spectrum of $\hat{\Sigma}_1$ in
various settings. This spectrum depends on both $\Sigma_1$ and $\Sigma_2$.
Overlaid on each histogram is the density of $\mu_0$,
approximated as $f(x)=\pi^{-1}\Im m_0(x+0.0001i)$ and computed using the
procedure of Theorem \ref{thm:algorithm}.

For $\hat{\Sigma}_2$ (but not $\hat{\Sigma}_1$), as in Remark \ref{remark:MP},
the three equations of Corollary \ref{cor:oneway}(b) may be simplified to the
single Marcenko-Pastur equation for population covariance $\Sigma_2$.
This also follows directly from the observation that $\hat{\Sigma}_2$ is equal
in law to $\eps^T\pi \eps$ where $\eps \in \R^{n \times p}$ is the matrix of
residual errors and $\pi$ is a normalized projection onto a space of
dimensionality $n-I$. This phenomenon holds generally for the MANOVA estimate
of the residual error covariance in usual classification designs.

\begin{figure}
\includegraphics[width=0.35\textwidth]{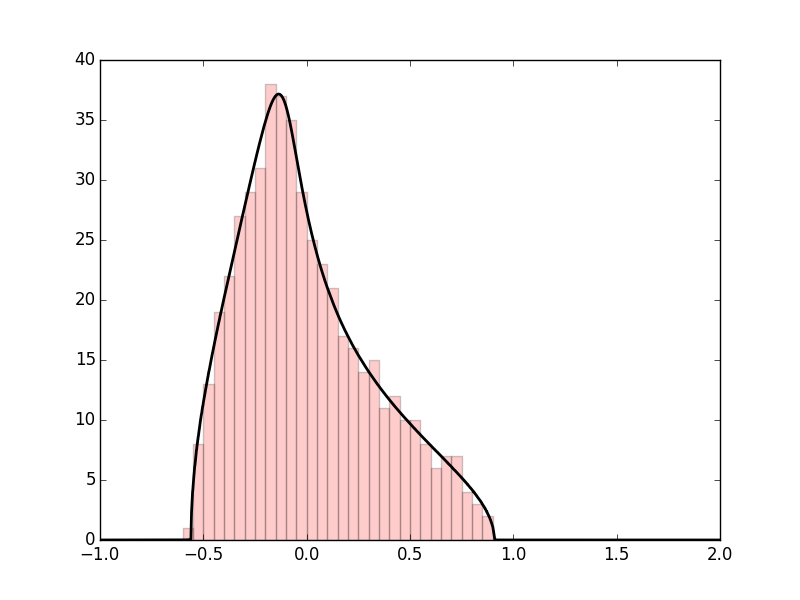}%
\includegraphics[width=0.35\textwidth]{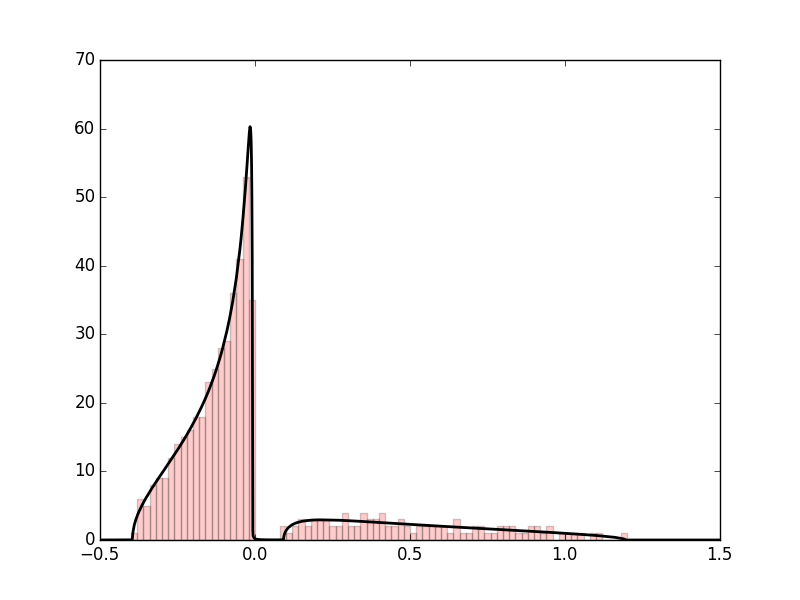}
\includegraphics[width=0.35\textwidth]{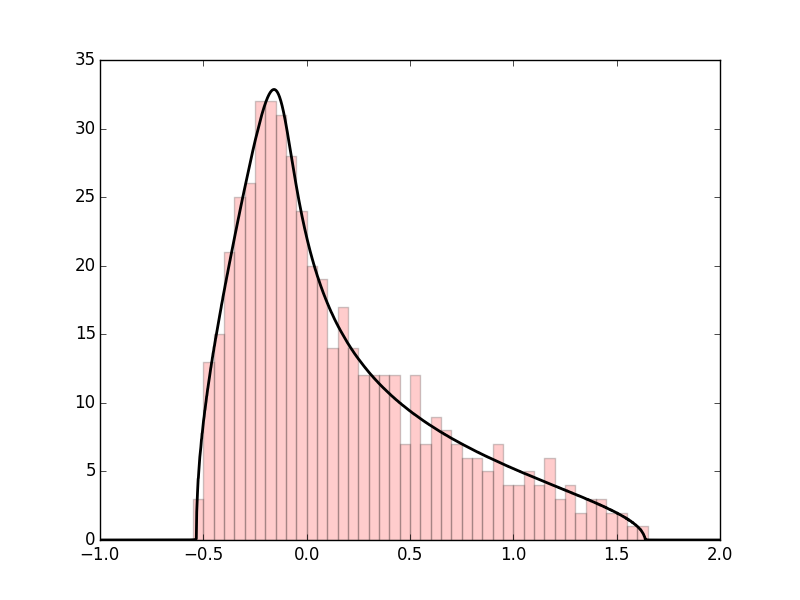}%
\includegraphics[width=0.35\textwidth]{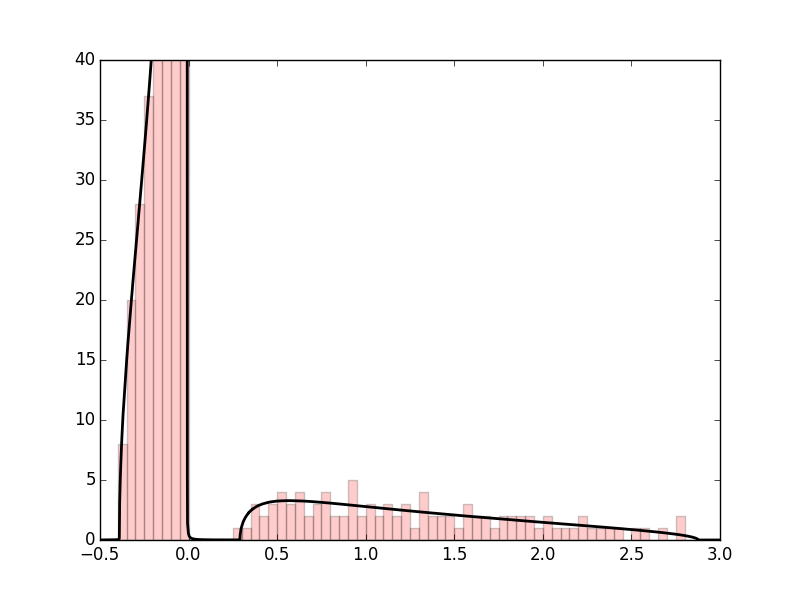}%
\caption{Simulated spectrum of $\hat{\Sigma}_1$ for the balanced
one-way classification model, $p=500$, with theoretical predictions of
Corollary \ref{cor:oneway} overlaid in black. Left: 400 groups of size 4.
Right: 100 groups of size 8. Top: $\Sigma_1=0$, $\Sigma_2=\Id$. Bottom:
$\Sigma_1$ with equally spaced eigenvalues in $[0,0.3]$, $\Sigma_2=\Id$.}
\label{fig:oneway}
\end{figure}

\subsection{Balanced nested classification}\label{subsec:nested}
$\{Y_{j_1,\ldots,j_k} \in \R^p:1 \leq j_1 \leq J_1,\,\ldots,\,
1 \leq j_k \leq J_k\}$ are
observations of $p$ traits across $n=J_1J_2\ldots J_k$
samples. The samples are divided into $J_1 \geq 2$ groups of equal
size $J_2\ldots J_k$, the samples within each group are further divided
into $J_2 \geq 2$ subgroups of equal size $J_3 \ldots J_k$, etc., and
there are $J_k \geq 2$ samples in each subgroup at the finest level of division.
The data are modeled as
\begin{equation}\label{eq:balancednestedmodel}
Y_{j_1,\ldots,j_k}=\mu+\alpha^{(1)}_{j_1}+\alpha^{(2)}_{j_1,j_2}
+\ldots+\alpha^{(k-1)}_{j_1,\ldots,j_{k-1}}+\eps_{j_1,\ldots,j_k},
\end{equation}
where $\mu \in \R^p$ is the population mean,
$\alpha^{(r)}_{j_1,\ldots,j_r} \sim \Nor(0,\Sigma_r)$ are i.i.d.\ group
effects for the $r^\text{th}$ level of grouping, and
$\eps_{j_1,\ldots,j_k} \sim \Nor(0,\Sigma_k)$ are i.i.d.\ residual
errors. The case $k=2$ is the one-way classification model of
Section \ref{subsec:oneway} when the design is balanced.
The two-way model ($k=3$) is the model for the full-sib half-sib design
in which outer groups correspond
to half-siblings and inner groups to full siblings. It is also the model for
the monozygotic-twin half-sib design, in which outer groups
correspond to offspring of one of two twins, and inner groups to offspring of
one twin in the pair \cite{lynchwalsh}.

Sums-of-squares and MANOVA estimators $\hat{\Sigma}_r$ for $\Sigma_r$
are defined analogously to the one-way model of Section
\ref{subsec:oneway}; we review these definitions in
Appendix \ref{appendix:applications}. Theorem \ref{thm:bulkdistribution} yields
the following corollary for these estimators:
\begin{corollary}\label{cor:nested}
Fix $J_2,\ldots,J_k \geq 2$, let $n=J_1J_2\ldots J_k$, and assume
$p,n,J_1 \to \infty$ such that $c<p/n<C$
and $\|\Sigma_r\|<C$ for all $r=1,\ldots,k$ and some $C,c>0$.
Then for any $r \in \{1,\ldots,k\}$,
$\mu_{\hat{\Sigma}_r}-\mu_0 \to 0$ weakly a.s.\ where $\mu_0$ has Stieltjes
transform $m_0(z)$ determined by
\begin{align*}
a_s&=-\tfrac{1}{J_1\ldots J_s}\Tr\left(
(z\Id+b_r\Sigma_r+\ldots+b_k\Sigma_k)^{-1}\Sigma_s\right)
\quad \text{for }s=r,\ldots,k,\\
b_s&=\begin{cases} -\frac{J_r-1}{J_r-1+J_r\sum_{j=r}^k a_j} & \text{if } s=r,\\
-\frac{1}{J_{r+1}\ldots J_s}\left(\frac{J_r-1}
{J_r-1+J_r\sum_{j=r}^k a_j}
-\frac{J_{r+1}-1}{J_{r+1}-1-\sum_{j=r+1}^k a_j}\right) & \text{if } s \geq r+1,
\end{cases}\\
m_0(z)&=-\tfrac{1}{p}\Tr\left((z\Id+b_r\Sigma_r+\ldots+b_k\Sigma_k)^{-1}\right).
\end{align*}
\end{corollary}

\subsection{Replicated crossed two-way classification}\label{subsec:crossed}
$\{Y_{i,j,k,l} \in \R^p:1 \leq i \leq I,1 \leq j \leq J, 1 \leq k \leq K,
1 \leq l \leq L\}$ represent observations across $n=IJKL$ samples.
The samples belong to $I$ replicated experiments of a $J \times K$ crossed
design with fixed numbers $J$ and $K$ of levels for two factors, and with
$L$ samples in each replicate $i$ corresponding to each level cross
$j \times k$. The data are modeled as
\begin{equation}\label{eq:crossedtwowaymodel}
Y_{i,j,k,l}=\mu+\alpha_i+\beta_{i,j}+\gamma_{i,k}+\delta_{i,j,k}
+\eps_{i,j,k,l},
\end{equation}
where $\mu \in \R^p$ is the population mean, $\alpha_i
\sim \Nor(0,\Sigma_1)$ are replicate effects,
$\beta_{i,j} \sim \Nor(0,\Sigma_2)$ and $\gamma_{i,k}
\sim \Nor(0,\Sigma_3)$ are effects for the two factors,
$\delta_{i,j,k} \sim \Nor(0,\Sigma_4)$ are
effects for the factor interactions, and
$\eps_{i,j,k,l} \sim \Nor(0,\Sigma_5)$ are residual errors.
This crossed $J \times K$ design corresponds to the
Comstock-Robinson model or North Carolina Design II commonly used in plant
studies, in which each of $J$ males is mated to each of
$K$ females. We consider the replicated setting with small $J,K,L$ and large
$I$, as is often done in practice for reasons of experimental design
\cite{lynchwalsh}.

Definitions of MANOVA estimators $\hat{\Sigma}_1,\ldots,\hat{\Sigma}_5$
are reviewed in Appendix \ref{appendix:applications}.
Theorem \ref{thm:bulkdistribution} yields the following for, e.g.,
the factor effect estimate $\hat{\Sigma}_2$.
\begin{corollary}\label{cor:crossedtwoway}
Fix $J,K,L \geq 2$, let $n=IJKL$, and assume $p,n,I \to \infty$ such that
$c<p/n<C$ and $\|\Sigma_r\| \leq C$ for each $r=1,\ldots,5$ and some $C,c>0$.
Denote $I_2=IJ$, $I_4=IJK$, and $I_5=n$. Then
$\mu_{\hat{\Sigma}_2}-\mu_0 \to 0$ weakly
a.s.\ where $\mu_0$ has Stieltjes transform $m_0(z)$ determined by
\begin{align*}
a_s&=-\tfrac{1}{I_s}\Tr\left((z\Id+b_2\Sigma_2
+b_4\Sigma_4+b_5\Sigma_5)^{-1}
\Sigma_s\right) \qquad \text{for }s=2,4,5,\\
b_2&=-\tfrac{J-1}{J-1+J(a_2+a_4+a_5)},\\
b_4&=-\tfrac{1}{K}\left(\tfrac{J-1}{J-1+J(a_2+a_4+a_5)}
-\tfrac{(J-1)(K-1)}{(J-1)(K-1)-J(a_4+a_5)}\right),\\
b_5&=\tfrac{1}{L}b_4,\\
m_0(z)&=-\tfrac{1}{p}\Tr\left((z\Id+b_2\Sigma_2
+b_4\Sigma_4+b_5\Sigma_5)^{-1}\right).
\end{align*}
\end{corollary}
Appendix \ref{appendix:applications} discusses how to obtain analogous
results for $\hat{\Sigma}_1,\hat{\Sigma}_3,\hat{\Sigma}_4,\hat{\Sigma}_5$.

%% file: Section_3_fix.tex
\subsection{Background}
We review definitions from operator-valued free probability theory and
its application to rectangular random matrices, drawn from
\cite{voiculescubook,voiculescu,benaychgeorges}.

\begin{definition*}\label{def:noncommutativespace}
A {\bf non-commutative probability space} $(\A,\tau)$ is a unital
$*$-algebra $\A$ over $\C$ and a $*$-linear functional $\tau:\A \to \C$ called
the {\bf trace} that satisfies, for all $a,b \in \A$ and for $1_\A \in \A$ the
multiplicative unit,
\[\tau(1_\A)=1,\;\;\tau(ab)=\tau(ba).\]
\end{definition*}

In this paper, $\A$ will always be a von Neumann algebra having norm
$\|\cdot\|$, and $\tau$ a positive, faithful, and normal trace.
(These definitions are reviewed in
Appendix \ref{appendix:freeprobdetails}.) In particular,
$\tau$ will be norm-continuous with $|\tau(a)| \leq \|a\|$.

Following \cite{benaychgeorges}, we embed rectangular matrices
into a larger square space according to the following structure.
\begin{definition*}\label{def:rectangularspace}
Let $(\A,\tau)$ be a non-commutative probability space and $d \geq 1$ a positive
integer. For $p_1,\ldots,p_d \in \A$, 
$(\A,\tau,p_1,\ldots,p_d)$ is a {\bf rectangular probability space} if
$p_1,\ldots,p_d$ are non-zero pairwise-orthogonal projections
summing to 1, i.e.\ for all $r \neq s \in \{1,\ldots,d\}$,
\[p_r \neq 0,\;\;p_r=p_r^*=p_r^2,\;\;p_rp_s=0,\;\;
p_1+\ldots+p_d=1.\]
An element $a \in \A$ is {\bf simple} if $p_rap_s=a$ for some $r,s \in
\{1,\ldots,d\}$ (possibly $r=s$).
\end{definition*}

\begin{example}\label{ex:matrixspace}
Let $N_1,\ldots,N_d \geq 1$ be positive integers and denote $N=N_1+\ldots+N_d$.
Consider the $*$-algebra $\A=\C^{N \times N}$, with the involution $*$ 
given by the conjugate transpose map $A \mapsto A^*$.
For $A \in \C^{N \times N}$,
let $\tau(A)=N^{-1}\Tr A$. Then $(\A,\tau)=(\C^{N \times N},N^{-1} \Tr)$ is a
non-commutative probability space. Any $A \in \C^{N \times N}$ may be written
in block form as
\[A=\begin{pmatrix}
A_{11} & A_{12} & \cdots & A_{1d} \\
A_{21} & A_{22} & \cdots & A_{2d} \\
\vdots & \vdots & \ddots & \vdots \\
A_{d1} & A_{d2} & \cdots & A_{dd}
\end{pmatrix},\]
where $A_{st} \in \C^{N_s \times N_t}$. For each $r=1,\ldots,d$, denote
by $P_r$ the matrix with $(r,r)$ block equal to $\Id_{N_r}$ and
$(s,t)$ block equal to 0 for all other $s,t$. Then $P_r$ is a projection, and
$(\C^{N \times N},N^{-1}\Tr,P_1,\ldots,P_d)$
is a rectangular probability space. $A \in \C^{N \times N}$
is simple if $A_{st} \neq 0$ for at most one block $(s,t)$.
\end{example}

In a rectangular probability space, the projections $p_1,\ldots,p_d$ generate
a sub-$*$-algebra
\begin{equation}\label{eq:D}
\D:=\langle p_1,\ldots,p_d \rangle=\left\{\sum_{r=1}^d z_rp_r:z_r \in
\C\right\}.
\end{equation}
We may define a $*$-linear map $\F^\D:\A \to \D$ by
\begin{equation}\label{eq:FD}
\F^\D(a)=\sum_{r=1}^d p_r\tau_r(a),\qquad
\tau_r(a)=\tau(p_rap_r)/\tau(p_r),
\end{equation}
which is a projection onto $\D$ in the sense $\F^\D(d)=d$ for all $d \in \D$.
In Example \ref{ex:matrixspace}, $\D$ consists of matrices $A \in \C^{N \times
N}$ for which $A_{rr}$ is a multiple of the identity for each $r$ and
$A_{rs}=0$ for each $r \neq s$. In this example,
$\tau_r(A)=N_r^{-1}\Tr_r A$ where $\Tr_r A=\Tr A_{rr}$, so $\F^\D$ encodes the
trace of each diagonal block.

The tuple
$(\A,\D,\F^\D)$ is an example of the following definition for an
operator-valued probability space.
\begin{definition*}\label{def:Bspace}
A {\bf $\B$-valued probability space} $(\A,\B,\F^\B)$ is a $*$-algebra $\A$, a
sub-$*$-algebra $\B \subseteq \A$ containing $1_\A$, and a $*$-linear
map $\F^\B:\A \to \B$ called the {\bf conditional expectation} satisfying, for
all $b,b' \in \B$ and $a \in \A$,
\[\F^\B(bab')=b\F^\B(a)b',\;\;\F^\B(b)=b.\]
\end{definition*}

We identify $\C \subset \A$
as a sub-algebra via the inclusion map $z \mapsto z1_\A$, and we write 
1 for $1_\A$ and $z$ for $z1_\A$. Then a non-commutative probability space
$(\A,\tau)$ is also a $\C$-valued probability space with $\B=\C$ and
$\F^\B=\tau$. 
\begin{definition*}\label{def:invariant}
Let $(\A,\tau)$ be a non-commutative probability space and $\F^\B:\A \to \B$ a
conditional expectation onto a sub-algebra $\B \subset \A$.
$\F^\B$ is {\bf $\pmb{\tau}$-invariant} if $\tau \circ \F^B=\tau$.
\end{definition*}

It is verified that $\F^\D:\A \to \D$ defined by
(\ref{eq:FD}) is $\tau$-invariant. If $\B$ is a von Neumann
sub-algebra of (a von Neumann algebra) $\A$ and $\tau$ is a positive, faithful,
and normal trace, then there exists a unique $\tau$-invariant conditional
expectation $\F^\B:\A \to \B$, which is norm-continuous and satisfies
$\|\F^\B(a)\| \leq \|a\|$ (see \cite[Theorem 7 and Proposition 1]{kadison}).
If $\D \subseteq \B \subseteq \A$ are nested von Neumann
sub-algebras with $\tau$-invariant conditional expectations $\F^\D:\A \to \D$,
$\F^\B:\A \to \B$, then we have the analogue of the classical tower property,
\begin{equation}\label{eq:tower}
\F^\D=\F^\D \circ \F^\B.
\end{equation}
We note that $\D$ in (\ref{eq:D}) is a von Neumann sub-algebra of
$\A$, as it is finite-dimensional. 

In the space $(\A,\tau)$, $a \in \A$ may be thought
of as an analogue of a bounded random variable, $\tau(a)$ its expectation, and
$\F^\B(a)$ its conditional expectation with respect to a sub-sigma-field. The 
following definitions then provide an analogue of the conditional distribution
of $a$, and more generally of the conditional joint distribution of a collection
$(a_i)_{i \in \I}$.

\begin{definition*}
Let $\B$ be a $*$-algebra and $\I$ be any set. A {\bf $*$-monomial} in the
variables $\{x_i:i \in \I\}$ with coefficients in $\B$ is an expression of the
form $b_1y_1b_2y_2\ldots b_{l-1}y_{l-1}b_l$
where $l \geq 1$, $b_1,\ldots,b_l \in \B$, and $y_1,\ldots,y_{l-1} \in
\{x_i,x_i^*:i \in \I\}$. A {\bf $*$-polynomial} in $\{x_i:i \in \I\}$ with
coefficients in $\B$ is any finite sum of such monomials.
\end{definition*}

We write $Q(a_i:i \in \I)$ as the evaluation of a $*$-polynomial $Q$ at
$x_i=a_i$.

\begin{definition*}\label{def:Blaw}
Let $(\A,\B,\F^\B)$ be a $\B$-valued probability space,
let $(a_i)_{i \in \I}$ be elements of $\A$, and let $\mathcal{Q}$ denote the
set of all $*$-polynomials in variables $\{x_i:i \in \I\}$ with coefficients
in $\B$. The (joint) {\bf $\B$-law} of
$(a_i)_{i \in \I}$ is the collection of values in $\B$
\begin{equation}
\big\{\F^\B\big[Q(a_i:i \in I)\big]\big\}_{Q \in \mathcal{Q}}.
\label{eq:Bmoments}
\end{equation}
\end{definition*}

In the scalar setting where $\B=\C$ and $\F^\B=\tau$, a $*$-monomial takes the
simpler form $zy_1y_2\ldots y_{l-1}$ for $z \in \C$ and $y_1,\ldots,y_{l-1} \in
\{x_i,x_i^*: i \in \I\}$ (because $\C$ commutes with $\A$). Then the collection
of values (\ref{eq:Bmoments}) is determined by the scalar-valued moments
$\tau(w)$ for all words $w$ in the letters $\{x_i,x_i^*:i \in \I\}$. This is
the analogue of the unconditional joint distribution of a family of bounded
random variables, as specified by the joint moments.

Finally, the following definition of operator-valued freeness, introduced
in {\cite{voiculescu}, has similarities to the notion of conditional
independence of sub-sigma-fields in the classical setting.
\begin{definition*}\label{def:free}
Let $(\A,\B,\F^\B)$ be a $\B$-valued probability space and $(\A_i)_{i \in \I}$
a collection of sub-$*$-algebras of $\A$ which contain $\B$. $(\A_i)_{i \in \I}$
are {\bf $\pmb{\B}$-free}, or free with amalgamation over $\B$, if for all
$m \geq 1$, for all
$i_1,\ldots,i_m \in \I$ with $i_1 \neq i_2$, $i_2 \neq i_3$, $\ldots$, $i_{m-1} 
\neq i_m$, and for all $a_1 \in \A_{i_1},\ldots,a_m \in
\A_{i_m}$, the following implication holds:
\[\F^\B(a_1)=\F^\B(a_2)=\ldots=\F^\B(a_m)=0 \Rightarrow
\F^\B(a_1a_2\ldots a_m)=0.\]
Subsets $(S_i)_{i \in \I}$ of $\A$ are $\B$-free
if the sub-$*$-algebras $(\langle S_i,\B\rangle)_{i \in \I}$ are.
\end{definition*}

In the classical setting, the joint law of (conditionally) independent random
variables is determined by their marginal (conditional) laws. A similar
statement holds for freeness:
\begin{proposition}\label{prop:jointlawdetermined}
Suppose $(\A,\B,\F^\B)$ is a $\B$-valued probability space, and subsets
$(S_i)_{i \in \I}$ of $\A$ are $\B$-free. Then the $\B$-law of
$\bigcup_{i \in \I} S_i$ is determined by the individual
$\B$-laws of the $S_i$'s.
\end{proposition}
\begin{proof}
See \cite[Proposition 1.3]{voiculescu}.
\end{proof}

\subsection{Free deterministic equivalents and asymptotic freeness}
\label{sec:momentapprox}
Free deterministic equivalents were introduced in
\cite{speichervargas}. Here, we formalize a bit this definition for independent
jointly orthogonally-invariant families of matrices, and we establish closeness
of the random matrices and the free approximation in a general setting.

\begin{definition}\label{def:equalDlaw}
For fixed $d \geq 1$, consider two sequences of $N$-dependent
rectangular probability spaces $(\A,\tau,p_1,\ldots,p_d)$ and
$(\A',\tau',p_1',\ldots,p_d')$
such that for each $r \in \{1,\ldots,d\}$, as $N \to \infty$,
\[|\tau(p_r)-\tau'(p_r')| \to 0.\]
For a common index set $\I$, consider elements
$(a_i)_{i \in \I}$ of $\A$ and $(a_i')_{i \in
\I}$ of $\A'$. Then $(a_i)_{i \in \I}$ and $(a_i')_{i \in \I}$ are
{\bf asymptotically equal in $\pmb{\D}$-law} if the following holds:
For any $r \in \{1,\ldots,d\}$ and any $*$-polynomial $Q$ in the
variables $\{x_i:i \in \I\}$
with coefficients in $\D=\langle p_1,\ldots,p_d \rangle$, denoting by
$Q'$ the corresponding $*$-polynomial with coefficients
in $\D'=\langle p_1,\ldots,p_d\rangle$, as $N \to \infty$,
\begin{equation}\label{eq:freeequivalent}
\big|\tau_r\big[Q(a_i:i \in \I)\big]-\tau_r'
\big[Q'(a_i':i \in \I)\big]\big| \to 0.
\end{equation}
If $(a_i)_{i \in \I}$ and/or $(a_i')_{i \in \I}$ are random elements of
$\A$ and/or $\A'$, then they are {\bf asymptotically equal in $\pmb{\D}$-law
a.s.} if the above holds almost surely for each individual $*$-polynomial $Q$.
\end{definition}

In the above, $\tau_r$ and $\tau_r'$ are defined by (\ref{eq:FD}).
``Corresponding'' means that $Q'$ is obtained by expressing each
coefficient $d \in \D$ of 
$Q$ in the form (\ref{eq:D}) and replacing $p_1,\ldots,p_d$ by
$p_1',\ldots,p_d'$. 

We will apply Definition \ref{def:equalDlaw} by taking one of the two
rectangular spaces to be $(\C^{N \times N},N^{-1}\Tr)$ as in Example
\ref{ex:matrixspace}, containing random elements, and the other to be an
approximating deterministic model. (We will use ``distribution'' for
random matrices to mean their distribution as random elements of $\C^{N
\times N}$ in the usual sense, reserving the term ``$\B$-law''
for Definition \ref{def:Blaw}.) Freeness relations in
the deterministic model will emerge from the following notion of rotational
invariance of the random matrices.

\begin{definition}\label{def:orthogonalinvariance}
Consider $(\C^{N \times N},N^{-1}\Tr,P_1,\ldots,P_d)$ as in Example
\ref{ex:matrixspace}. A family of random matrices
$(H_i)_{i \in \I}$ in $\C^{N \times N}$ is {\bf block-orthogonally invariant}
if, for any orthogonal matrices
$O_r \in \R^{N_r \times N_r}$ for $r=1,\ldots,d$, denoting
$O=\diag(O_1,\ldots,O_d) \in \R^{N \times N}$, the joint
distribution of $(H_i)_{i \in \I}$ is equal to that of $(O^TH_iO)_{i \in \I}$.
\end{definition}

Let us provide several examples. We discuss the constructions of the spaces
$(\A,\tau,p_1,\ldots,p_d)$ for these examples in Appendix
\ref{appendix:freeprobdetails}.

\begin{example}\label{ex:semicircle}
Fix $r \in \{1,\ldots,d\}$ and let $G \in \C^{N \times N}$ be a simple random
matrix such that the diagonal block $G_{rr} \in \C^{N_r \times N_r}$ is
distributed as the GUE or GOE, scaled to have entries of variance $1/N_r$.
(Simple means $G_{st}=0$ for all other blocks $(s,t)$.)
Let $(\A,\tau,p_1,\ldots,p_d)$ be a rectangular space with
$\tau(p_s)=N_s/N$ for each $s=1,\ldots,d$, such that $\A$ contains
a self-adjoint simple element $g$ satisfying $g=g^*$ and $p_rgp_r=g$, with
moments given by the semi-circle law:
\[\tau_r(g^l)=\int_{-2}^2
\frac{x^l}{2\pi}\sqrt{4-x^2}\,dx\qquad \text{for all } l \geq 0.\]
For any corresponding $*$-polynomials $Q$ and $q$ as in
Definition \ref{def:equalDlaw}, we may verify
$N_r^{-1}\Tr_r Q(G)-\tau_r(q(g)) \to 0$ a.s.\ by the classical Wigner
semi-circle theorem \cite{wigner}.
Then $G$ and $g$ are asymptotically equal in $\D$-law a.s.
Furthermore, $G$ is block-orthogonally invariant.
\end{example}

\begin{example}\label{ex:MP}
Fix $r_1 \neq r_2 \in \{1,\ldots,d\}$ and let $G \in \C^{N \times N}$ be a
simple random matrix such that the block $G_{r_1r_2}$
has i.i.d.\ Gaussian or complex Gaussian entries with variance $1/N_{r_1}$. Let
$(\A,\tau,p_1,\ldots,p_d)$ satisfy
$\tau(p_s)=N_s/N$ for each $s$, such that $\A$ contains
a simple element $g$ satisfying $p_{r_1}gp_{r_2}=g$, where
$g^*g$ has moments given by the Marcenko-Pastur law:
\[\tau_{r_2}((g^*g)^l)=\int x^l \nu_{N_{r_2}/N_{r_1}}(x)dx
\qquad \text{for all }l \geq 0\]
where $\nu_\lambda$ is the standard Marcenko-Pastur density
\begin{equation}\label{eq:standardMP}
\nu_\lambda(x)=\frac{1}{2\pi}\frac{\sqrt{(\lambda_+-x)(x-\lambda_-)}}{
\lambda x}\1_{[\lambda_-,\lambda_+]}(x),\qquad
\lambda_\pm=(1\pm \sqrt{\lambda})^2.
\end{equation}
By definition of $\tau_r$ and the cyclic property of $\tau$, we also have
\[\tau_{r_1}((gg^*)^l)=(N_{r_2}/N_{r_1})\tau_{r_2}((g^*g)^l).\]
For any corresponding $*$-polynomials $Q$ and $q$ as in
Definition \ref{def:equalDlaw}, we may verify
$N_{r_2}^{-1}\Tr_{r_2} Q(G)-\tau_{r_2}(q(g)) \to 0$ and
$N_{r_1}^{-1} \Tr_{r_1} Q(G)-\tau_{r_1}(q(g)) \to 0$ a.s.\ by the classical
Marcenko-Pastur theorem \cite{marcenkopastur}. Then $G$ and $g$ are
asymptotically equal in $\D$-law a.s., and $G$ is block-orthogonally
invariant.
\end{example}

\begin{example}\label{ex:deterministic}
Let $B_1,\ldots,B_k \in \C^{N \times N}$ be deterministic simple matrices,
say with $P_{r_i}B_iP_{s_i}=B_i$ for each $i=1,\ldots,k$ and $r_i,s_i
\in \{1,\ldots,d\}$. Let
$O_1 \in \R^{N_1 \times N_1},\ldots,O_d \in \R^{N_d \times N_d}$ be
independent Haar-distributed orthogonal matrices, define 
$O=\diag(O_1,\ldots,O_d) \in \R^{N \times N}$, and let $\check{B}_i=O^TB_iO$.
Let $(\A,\tau,p_1,\ldots,p_d)$ satisfy $\tau(p_s)=N_s/N$ for each $s$,
such that $\A$ contains simple elements $b_1,\ldots,b_k$ satisfying
$p_{r_i}b_ip_{s_i}=b_i$ for each $i=1,\ldots,k$, and
\begin{equation}\label{eq:exdeterministic}
N_r^{-1} \Tr_r Q(B_1,\ldots,B_k)=\tau_r(q(b_1,\ldots,b_k))
\end{equation}
for any corresponding $*$-polynomials $Q$ and $q$ with coefficients in
$\langle P_1,\ldots,P_d \rangle$ and $\langle p_1,\ldots,p_d \rangle$.
As $\Tr_r Q(B_1,\ldots,B_k)$ is invariant under $B_i \mapsto O^TB_iO$,
(\ref{eq:exdeterministic}) holds also
with $\check{B}_i$ in place of $B_i$. Then
$(\check{B}_i)_{i \in \{1,\ldots,k\}}$ and $(b_i)_{i \in \{1,\ldots,k\}}$ are
exactly (and hence also asymptotically) equal in $\D$-law,
and $(\check{B}_i)_{i \in \{1,\ldots,k\}}$ is block-orthogonally
invariant by construction.
\end{example}

To study the interaction of several independent and block-orthogonally
invariant matrix families,
we will take a deterministic model for each family, as in Examples
\ref{ex:semicircle}, \ref{ex:MP}, and \ref{ex:deterministic} above, and consider
a combined model in which these families are $\D$-free:

\begin{definition}\label{def:freeequivalent}
Consider $(\C^{N \times N},N^{-1}\Tr,P_1,\ldots,P_d)$ as in
Example \ref{ex:matrixspace}. Suppose $(H_i)_{i
\in \I_1}$, $\ldots$, $(H_i)_{i \in \I_J}$ are finite families of
random matrices in $\C^{N \times N}$ such that:
\begin{itemize}
\item These families are independent from each other, and
\item For each $j=1,\ldots,J$,
$(H_i)_{i \in \I_j}$ is block-orthogonally invariant.
\end{itemize}
Then a {\bf free deterministic equivalent} for $(H_i)_{i \in \I_1},
\ldots,(H_i)_{i \in \I_J}$ is any ($N$-dependent)
rectangular probability space $(\A,\tau,p_1,\ldots,p_d)$ and families
$(h_i)_{i \in \I_1},\ldots,(h_i)_{i \in \I_J}$ of deterministic elements in
$\A$ such that, as $N \to \infty$:
\begin{itemize}
\item For each $r=1,\ldots,d$, $|N^{-1} \Tr P_r-\tau(p_r)| \to 0$,
\item For each $j=1,\ldots,J$,
$(H_i)_{i \in \I_j}$ and $(h_i)_{i \in \I_j}$ are asymptotically equal in
$\D$-law a.s., and
\item $(h_i)_{i \in \I_1},\ldots,(h_i)_{i \in \I_J}$
are free with amalgamation over $\D=\langle p_1,\ldots,p_d\rangle$.
\end{itemize}
\end{definition}

The main result of this section is the following asymptotic freeness theorem,
which establishes the validity of this approximation.

\begin{theorem}\label{thm:momentapprox}
In the space $(\C^{N \times N},N^{-1}\Tr,P_1,\ldots,P_d)$ of Example
\ref{ex:matrixspace}, suppose
$(H_i)_{i \in \I_1}$, $\ldots$, $(H_i)_{i \in \I_J}$ are independent,
block-orthogonally invariant families of random matrices, and
let $(h_i)_{i \in \I_1}$, $\ldots$, $(h_i)_{i \in \I_J}$
be any free deterministic
equivalent in $(\A,\tau,p_1,\ldots,p_d)$. If there exist constants
$C,c>0$ (independent of $N$) such that
$c<N_r/N$ for all $r$ and $\|H_i\|<C$ a.s.\ for all $i \in \I_j$, all $\I_j$,
and all large $N$, then $(H_i)_{i \in \I_j,j \in \{1,\ldots,J\}}$
and $(h_i)_{i \in \I_j,j \in \{1,\ldots,J\}}$ are asymptotically
equal in $\D$-law a.s.
\end{theorem}
More informally, if $(h_i)_{i \in \I_j}$ asymptotically
models the family $(H_i)_{i \in \I_j}$ for each $j$, and
these matrix families are independent and block-orthogonally invariant,
then a system in which $(h_i)_{i \in \I_j}$ are $\D$-free asymptotically models 
the matrices jointly over $j$.

Theorem \ref{thm:momentapprox} is analogous to
\cite[Theorem 1.6]{benaychgeorges} and \cite[Theorem 2.7]{speichervargas},
which establish similar results for complex unitary invariance.
It permits multiple matrix families (where matrices
within each family are not independent),
uses the almost-sure trace $N^{-1} \Tr$ rather than $\E \circ N^{-1}\Tr$,
and imposes boundedness rather than joint convergence assumptions.
This last point fully embraces the deterministic equivalents approach.

We will apply Theorem \ref{thm:momentapprox} in the form of the following
corollary: Suppose that
$w \in \A$ satisfies $|\tau(w^l)| \leq C^l$ for a constant $C>0$ and all $l \geq
1$. We may define its Stieltjes transform by the convergent series
\begin{equation}\label{eq:stieltjesw}
m_w(z)=\tau\big((w-z)^{-1}\big)=-\sum_{l \geq 0}^\infty z^{-(l+1)}\tau(w^l)
\end{equation}
for $z \in \C^+$ with $|z|>C$, where we use the convention $w^0=1$ for all
$w \in \A$.
\begin{corollary}\label{cor:stieltjesapprox}
Under the assumptions of Theorem \ref{thm:momentapprox}, let
$Q$ be a self-adjoint $*$-polynomial (with $\C$-valued coefficients) in
$(x_i)_{i \in \I_j,j \in \{1,\ldots,J\}}$, and let
\begin{align*}
W&=Q(H_i:i \in \I_j,j \in \{1,\ldots,J\}) \in \C^{N \times N},\\
w&=Q(h_i:i \in \I_j,j \in \{1,\ldots,J\}) \in \A.
\end{align*}
Suppose $|\tau(w^l)| \leq C^l$ for all $N,l \geq 1$ and some $C>0$.
Then for a sufficiently large constant $C_0>0$,
letting $\mathbb{D}=\{z \in \C^+:|z|>C_0\}$ and defining
$m_W(z)=N^{-1} \Tr (W-z\Id_N)^{-1}$ and $m_w(z)=\tau((w-z)^{-1})$,
\[m_W(z)-m_w(z) \to 0\]
pointwise almost surely over $z \in \mathbb{D}$.
\end{corollary}
Proofs of Theorem \ref{thm:momentapprox} and Corollary
\ref{cor:stieltjesapprox} are contained in Appendix \ref{appendix:freeapprox}.

\subsection{Computational tools}
Our computations in the free model will use the tools of free
cumulants, $\cR$-transforms, and Cauchy transforms discussed in
\cite{speicherbook,nicaetal,speichervargas}. We review some relevant concepts
here.

Let $(\A,\B,\F^\B)$ be a $\B$-valued probability space and $\F^\B:\A \to \B$
a conditional expectation.
For $l \geq 1$, the $l^\text{th}$ order {\bf free cumulant} of $\F^\B$ is a map
$\kappa_l^\B:\A^l \to \B$ defined by
$\F^\B$ and certain moment-cumulant relations over the
non-crossing partition lattice; we refer the reader to \cite{speichervargas} and
\cite[Chapters 2 and 3]{speicherbook} for details. We will use the
properties that $\kappa_l^\B$ is linear in each argument and 
satisfies the relations
\begin{align}
\kappa_l^\B(ba_1,a_2,\ldots,a_{l-1},a_lb')&=b\kappa_l^\B(a_1,\ldots,a_l)b',
\label{eq:kappaid1}\\
\kappa_l^\B(a_1,\ldots,a_{j-1},a_jb,a_{j+1},\ldots,a_l)
&=\kappa_l^\B(a_1,\ldots,a_j,ba_{j+1},\ldots,a_l)\label{eq:kappaid2}
\end{align}
for any $b,b' \in \B$ and $a_1,\ldots,a_l \in \A$.

For $a \in \A$, the {\bf $\pmb{\B}$-valued $\pmb{\cR}$-transform} of $a$ is
defined, for $b \in \B$, as
\begin{equation}\label{eq:Rtransform}
\cR_a^\B(b):=\sum_{l \geq 1} \kappa_l^{\B}(ab,\ldots,ab,a).
\end{equation}
The {\bf $\pmb{\B}$-valued Cauchy transform} of $a$ is defined, for invertible
$b \in \B$, as
\begin{equation}\label{eq:Cauchytransform}
G_a^\B(b):=\F^\B((b-a)^{-1})=\sum_{l \geq 0} \F^\B(b^{-1}(ab^{-1})^l),
\end{equation}
with the convention $a^0=1$ for all $a \in \A$. The moment-cumulant
relations imply that
$G_a^\B(b)$ and $\cR_a^\B(b)+b^{-1}$ are inverses with respect to composition:

\begin{proposition}\label{prop:GRrelation}
Let $(\A,\B,\F^\B)$ be a $\B$-valued probability space.
For $a \in \A$ and invertible $b \in \B$,
\begin{align}
G_a^{\B}(b^{-1}+\cR_a^{\B}(b))&=b,\label{eq:GRrelation}\\
G_a^{\B}(b)&=\left(b-\cR_a^{\B}(G_a^{\B}(b))\right)^{-1}.
\label{eq:GRrelation2}
\end{align}
\end{proposition}
\begin{proof}
See \cite[Theorem 4.9]{voiculescu} and also \cite[Theorem
4.1.12]{speicherbook}.
\end{proof}

\begin{remark*}
When $\A$ is a von Neumann algebra, the right sides of (\ref{eq:Rtransform})
and (\ref{eq:Cauchytransform}) may be understood as
convergent series in $\A$ with respect to the norm $\|\cdot\|$, for sufficiently
small $\|b\|$ and $\|b^{-1}\|$ respectively. Indeed, (\ref{eq:Cauchytransform})
defines a convergent series in $\B$ when $\|b^{-1}\|<1/\|a\|$, with
\begin{equation}\label{eq:Cauchybound}
\|G^\B_a(b)\| \leq \sum_{l \geq 0} \|b^{-1}\|^{l+1}\|a\|^l
=\frac{\|b^{-1}\|}{1-\|a\|\|b^{-1}\|}.
\end{equation}
Also, explicit inversion of the moment-cumulant relations for the non-crossing
partition lattice yields the cumulant bound
\begin{equation}\label{eq:cumulantbound}
\kappa_l^\B(a_1,\ldots,a_l) \leq 16^l\prod_{i=1}^l \|a_i\|
\end{equation}
(see \cite[Proposition 13.15]{nicaspeicherlectures}),
so (\ref{eq:Rtransform}) defines a convergent series in $\B$
when $16\|b\|<1/\|a\|$, with
\[\|\cR^\B_a(b)\| \leq \sum_{l \geq 1} 16^l \|a\|^l\|b\|^{l-1}
=\frac{16\|a\|}{1-16\|a\|\|b\|}.\]
The identities (\ref{eq:GRrelation}) and (\ref{eq:GRrelation2}) hold
as equalities of elements in $\B$ when $\|b\|$ and $\|b^{-1}\|$ are
sufficiently small, respectively.
\end{remark*}

Our computation will pass between $\cR$-transforms
and Cauchy transforms with respect to nested sub-algebras of $\A$.
Central to this approach is the following result from \cite{nicaetal} (see also
\cite{speichervargas}):
\begin{proposition}\label{prop:freeness}
Let $(\A,\D,\F^\D)$ be a $\D$-valued probability space, let $\B,\H
\subseteq \A$ be sub-$*$-algebras containing $\D$,
and let $\F^\B:\A \to \B$ be a conditional expectation such
that $\F^\D \circ \F^\B=\F^\D$. Let $\kappa_l^\B$ and $\kappa_l^\D$
denote the free cumulants for $\F^\B$ and $\F^\D$. If $\B$ and $\H$ are
$\D$-free, then for all $l \geq 1$, $h_1,\ldots,h_l \in \H$,
and $b_1,\ldots,b_{l-1} \in \B$,
\[\kappa_l^\B(h_1b_1,\ldots,h_{l-1}b_{l-1},h_l)
=\kappa_l^\D(h_1\F^\D(b_1),\ldots,h_{l-1}\F^\D(b_{l-1}),h_l).\]
\end{proposition}
\begin{proof}
See \cite[Theorem 3.6]{nicaetal}.
\end{proof}

For sub-algebras $\D \subseteq \B \subseteq \A$ and conditional expectations
$\F^\D:\A \to \D$ and $\F^\B:\A \to \B$ satisfying (\ref{eq:tower}), we also
have for any $a \in \A$ and invertible $d \in \D$ (with sufficiently small
$\|d^{-1}\|$), by (\ref{eq:Cauchytransform}),
\begin{equation}\label{eq:Gprojection}
G_a^\D(d)=\F^\D \circ G_a^\B(d).
\end{equation}
Finally, note that for $\B=\C$ and $\F^\B=\tau$, the scalar-valued Cauchy
transform $G_a^\C(z)$ is simply $-m_a(z)$ from (\ref{eq:stieltjesw}).
(The minus sign is a difference in sign convention for the Cauchy/Stieltjes
transform.)

%% file: Section_4_fix.tex
We will prove analogues of Theorems
\ref{thm:bulkdistribution} and \ref{thm:algorithm} for a slightly more general
matrix model: Fix $k \geq 1$, let $p,n_1,\ldots,n_k,m_1,\ldots,m_k \in \N$, and
denote $n_+=\sum_{r=1}^k n_r$. Let $F \in \C^{n_+ \times n_+}$ be deterministic
with $F^*=F$, and denote by $F_{rs} \in \C^{n_r \times n_s}$ its $(r,s)$
submatrix. For $r=1,\ldots,k$, let $H_r \in \C^{m_r \times p}$ be
deterministic, and let $G_r$ be independent random matrices such that either
$G_r \in \R^{n_r \times m_r}$ with
$(G_r)_{ij} \overset{iid}{\sim} \Nor(0,n_r^{-1})$ or
$G_r \in \C^{n_r \times m_r}$ with $\Im (G_r)_{ij},\Re (G_r)_{ij}
\overset{iid}{\sim} \Nor(0,(2n_r)^{-1})$. Define
\[W:=\sum_{r,s=1}^k H_r^*G_r^*F_{rs}G_sH_s \in \C^{p \times p},\]
with empirical spectral measure $\mu_W$. Denote
$b \cdot H^*H=\sum_{s=1}^k b_s H_s^*H_s$,
and let $D(a)$ and $\Tr_r$ be as in Theorem \ref{thm:bulkdistribution}.

\begin{theorem}\label{thm:Wdistribution}
Suppose $p,n_1,\ldots,n_k,m_1,\ldots,m_k \to \infty$, such that $c<n_r/p<C$,
$c<m_r/p<C$, $\|H_r\|<C$, and $\|F_{rs}\|<C$ for all
$r,s=1,\ldots,k$ and some constants $C,c>0$. Then:
\begin{enumerate}[(a)]
\item For each $z \in \C^+$, there exist unique values
$a_1,\ldots,a_k \in \C^+ \cup \{0\}$ and
$b_1,\ldots,b_k \in \overline{\C^+}$
that satisfy, for $r=1,\ldots,k$, the equations
\begin{align}
a_r&=-\tfrac{1}{n_r}\Tr\left((z\Id_p+b \cdot H^*H)^{-1}H_r^*H_r\right),
\label{eq:Warecursion}\\
b_r&=-\tfrac{1}{n_r}\Tr_r\left([\Id_{n_+}+FD(a)]^{-1}F\right).
\label{eq:Wbrecursion}
\end{align}
\item $\mu_W-\mu_0 \to 0$ weakly a.s.\ for a probability measure $\mu_0$
on $\R$ with Stieltjes transform
\begin{equation}\label{eq:Wm0}
m_0(z):=-\tfrac{1}{p}\Tr\left((z\Id_p+b \cdot H^*H)^{-1}\right).
\end{equation}
\item For each $z \in \C^+$, the values $a_r,b_r$
in (a) are the limits, as $t \to \infty$, of $a_r^{(t)},b_r^{(t)}$
computed by iterating (\ref{eq:Warecursion}--\ref{eq:Wbrecursion})
in the manner of Theorem \ref{thm:algorithm}.
\end{enumerate}
\end{theorem}
\noindent Theorems \ref{thm:bulkdistribution} and \ref{thm:algorithm}
follow by specializing this result to $F=U^TBU$ and $m_r=p$, $n_r=I_r$, and
$H_r=\Sigma_r^{1/2}$ for each $r=1,\ldots,k$.

In this section, we carry out the bulk of the proof of Theorem
\ref{thm:Wdistribution} by
\begin{enumerate}[1.]
\item Defining a free deterministic equivalent for this matrix model, and
\item Showing that the Stieltjes transform of the element $w$
(modeling $W$) satisfies (\ref{eq:Warecursion}--\ref{eq:Wm0}).
\end{enumerate}
These steps correspond to the separation of approximation and computation
discussed in Section \ref{subsec:proofsketch}.

For the reader's convenience, in Appendix \ref{appendix:march-past-case}, we
provide a simplified version of these steps for the special case of
Theorem \ref{thm:Wdistribution} corresponding to Theorem \ref{thm:mp}
for sample covariance matrices, which illustrates the main ideas.

\subsection{Defining a free deterministic equivalent}

Consider the transformations
\[H_r \mapsto O_r^TH_rO_0,\qquad F_{rs} \mapsto O_{k+r}^TF_{rs}O_{k+s}\]
for independent Haar-distributed orthogonal matrices $O_0,\ldots,O_{2k}$
of the appropriate sizes. As in Section \ref{subsec:proofsketch}, 
$\mu_W$ remains invariant in law under these transformations.
Hence it suffices to prove
Theorem \ref{thm:Wdistribution} with $H_r$ and $F_{rs}$ replaced by these
randomly-rotated matrices, which (with a slight abuse of notation) we
continue to denote by $H_r$ and $F_{rs}$.

Let $N=p+\sum_{r=1}^k m_r+\sum_{r=1}^k n_r$, and embed the matrices
$W,H_r,G_r,F_{rs}$ as simple elements of $\C^{N \times N}$ in the following
regions of the block-matrix decomposition corresponding to
$\C^N=\C^p \oplus \C^{m_1} \oplus \cdots \oplus \C^{m_k}
\oplus \C^{n_1} \oplus \cdots \oplus \C^{n_k}$:
\[\begin{array}{|c|c|c|c|c|c|c|}
\hline
W & H_1^* & \cdots & H_k^* & & & \\
\hline
H_1 & & & & G_1^* & & \\
\hline
\vdots & & & & & \ddots & \\
\hline
H_k & & & & & & G_k^* \\
\hline
& G_1 & & & F_{1,1} & \cdots & F_{1,k} \\
\hline
 & & \ddots & & \vdots & \ddots & \vdots \\
\hline
& & & G_k & F_{k,1} & \cdots & F_{k,k} \\
\hline
\end{array}\]
Denote by $P_0,\ldots,P_{2k}$ the diagonal projections corresponding to 
the above decomposition, and by
$\tilde{W},\tilde{F}_{r,s},\tilde{G}_r,\tilde{H}_r \in
\C^{N \times N}$ the embedded matrices. (I.e.\ we have
$P_0=\diag(\Id_p,0,\ldots,0)$,
$P_1=\diag(0,\Id_{m_1},\ldots,0)$, etc. $\tilde{W}$ has
upper-left block equal to $W$ and remaining blocks 0, etc.)
Then $\tilde{W},\tilde{F}_{r,s},\tilde{G}_r,\tilde{H}_r$ are
simple elements of the rectangular space
$(\C^{N \times N},N^{-1} \Tr,P_0,\ldots,P_{2k})$, and the $k+2$ families
$\{\tilde{F}_{r,s}\}$, $\{\tilde{H}_r\}$, $\tilde{G}_1$, $\ldots$, $\tilde{G}_k$
are independent of each other and are block-orthogonally invariant.

For the approximating free model, consider a second
($N$-dependent) rectangular probability space $(\A,\tau,p_0,\ldots,p_{2k})$
with deterministic elements $f_{rs},g_r,h_r \in \A$, such that the following
hold:
\begin{enumerate}[1.]
\item $p_0,\ldots,p_{2k}$ have traces
\[\tau(p_0)=p/N,\;\tau(p_r)=m_r/N,
\;\tau(p_{k+r})=n_r/N\qquad \text{for all } r=1,\ldots,k.\]
\item $f_{rs},g_r,h_r$ are simple elements such that for
all $r,s \in \{1,\ldots,k\}$,
\[p_{k+r}f_{rs}p_{k+s}=f_{rs},\qquad p_{k+r}g_rp_r=g_r,\qquad p_rh_rp_0=h_r.\]
\item $\{f_{rs}:1 \leq r,s \leq k\}$ has the same joint $\D$-law as
$\{\tilde{F}_{r,s}:1 \leq r,s \leq k\}$, and $\{h_r:1 \leq r \leq k\}$ has the
same joint $\D$-law as $\{\tilde{H}_r:1 \leq r \leq k\}$. I.e., for any $r \in
\{0,\ldots,2k\}$ and
any non-commutative $*$-polynomials $Q_1,Q_2$ with coefficients in
$\langle P_0,\ldots,P_{2k} \rangle$, letting $q_1,q_2$ denote the
corresponding $*$-polynomials with coefficients in $\langle p_0,\ldots,p_{2k}
\rangle$,
\begin{align}
\tau_r\left[q_1(f_{st}:s,t \in \{1,\ldots,k\})\right]
&=N_r^{-1}\Tr_r Q_1(\tilde{F}_{s,t}:s,t \in \{1,\ldots,k\}),\label{eq:fF}\\
\tau_r\left[q_2(h_s:s \in \{1,\ldots,k\})\right]
&=N_r^{-1}\Tr_r Q_2(\tilde{H}_s:s \in \{1,\ldots,k\}).\label{eq:hH}
\end{align}
\item For each $r$, $g_r^*g_r$ has
Marcenko-Pastur law with parameter $\lambda=m_r/n_r$. I.e.\ for $\nu_\lambda$
as in (\ref{eq:standardMP}),
\begin{equation}\label{eq:gG}
\tau_r((g_r^*g_r)^l)=\int x^l \nu_{m_r/n_r}(x)dx
\qquad \text{for all }l \geq 0.
\end{equation}
\item The $k+2$ families $\{f_{rs}\}$,
$\{h_r\}$, $g_1,\ldots,g_k$
are free with amalgamation over $\D=\langle p_0,\ldots,p_{2k}\rangle$.
\end{enumerate}

The right sides of (\ref{eq:fF}) and (\ref{eq:hH}) are deterministic,
as they are invariant to the random rotations of $F_{rs}$ and $H_r$.
Also, (\ref{eq:gG}) completely specifies $\tau(q(g_r))$ for any $*$-polynomial
$q$ with coefficients in $\D$.
Then these conditions 1--5 fully specify the joint $\D$-law of all elements
$f_{rs},g_r,h_r \in \A$. These elements are a free deterministic
equivalent for $\tilde{F}_{r,s},\tilde{G}_r,\tilde{H}_r \in \C^{N \times
N}$ in the sense of Definition \ref{def:freeequivalent}.

The following lemma establishes existence of this model as a von Neumann
algebra; its proof is deferred to Appendix \ref{appendix:freeprobdetails}.
\begin{lemma}\label{lemma:freeconstruction}
Under the conditions of Theorem \ref{thm:Wdistribution},
there exists a ($N$-dependent) rectangular probability space
$(\A,\tau,p_0,\ldots,p_{2k})$ such that:
\begin{enumerate}[(a)]
\item $\A$ is a von Neumann algebra and $\tau$ is a positive, faithful,
normal trace.
\item $\A$ contains elements $f_{rs},g_r,h_r$ for $r,s \in \{1,\ldots,k\}$ that
satisfy the above conditions. Furthermore, the von Neumann sub-algebras
$\langle \D,\{f_{rs}\} \rangle_{W^*}$, $\langle \D,\{h_r\}\rangle_{W^*}$,
$\langle \D,g_1 \rangle_{W^*}$, ..., $\langle \D,g_k \rangle_{W^*}$ are free
over $\D$.
\item There exists a constant $C>0$ such that
$\|f_{rs}\|,\|h_r\|,\|g_r\| \leq C$ for all $N$ and all $r,s$.
\end{enumerate}
\end{lemma}

\subsection{Computing the Stieltjes transform of $w$}

We will use twice the following intermediary lemma:
\begin{lemma}\label{lemma:freecompute}
Let $(\A,\tau,q_0,q_1,\ldots,q_k)$ be a rectangular probability space, where
$\A$ is von Neumann and $\tau$ is positive, faithful, and normal.
Let $\D=\langle q_0,\ldots,q_k \rangle$, let $\B,\cC \subset \A$ be von Neumann
sub-algebras containing $\D$ that are free over $\D$, and let
$\F^\D:\A \to \D$ and $\F^\cC:\A \to \cC$ be the $\tau$-invariant conditional
expectations.

Let $b_{rs} \in \mathcal{B}$ and $c_r \in \mathcal{C}$ for $1 \leq r,s \leq k$
be such that $q_rb_{rs}q_s=b_{rs}$, $q_rc_r=c_r$, $\|b_{rs}\| \leq C$, and
$\|c_r\| \leq C$ for some constant $C>0$. Define
$a=\sum_{r,s=1}^k c_r^*b_{rs}c_s$ and $b=\sum_{r,s=1}^k b_{rs}$. Then
for $e \in \cC$ with $\|e\|$ sufficiently small,
\[\cR^\cC_a(e)=\sum_{r=1}^k c_r^*c_r\tau_r\left(
\cR^\D_b\left(\sum_{s=1}^k \tau_s(c_sec_s^*)q_s\right)\right),\]
where $\cR_a^\cC$ and $\cR_b^\D$ are the $\cC$-valued and $\D$-valued
$\cR$-transforms of $a$ and $b$.
\end{lemma}
\begin{proof}
We use the computational idea of \cite{speichervargas}:
Denote by $\kappa_l^\cC$ and $\kappa_l^\D$ the $\cC$-valued and $\D$-valued
free cumulants. For $l \geq 1$ and $e \in \cC$,
\begin{align*}
&\kappa_l^\cC(ae,\ldots,ae,a)\\
&=\kappa_l^\cC\left(\sum_{r,s=1}^k c_r^*b_{rs}c_se,\;
\ldots,\;\sum_{r,s=1}^k c_r^*b_{rs}c_se,
\sum_{r,s=1}^d c_r^*b_{rs}c_s\right)\\
&=\sum_{r_1,s_1,\ldots,r_l,s_l=1}^k
\kappa_l^\cC(c_{r_1}^*b_{r_1s_1}c_{s_1}e,\;\ldots,\;
c_{r_{l-1}}^*b_{r_{l-1}s_{l-1}}c_{s_{l-1}}e,\;
c_{r_l}^*b_{r_ls_l}c_{s_l})\\
&=\sum_{r_1,s_1,\ldots,r_l,s_l=1}^k
c_{r_1}^*\kappa_l^\cC(b_{r_1s_1}c_{s_1}ec_{r_2}^*,\;\ldots,\;
b_{r_{l-1}s_{l-1}}c_{s_{l-1}}ec_{r_l}^*,\;
b_{r_ls_l})\,c_{s_l}\\
&=\sum_{r_1,s_1,\ldots,r_l,s_l=1}^k
c_{r_1}^*\kappa_l^\D(b_{r_1s_1}\F^\D(c_{s_1}ec_{r_2}^*),\;\ldots,\;
b_{r_{l-1}s_{l-1}}\F^\D(c_{s_{l-1}}ec_{r_l}^*),\;
b_{r_ls_l})\,c_{s_l},
\end{align*}
where we applied the definition of $a$, multi-linearity of $\kappa_l^\cC$,
the identities (\ref{eq:kappaid1}) and (\ref{eq:kappaid2}), and Proposition
\ref{prop:freeness} using freeness of $\B$ and $\cC$ over $\D$.

By the identity $c_r=q_rc_r$, each $c_sec_r^*$ is simple,
and we have from (\ref{eq:FD})
\[\F^\D(c_sec_r^*)=\begin{cases} 0 & \text{if }s \neq r \\
\tau_s(c_sec_s^*)q_s & \text{if }s=r. \end{cases}\]
Furthermore, for any $d \in \D$, as $d=\tau_0(d)q_0+\ldots
+\tau_k(d)q_k$, we have
$c_r^*dc_s=c_r^*c_r\tau_r(d)$ if $r=s$ and 0 otherwise.
Hence we may restrict the above sum
to $s_1=r_2,\;s_2=r_3,\;\ldots,\;s_{l-1}=r_l,\;s_l=r_1$.
Then, setting
\begin{equation}\label{eq:di}
d=\sum_{r=1}^k \tau_r(c_rec_r^*)q_r
\end{equation}
and applying the identity $q_rb_{rs}q_s=b_{rs}$,
\begin{equation}\label{eq:kappaHw}
\kappa_l^\cC(ae,\ldots,ae,a)
=\sum_{r_1,\ldots,r_l=1}^k c_{r_1}^*c_{r_1}\tau_{r_1}\left(
\kappa_l^\D(b_{r_1r_2}d,\;\ldots,\;b_{r_{l-1}r_l}d,\;
b_{r_lr_1})\right).
\end{equation}

On the other hand, similar arguments yield
\begin{align*}
&\kappa_l^\D(bd,\ldots,bd,b)\\
&=\sum_{r_1,s_1,\ldots,r_l,s_l=1}^k
\kappa_l^\D(b_{r_1s_1}d,\;\ldots,\;b_{r_{l-1}s_{l-1}}d,\;b_{r_ls_l})\\
&=\sum_{r_1,s_1,\ldots,r_l,s_l=1}^k
q_{r_1}\kappa_l^\D(b_{r_1s_1}q_{s_1}dq_{r_2},\;\ldots,\;
b_{r_{l-1}s_{l-1}}q_{s_{l-1}}dq_{r_l},\;b_{r_ls_l})\,q_{s_l}\\
&=\sum_{r_1,\ldots,r_l=1}^k q_{r_1}\kappa_l^\D(b_{r_1r_2}d,\;\ldots,\;
b_{r_{l-1}r_l}d,\;b_{r_lr_1}).
\end{align*}
Comparing with (\ref{eq:kappaHw}), $\kappa_l^{\cC}(ae,\ldots,ae,a)
=\sum_{r=1}^k
c_r^*c_r\tau_r\left(\kappa_l^\D(bd,\ldots,bd,b)\right)$.
Summing over $l$ and recalling (\ref{eq:Rtransform}), for $\|e\|$ sufficiently
small,
\[\cR^\cC_a(e)=\sum_{l \geq 1}\sum_{r=1}^k
c_r^*c_r\tau_r\left(\kappa_l^\D(bd,\ldots,bd,b)\right).\]
Noting that $\|d\| \leq \sum_{s=1}^k \|c_s\|^2\|e\|$ and applying
(\ref{eq:cumulantbound}), we may exchange the order of
summations on the right and move the summation over $l$ inside
$\tau_r$ by linearity and norm-continuity of $\tau$, yielding the desired
result.
\end{proof}

We now perform the desired computation of the Stieltjes transform of $w$.
\begin{lemma}\label{lemma:cauchycomputation}
Under the conditions of Theorem \ref{thm:Wdistribution}, let
$(\A,\tau,p_0,\ldots,p_{2k})$ and
$f_{rs},g_r,h_r$ be as in Lemma \ref{lemma:freeconstruction},
and let $w=\sum_{r,s=1}^k h_r^*g_r^*f_{rs}g_sh_s$. Then for
a constant $C_0>0$, defining $\mathbb{D}:=\{z \in \C^+:|z|>C_0\}$,
there exist analytic functions $a_1,\ldots,a_k:\mathbb{D} \to \C^+ \cup \{0\}$
and $b_1,\ldots,b_k:\mathbb{D} \to \C$ that satisfy, for every
$z \in \mathbb{D}$ and for $m_0(z)=\tau_0((w-z)^{-1})$,
the equations (\ref{eq:Warecursion}--\ref{eq:Wm0}).
\end{lemma}
\begin{proof}
If $H_r=0$ for some $r$, then we may 
set $a_r \equiv 0$, define $b_r$ by (\ref{eq:Wbrecursion}),
and reduce to the case $k-1$. Hence, it suffices to consider
$H_r \neq 0$ for all $r$. 

Define the von Neumann sub-algebras
$\D=\langle p_r:0 \leq r \leq 2k \rangle$,
$\mathcal{F}=\langle \D,\{f_{rs}\}\rangle_{W^*}$,
$\G=\langle \D,\{g_r\}\rangle_{W^*}$,
and $\H=\langle \D,\{h_r\}\rangle_{W^*}$.
Denote by $\F^\D$, $\cR^\D$, and $G^\D$ the $\tau$-invariant
conditional expectation onto $\D$ and the $\D$-valued $\cR$-transform and
Cauchy transform, and similarly for $\mathcal{F}$, $\G$, and $\H$.

We first work algebraically (Steps 1--3), assuming that
arguments $b$ to Cauchy transforms are invertible with $\|b^{-1}\|$
sufficiently small, arguments $b$ to $\cR$-transforms have
$\|b\|$ sufficiently small,
and applying series expansions for $(b-a)^{-1}$. We will check
that these assumptions hold and also establish the desired analyticity
properties in Step 4.\\

\noindent {\bf Step 1:}
We first relate the $\D$-valued Cauchy transform of $w$ to that of
$v:=\sum_{r,s=1}^k g_r^*f_{rs}g_s$. We apply Lemma \ref{lemma:freecompute} with
$q_0=p_0+\sum_{r=k+1}^{2k} p_r$,
$q_r=p_r$ for $r=1,\ldots,k$, $\cC=\H$, and $\B=\langle \mathcal{F},\G \rangle$.
Then for $c \in \H$,
\begin{equation}\label{eq:RHwc}
\cR_w^{\H}(c)=\sum_{r=1}^k h_r^*h_r\tau_r\left(\cR_v^\D\left(
\sum_{s=1}^k p_s\tau_s(h_sch_s^*)\right)\right).
\end{equation}

To rewrite this using Cauchy transforms, for invertible
$d \in \D$ and each $r=1,\ldots,k$, define
\begin{align}
\alpha_r(d)&:=\tau_r\left(h_rG_w^{\H}(d)h_r^*\right),
\label{eq:alphar1}\\
\beta_r(d)&:=
\tau_r\left(\cR_v^\D\left(\sum_{s=1}^k p_s\alpha_s(d)\right)\right).
\label{eq:betar1}
\end{align}
Then (\ref{eq:GRrelation2}) and (\ref{eq:RHwc}) with $c=G_w^\H(d)$ imply
\begin{equation}\label{eq:GwH}
G_w^{\H}(d)
=\left(d-\cR_w^{\H}\left(G_w^{\H}(d)\right)\right)^{-1}
=\left(d-\sum_{r=1}^k h_r^*h_r\beta_r(d)\right)^{-1}.
\end{equation}
Projecting down to $\D$ using (\ref{eq:Gprojection}) yields
\begin{equation}\label{eq:GwD}
G_w^\D(d)=\F^\D\left(\left(d-\sum_{r=1}^k h_r^*h_r\beta_r(d)\right)^{-1}\right).
\end{equation}
Applying (\ref{eq:GwH}) to (\ref{eq:alphar1}),
\begin{equation}\label{eq:alphar2}
\alpha_r(d)=\tau_r\left(h_r\left(d-\sum_{s=1}^k
h_s^*h_s\beta_s(d)\right)^{-1}h_r^*\right).
\end{equation}
Noting that $(p_1+\ldots+p_k)v(p_1+\ldots+p_k)=v$, (\ref{eq:Rtransform}) and
(\ref{eq:kappaid1}) imply $\cR_v^\D(d) \in \langle
p_1,\ldots,p_k \rangle$ for any $d \in \D$, so we may write (\ref{eq:betar1}) as
\[\cR_v^\D\left(\sum_{r=1}^k p_r\alpha_r(d)\right)=\sum_{r=1}^k p_r\beta_r(d).\]
For $r=0$ and $r \in \{k+1,\ldots,2k\}$, set $\beta_r(d)=0$ and define
$\alpha_r(d)$ arbitrarily, say by $\alpha_r(d)=\|d^{-1}\|$.
Since $vp_r=p_rv=0$ if $r=0$ or $r \in \{k+1,\ldots,2k\}$,
applying (\ref{eq:Rtransform}) and multi-linearity of $\kappa_l^\D$, we may
rewrite the above as
\[\cR_v^\D\left(\sum_{r=0}^{2k} p_r\alpha_r(d)\right)=\sum_{r=0}^{2k}
p_r\beta_r(d).\]
Applying (\ref{eq:GRrelation}) with $b=\sum_{r=0}^{2k} p_r\alpha_r(d)$, we get
\begin{equation}\label{eq:GvD}
G^{\D}_v\left(\sum_{r=0}^{2k}p_r\left(\frac{1}{\alpha_r(d)}
+\beta_r(d)\right)\right)=\sum_{r=0}^{2k} p_r\alpha_r(d).
\end{equation}
The relation between $G^\D_w$ and $G^\D_v$ is given by
(\ref{eq:GwD}), (\ref{eq:alphar2}), and (\ref{eq:GvD}).\\

\noindent{\bf Step 2:}
Next, we relate the $\D$-valued Cauchy transforms of $v$ and
$u:=\sum_{r,s=1}^k f_{rs}$. 
We apply Lemma \ref{lemma:freecompute} with $q_0=\sum_{r=0}^k p_r$,
$q_r=p_{r+k}$ for $r=1,\ldots,k$, $\cC=\G$, and $\B=\mathcal{F}$.
Then for $c \in \G$,
\begin{equation}\label{eq:RvG}
\cR_v^{\G}(c)=\sum_{r=1}^k g_r^*g_r\tau_{r+k}\left(\cR_u^\D\left(
\sum_{s=1}^k p_{s+k}\tau_{s+k}(g_scg_s^*)\right)\right).
\end{equation}

To rewrite this using Cauchy transforms, for invertible $d \in \D$
and all $r=1,\ldots,k$, define
\begin{align}
\gamma_{r+k}(d)&=\tau_{r+k}(g_rG_v^{\G}(d)g_r^*),\label{eq:gammar1}\\
\delta_{r+k}(d)&=\tau_{r+k}\left(\cR_u^\D\left(\sum_{s=1}^k
p_{s+k}\gamma_{s+k}(d)\right)\right).\label{eq:deltar1}
\end{align}
As in Step 1, for $r=0,\ldots,k$ let us also define $\delta_r(d)=0$ and
$\gamma_r(d)=\|d^{-1}\|$. Then, noting
$(p_{k+1}+\ldots+p_{2k})u(p_{k+1}+\ldots+p_{2k})=u$, the same arguments as in
Step 1 yield the analogous identities
\begin{align}
G_v^{\D}(d)&=\F^\D\left(\left(d-\sum_{s=1}^k
g_s^*g_s\delta_{s+k}(d)\right)^{-1}\right),\label{eq:GvDpre}\\
\gamma_{r+k}(d)&=\tau_{r+k}\left(g_r\left(d-\sum_{s=1}^k
g_s^*g_s\delta_{s+k}(d)\right)^{-1}g_r^*\right),\label{eq:gammapre}
\end{align}
\begin{equation}
G^{\D}_u\left(\sum_{r=0}^{2k}p_r\left(\frac{1}{\gamma_r(d)}
+\delta_r(d)\right)\right)=\sum_{r=0}^{2k} p_r\gamma_r(d).\label{eq:GuD}
\end{equation}

As $g_r^*g_r$ has moments given by (\ref{eq:gG}), we may write
(\ref{eq:GvDpre}) and (\ref{eq:gammapre}) explicitly:
Denote $d=d_0p_0+\ldots+d_{2k}p_{2k}$ for $d_0,\ldots,d_{2k} \in \C$. As
$d$ is invertible, we have $d^{-1}=d_0^{-1}p_0+\ldots+d_{2k}^{-1}p_{2k}$.
For any $x \in \A$ that commutes with $\D$,
\[(d-x)^{-1}=\sum_{l \geq 0} d^{-1}(xd^{-1})^l=\sum_{l \geq 0} x^ld^{-l-1}.\]
So for $r=1,\ldots,k$, noting that $p_r=p_r^2$ and that $\D$ commutes with
itself,
\begin{align*}
\tau_r\left((d-x)^{-1}\right)&=\frac{N}{m_r}\sum_{l \geq 0}
\tau\left(p_rx^ld^{-l-1}p_r\right)\\
&=\frac{N}{m_r}\sum_{l \geq 0}
\tau\left((p_rx^lp_r)(p_rd^{-1}p_r)^{l+1}\right)
=\sum_{l \geq 0}
\frac{\tau_r(x^l)}{d_r^{l+1}}.
\end{align*}
Noting that $g_s^*g_s$ commutes with $\D$, applying the above to
(\ref{eq:GvDpre}) with $x=\sum_{s=1}^k g_s^*g_s$, and recalling (\ref{eq:gG}),
\begin{align}
\tau_r(G_v^{\D}(d))&=\sum_{l \geq 0}
\frac{\tau_r((g_r^*g_r)^l)\delta_{r+k}(d)^l}{d_r^{l+1}}\nonumber\\
&=\int \sum_{l \geq 0} \frac{x^l\delta_{r+k}(d)^l}{d_r^{l+1}}
\nu_{m_r/n_r}(x)dx\nonumber\\
&=\int \frac{1}{d_r-x\delta_{r+k}(d)}\nu_{m_r/n_r}(x)dx\nonumber\\
&=\frac{1}{\delta_{r+k}(d)}G_{\nu_{m_r/n_r}}^\C(d_r/\delta_{r+k}(d)),
\label{eq:taurGv}
\end{align}
where $G_{\nu_{m_r/n_r}}^\C$ is the Cauchy
transform of the Marcenko-Pastur law $\nu_{m_r/n_r}$.

Similarly, we may write (\ref{eq:gammapre}) as
\begin{align}
\gamma_{r+k}(d)&=\frac{m_r}{n_r}\tau_r\left(\left(d-\sum_{s=1}^k
g_s^*g_s\delta_{s+k}(d)\right)^{-1}g_r^*g_r\right)\nonumber\\
&=\frac{m_r}{n_r} \int \frac{x}{d_r-x\delta_{r+k}(d)}\nu_{m_r/n_r}(x)dx
\nonumber\\
&=\frac{m_r}{n_r}\left(-\frac{1}{\delta_{r+k}(d)}+\frac{d_r}{\delta_{r+k}(d)^2}
G_{\nu_{m_r/n_r}}^\C(d_r/\delta_{r+k}(d))\right)\nonumber\\
&=\frac{m_r}{n_r}\left(-\frac{1}{\delta_{r+k}(d)}+\frac{d_r}{\delta_{r+k}(d)}
\tau_r(G_v^\D(d))\right),\label{eq:gammar2}
\end{align}
where the first equality applies the cyclic property of $\tau$
and the definitions of $\tau_{r+k}$ and $\tau_r$, the second applies
(\ref{eq:gG}) upon passing to a power series and back as above,
the third applies the definition of the Cauchy transform, and the last applies
(\ref{eq:taurGv}). The 
relation between $G_v^\D$ and $G_u^\D$ is given by (\ref{eq:taurGv}),
(\ref{eq:gammar2}), and (\ref{eq:GuD}).\\

\noindent{\bf Step 3:}
We compute $m_0(z)$ for $z \in \C^+$ using
(\ref{eq:GwD}), (\ref{eq:alphar2}), (\ref{eq:GvD}), (\ref{eq:taurGv}),
(\ref{eq:gammar2}), and (\ref{eq:GuD}). Fixing $z \in \C^+$, let us write
\begin{align*}
\alpha_r=\alpha_r(z),\;\;\beta_r=\beta_r(z),\;\;
d_r=\frac{1}{\alpha_r}+\beta_r,\;\;d=\sum_{r=0}^{2k} d_rp_r,\\
\gamma_r=\gamma_r(d),\;\;\delta_r=\delta_r(d),\;\;
e_r=\frac{1}{\gamma_r}+\delta_r,\;\;e=\sum_{r=0}^{2k} e_rp_r.
\end{align*}

Applying (\ref{eq:GwD}) and projecting down to $\C$,
\[m_0(z)=-\tau_0\left(\left(z-\sum_{r=1}^k
h_r^*h_r\beta_r\right)^{-1}\right).\]
Note that $h_r^*h_r$ commutes with $\D$ and $p_0h_r^*h_rp_0=h_r^*h_r$ for each
$r=1,\ldots,k$. Then, passing to a power series as in Step 2, and
then applying (\ref{eq:hH}) and the spectral calculus,
\begin{align}
m_0(z)&=-\sum_{l \geq 0}z^{-(l+1)}
\tau_0\left(\left(\sum_{r=1}^k h_r^*h_r \beta_r\right)^l\right)\nonumber\\
&=-\sum_{l \geq 0}z^{-(l+1)}\frac{1}{p}\Tr \left(\left(\sum_{r=1}^k
\beta_rH_r^*H_r\right)^l\right)\nonumber\\
&=-\frac{1}{p}\Tr \left(z\Id_p-\sum_{r=1}^k \beta_rH_r^*H_r\right)^{-1}.
\label{eq:m0zfinal}
\end{align}
Similarly, (\ref{eq:alphar2}) implies for each $r=1,\ldots,k$
\begin{equation}\label{eq:alpharfinal}
\alpha_r=\frac{1}{m_r}\Tr\left(\left(z\Id_p-\sum_{s=1}^k \beta_sH_s^*H_s
\right)^{-1}H_r^*H_r\right).
\end{equation}

Now applying (\ref{eq:taurGv}) and recalling (\ref{eq:GvD}) and the definition
of $d_r$, for each $r=1,\ldots,k$,
\[\alpha_r=\tau_r(G_v^{\D}(d))
=\frac{1}{\delta_{r+k}}
G_{\nu_{m_r/n_r}}^\C\left(\frac{1}{\alpha_r\delta_{r+k}}
+\frac{\beta_r}{\delta_{r+k}}\right).\]
Applying (\ref{eq:GRrelation2}) and the Marcenko-Pastur $\cR$-transform
$\cR_{\nu_\lambda}^\C(z)=(1-\lambda z)^{-1}$, this is rewritten as
\begin{equation}\label{eq:alphabeta}
\frac{\beta_r}{\delta_{r+k}}
=\cR_{\nu_{m_r/n_r}}^\C(\alpha_r\delta_{r+k})
=\frac{n_r}{n_r-m_r\alpha_r\delta_{r+k}}.
\end{equation}
By (\ref{eq:gammar2}) and (\ref{eq:GvD}),
\begin{equation}\label{eq:gammarfinal}
\gamma_{r+k}=\frac{m_r}{n_r}\frac{\alpha_r\beta_r}{\delta_{r+k}}.
\end{equation}
We derive two consequences of (\ref{eq:alphabeta}) and (\ref{eq:gammarfinal}).
First, substituting for $\beta_r$ in (\ref{eq:gammarfinal}) using
(\ref{eq:alphabeta}) and recalling the definition of $e_{r+k}$ yields
\begin{equation}\label{eq:er}
e_{r+k}=\frac{n_r}{m_r\alpha_r}.
\end{equation}
Second, rearranging (\ref{eq:alphabeta}), we get
$\beta_r/\delta_{r+k}=1+m_r\alpha_r\beta_r/n_r$.
Inserting into (\ref{eq:gammarfinal}) yields this time
\begin{equation}\label{eq:lastbetar}
\beta_r=\frac{n_r}{m_r^2\alpha_r^2}(n_r\gamma_{r+k}-m_r\alpha_r).
\end{equation}

By (\ref{eq:GuD}), for each $r=1,\ldots,k$,
\[\gamma_{r+k}=\tau_{r+k}(G_u^\D(e))=\tau_{r+k}((e-u)^{-1}).\]
Passing to a power series for $(e-u)^{-1}$, applying (\ref{eq:fF}),
and passing back,
\begin{align}
\gamma_{r+k}&=\frac{1}{n_r}\Tr_{r+k}\left(
\diag\left(e_0\Id_{p},\ldots,e_{2k}\Id_{n_k}\right)-\tilde{F}\right)^{-1}
\nonumber\\
&=\frac{1}{n_r}\Tr_r\left(
\diag\left(e_{k+1}\Id_{n_1},\ldots,e_{2k}\Id_{n_k}\right)-F\right)^{-1}
\nonumber\\
&=\frac{1}{n_r}\Tr_r(D^{-1}-F)^{-1}\label{eq:gammadelta}
\end{align}
where the last line applies (\ref{eq:er}) and sets
$D=\diag(D_1\Id_{n_1},\ldots,D_k\Id_{n_k})$
for $D_r=m_r\alpha_r/n_r$. Noting $\Tr_r D=m_r\alpha_r$,
(\ref{eq:lastbetar}) yields
\begin{align}
\beta_r&=\frac{1}{n_rD_r^2}\Tr_r[(D^{-1}-F)^{-1}-D]\nonumber\\
&=\frac{1}{n_r}\Tr_r[(F^{-1}-D)^{-1}]
=\frac{1}{n_r}\Tr_r((\Id_{n_+}-FD)^{-1}F)\label{eq:betarfinal}
\end{align}
where we used the Woodbury identity and $\Tr_r DAD=D_r^2\Tr A$.
(These equalities hold when $F$ is invertible, and hence for all $F$ by
continuity.) Setting $a_r=-m_r\alpha_r/n_r$ and $b_r=-\beta_r$,
we obtain (\ref{eq:Warecursion}), (\ref{eq:Wbrecursion}), and (\ref{eq:Wm0})
from (\ref{eq:m0zfinal}), (\ref{eq:alpharfinal}), and (\ref{eq:betarfinal}).\\

\noindent {\bf Step 4:}
Finally, we verify the validity of the preceding calculations
when $z \in \mathbb{D}:=\{z \in \C^+:|z|>C_0\}$ and
$C_0>0$ is sufficiently large. Call a scalar quantity
$u:=u(N,z)$ ``uniformly bounded'' if $|u|<C$ for all $z \in \mathbb{D}$,
all $N$, and some constants $C_0,C>0$. Call $u$ ``uniformly small''
if for any constant $c>0$ there exists $C_0>0$ such that $|u|<c$ for all
$z \in \mathbb{D}$ and all $N$.

As $\|w\| \leq C$ by Lemma \ref{lemma:freeconstruction}(c),
$c=G_w^\H(z)$ is well-defined by the convergent series
(\ref{eq:Cauchytransform}) for $z \in \mathbb{D}$.
Furthermore by (\ref{eq:Cauchybound}), $\|c\|$ is uniformly small,
so we may apply (\ref{eq:RHwc}).
$\alpha_r(z)$ as defined by (\ref{eq:alphar1}) satisfies
\begin{align*}
\alpha_r(z)&=\tau_r\left(h_r\sum_{l=0}^\infty
\F^{\H}\left(z^{-1}(wz^{-1})^l\right)h_r^*\right)\\
&=\sum_{l=0}^\infty
z^{-(l+1)}\tau(p_r)^{-1}\tau\left(h_r\F^{\H}(w^l)h_r^*\right)
=\sum_{l=0}^\infty z^{-(l+1)}\frac{N}{m_r}\tau(w^lh_r^*h_r)
\end{align*}
for $z \in \mathbb{D}$. Since $|\tau(w^lh_r^*h_r)| \leq \|w\|^l\|h_r\|^2$,
$\alpha_r$ defines an analytic function on $\mathbb{D}$
such that $\alpha_r(z) \sim (zm_r)^{-1}\Tr(H_r^*H_r)$
as $|z| \to \infty$. In particular, since $H_r$ is non-zero by our initial
assumption, $\alpha_r(z) \neq 0$ and
$\Im \alpha_r(z)<0$ for $z \in \mathbb{D}$. This verifies that
$a_r(z)=-m_r\alpha_r(z)/n_r \in \C^+$ and $a_r$ is analytic on
$\mathbb{D}$. Furthermore, $\alpha_r$ is uniformly small for each
$r$. Then applying (\ref{eq:Rtransform}),
multi-linearity of $\kappa_l$, and (\ref{eq:cumulantbound}), it is verified
that $\beta_r(z)$ defined by (\ref{eq:betar1}) is uniformly bounded
and analytic on $\mathbb{D}$. So $b_r(z)=-\beta_r(z)$ is analytic on
$\mathbb{D}$.

As $\beta_r$ is uniformly bounded, the formal series leading
to (\ref{eq:m0zfinal}) and (\ref{eq:alpharfinal}) are convergent for $z \in
\mathbb{D}$. Furthermore, $d_r=1/\alpha_r+\beta_r$ is well-defined
as $\alpha_r \neq 0$, and $\|d^{-1}\|$ is uniformly small.
Then $c=G_v^\G(d)$ is well-defined by (\ref{eq:Cauchytransform})
and also uniformly small, so we may apply (\ref{eq:RvG}). By the same
arguments as above,
$\gamma_{r+k}(d)$ as defined by (\ref{eq:gammar1}) is non-zero and uniformly
small, and $\delta_{r+k}(d)$ as defined by (\ref{eq:deltar1}) is uniformly
bounded. Then the formal series leading to (\ref{eq:taurGv}) and
(\ref{eq:gammar2}) are convergent for $z \in \mathbb{D}$. Furthermore,
$e_r=1/\gamma_r+\delta_r$ is well-defined and $\|e^{-1}\|$ is uniformly small,
so the formal series leading to (\ref{eq:gammadelta}) is convergent for $z \in
\mathbb{D}$. This verifies the validity of the preceding calculations
and concludes the proof.
\end{proof}

To finish the proof of Theorem \ref{thm:Wdistribution}, we show using a
contractive mapping argument similar to \cite{couilletetal,dupuyloubaton}
that (\ref{eq:Warecursion}--\ref{eq:Wbrecursion})
have a unique solution in the stated domains, which is the limit
of the procedure in Theorem \ref{thm:algorithm}. The result then follows from
Lemma \ref{lemma:cauchycomputation} and Corollary \ref{cor:stieltjesapprox}.
These arguments are contained in Appendix \ref{appendix:fixedpoint}.

%% file: Appendix_A_fix.tex
We discuss the details of Section \ref{sec:applications}
and prove Corollaries \ref{cor:oneway}, \ref{cor:nested}, and
\ref{cor:crossedtwoway}. We denote by $\e_l \in \R^l$ the vector of all 1's
and $\J_l=\e_l\e_l^T \in \R^{l \times l}$ the matrix of all 1's.
For $A \in \R^{l \times m}$ and $B \in \R^{l' \times m'}$, we use standard
tensor product notation
\[A \otimes B=\begin{pmatrix} a_{11}B & \cdots & a_{1m}B \\
\vdots & \ddots & \vdots \\
a_{l1}B & \cdots & a_{lm}B \end{pmatrix} \in \R^{ll' \times mm'}.\]
For $m=l_1+\ldots+l_k$ and the direct sum decomposition
$\R^m=\R^{l_1} \oplus \ldots \oplus \R^{l_k}$, we call
\[A=\begin{pmatrix} A_{11} & \cdots & A_{1k} \\
\vdots & \ddots & \vdots \\
A_{k1} & \cdots & A_{kk} \end{pmatrix} \in \R^{m \times m},\qquad
A_{ij}\in \R^{l_i \times l_j} \text{ for all }i,j \in \{1,\ldots,k\}\]
the corresponding block matrix decomposition.
For subspaces $S_1 \subset S_2$ of $\R^m$, we denote by $S_2 \ominus S_1$ the
orthogonal complement of $S_1$ in $S_2$.

\subsection{One-way classification}\label{subapp:oneway}
The model (\ref{eq:onewaymodel}) may be written in the form 
(\ref{eq:mixedmodel}) with $k=2$ upon identifying $X\beta=\e_n \mu^T$,
stacking the $\alpha_i$'s and $\eps_{i,j}$'s as the rows of $\alpha_1 \in \R^{I
\times p}$ and $\alpha_2 \in \R^{n \times p}$, and setting
\[U_1=\begin{pmatrix} \e_{J_1} & 0 & \cdots & 0 \\
0 & \e_{J_2} & \cdots & 0 \\
\vdots & \vdots & \ddots & \vdots \\
0 & 0 & \cdots & \e_{J_I} \end{pmatrix},\]
$U_2=\Id_n$, and $I_1=I$ and $I_2=n$.

Consider the nested subspaces $\col(\e_n) \subset \col(U_1) \subset \R_n$,
where $\col(\e_n)$ is the 1-dimensional span of $\e_n$ and
$\col(U_1)$ is the column span of $U_1$. Let
$\pi_0,\pi_1,\pi_2$ be the orthogonal projections onto $\col(\e_n)$,
$\col(U_1) \ominus \col(\e_n)$, and $\R^n \ominus \col(U_1)$.
Then the quantities $\SS_1$ and $\SS_2$ are equivalently expressed as
$\SS_1=Y^T \pi_1 Y$ and $\SS_2=Y^T \pi_2 Y$, so the
MANOVA estimators (\ref{eq:onewayestimators}) are given by
$\hat{\Sigma}_1=Y^TB_1Y$ and $\hat{\Sigma}_2=Y^TB_2Y$ for
\[B_1=\frac{1}{K}\left(\frac{1}{I-1}\pi_1-\frac{1}{n-I}\pi_2\right),
\qquad B_2=\frac{1}{n-I}\pi_2.\]

To study $\hat{\Sigma}_1$, let us consider instead the matrix
$\check{\Sigma}_1=Y^T\check{B}_1Y$ for
\[\check{B}_1
=\frac{1}{K}\left(\frac{1}{I}(\pi_0+\pi_1)-\frac{1}{n-I}\pi_2\right).\]
Note that the given assumptions imply $K \geq c$ and $\|U_1\| \leq C$ for
constants $C,c>0$. Then
\[\left\|\left(\frac{1}{K(I-1)}-\frac{1}{KI}\right)Y^T\pi_1Y\right\|
\leq \frac{1}{KI(I-1)}\|Y\|^2 \leq C/N\]
a.s.\ for large $N$, so $\hat{\Sigma}_1-\check{\Sigma}_1$
is the sum of a rank-one matrix and a
matrix of norm at most $C/N$. Then we have
$\mu_{\hat{\Sigma}_1}-\mu_{\check{\Sigma}_1} \to 0$ a.s.

We apply Theorem \ref{thm:bulkdistribution} to $\check{\Sigma}_1$:
Let us set
\[E=\diag(J_1,\ldots,J_I)=U_1^TU_1, \qquad V_1=U_1E^{-1/2}.\]
Then $\pi_0+\pi_1=V_1V_1^T$. We may complete the basis and write
$\pi_2=V_2V_2^T$ for
$V_2$ of size $n \times (n-I)$, so that $[V_1 \mid V_2]$ is an orthogonal
matrix of size $n$. Define the orthogonal change of basis matrix
\[Q=\begin{pmatrix} \Id_I & 0 \\ 0 & [V_1 \mid V_2] \end{pmatrix} \in
\R^{(I+n) \times (I+n)}.\]
Recall $F=U^T\check{B}_1U$ from Theorem \ref{thm:bulkdistribution},
where $U=(\sqrt{I_1}U_1 \mid \sqrt{I_2}U_2)=(\sqrt{I}V_1E^{1/2}
\mid \sqrt{n} \Id_n)$. Then $UQ=(\sqrt{I} V_1E^{1/2} \mid
\sqrt{n}V_1 \mid \sqrt{n} V_2)$ in the decomposition
$\R^{I+n}=\R^I \oplus \R^I \oplus \R^{n-I}$, so
\[M:=Q^TFQ
=\frac{1}{K}\begin{pmatrix} E& \sqrt{\frac{n}{I}}E^{1/2} & 0 \\
\sqrt{\frac{n}{I}}E^{1/2} & \frac{n}{I}\Id_I & 0 \\
0 & 0 & -\frac{n}{n-I}\Id_{n-I}
\end{pmatrix}.\]

We must compute the block traces of $(\Id+FD(a))^{-1}F$ in the decomposition
$\R^{I+n}=\R^I \oplus \R^n$, where 
$D(a)=\diag(a_1 \Id_I,a_2\Id_n)$. Note that
$Q$ preserves this decomposition as well as $D(a)$, so
\[S:=Q^T(\Id+FD(a))^{-1}FQ=(\Id+MD(a))^{-1}M.\]
Moving now to the block decomposition $\R^{I+n}=\R^{2I} \oplus \R^{n-I}$,
let us write $M=\diag(RR^T,-r_2^2\Id_{n-I})$ and
$D(a)=\diag(\Delta,a_2\Id_{n-I})$, where we set $r_0^2=1/K$, $r_1^2=n/(KI)$,
$r_2^2=n/(K(n-I))$, and
\[R=\begin{pmatrix} r_0E^{1/2} \\ r_1 \Id_I \end{pmatrix},
\qquad \Delta=\begin{pmatrix} a_1\Id_I & 0 \\ 0 & a_2\Id_I \end{pmatrix}.\]
To compute the upper-left $2I \times 2I$
block $S_{11}$ in this decomposition, we use the
Woodbury identity
\[(\Delta^{-1}+RR^T)^{-1}=\Delta-\Delta R(\Id_I+R^T\Delta R)^{-1}R^T\Delta\]
to obtain
\begin{equation}\label{eq:woodbury}
S_{11}=(\Id_{2I}+RR^T\Delta)^{-1}RR^T
=\Delta^{-1}(\Delta^{-1}+RR^T)^{-1}RR^T
=R(\Id_I+R^T\Delta R)^{-1}R^T.
\end{equation}
We compute $\Id_I+R^T\Delta R=\Id_I+a_1r_0^2E+a_2r_1^2\Id_I$,
which yields
\[S_{11}=\begin{pmatrix}
\diag\left(\frac{IJ_i}{KI+IJ_ia_1+na_2}\right)
& * \\
* & \diag\left(\frac{n}{KI+IJ_ia_1+na_2}\right) \end{pmatrix}\]
for values $*$ that we omit for brevity. The lower-right $(n-I) \times (n-I)$
block of $S$ is given by
\[S_{22}=(\Id-r_2^2a_2\Id)^{-1}(-r_2^2\Id)=-\frac{n}{K(n-I)-na_2}\Id_{n-I}.\]
As $Q$ preserves the decomposition $\R^{I+n}=\R^I \oplus \R^n$, the block
traces of $S$ in this decomposition are the same as those of
$(\Id+FD(a))^{-1}F$. This yields the formulas for $b_1$ and $b_2$
in Corollary \ref{cor:oneway}(a).

We next apply Theorem \ref{thm:bulkdistribution} for
$\hat{\Sigma}_2$: The matrix $F=U^TB_2U$ is now given by
\[F=\begin{pmatrix} 0 & 0 \\
0 & \frac{n}{n-I}\pi_2 \end{pmatrix} \in \R^{(I+n) \times (I+n)}.\]
Then in the decomposition
$\R^{I+n}=\R^I \oplus \R^I \oplus \R^{n-I}$, we have
\[Q^TFQ=\begin{pmatrix} 0 & 0 & 0 \\
0 & 0 & 0\\
0 & 0 & \frac{n}{n-I}\Id_{n-I} \end{pmatrix},\qquad
Q^T\left(\Id+FD(a)\right)^{-1}FQ
=\begin{pmatrix} 0 & 0 & 0 \\
0 & 0 & 0\\
0 & 0 & \frac{n}{n-I+na_2}\Id_{n-I} \end{pmatrix}.\]
Taking block traces, $a_1$ is irrelevant, $b_1 \equiv 0$, and $b_2$ has the form
of Corollary \ref{cor:oneway}(b).

\subsection{Balanced models}\label{subapp:balanced}
We consider models of the form (\ref{eq:mixedmodel}) given by balanced
crossed and nested classification designs. These satisfy the
following ``balanced model conditions'':
\begin{enumerate}[1.]
\item For each $r$, let $c_r=n/I_r$. Then $U_r^TU_r=c_r\Id_{I_r}$, and
$\Pi_r:=c_r^{-1}U_rU_r^T$ is an orthogonal projection onto a subspace
$S_r \subset \R^n$ of dimension $I_r$.
\item Define $S_0=\col(X)$ as the column span of $X$. Then $S_0
\subset S_r \subset S_k=\R^n$ for each $r=1,\ldots,k-1$.
\item Partially
order the subspaces $S_r$ by inclusion, $r' \preceq r$ if $S_r' \subseteq S_r$.
Let $\mathring{S}_r$ denote the orthogonal complement in $S_r$ of all $S_{r'}$
properly contained in $S_r$. Then for each $r$,
\begin{equation}\label{eq:orthosum}
S_r=\bigoplus_{r' \preceq r} \mathring{S}_{r'},
\end{equation}
where $\oplus$ denotes the orthogonal direct sum.
In particular, $\R^n=S_k=\oplus_{r=0}^k \mathring{S}_r$.
\end{enumerate}
We verify below that the models of Sections \ref{subsec:nested} and
\ref{subsec:crossed} are of this form, with the partial orderings of
$S_0,S_1,\ldots,S_k$ depicted in Figure \ref{fig:inclusionlattice}.

\begin{figure}
\includegraphics{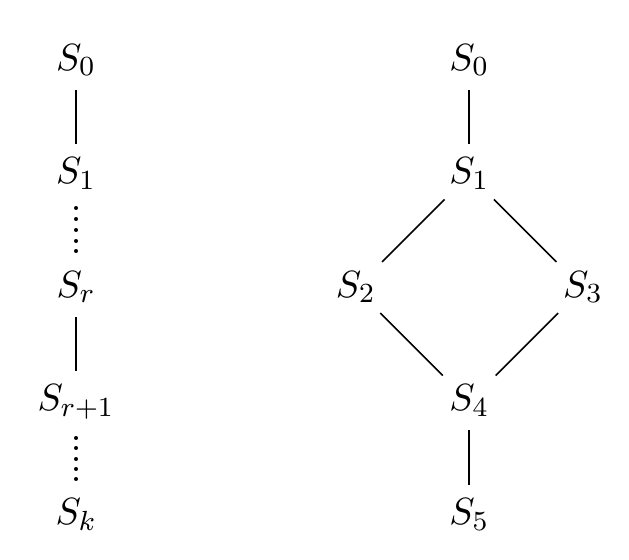}
\caption{Inclusion lattices for the subspaces $\{S_r\}$ determined by the nested
(left) and crossed (right) examples.}\label{fig:inclusionlattice}
\end{figure}

Let $d_r=\dim(\mathring{S}_r)$ and let $\pi_r$ denote the orthogonal
projection onto $\mathring{S}_r$. By (\ref{eq:orthosum}),
$\Pi_r=\sum_{r' \preceq r} \pi_{r'}$. The sum-of-squares and
mean-squares corresponding to $\mathring{S}_r$ are defined respectively as
\[\SS_r=Y^T \pi_r Y,\qquad \MS_r=\SS_r/d_r.\]
To evaluate $\E[\SS_t]$ for $t \geq 1$, note that $\pi_tX=0$ by the condition
$S_0 \subset S_t$. Then
\[\E[\SS_t]=\sum_{r=1}^k \E[\alpha_r^T U_r^T\pi_t U_r \alpha_r]
=\sum_{r=1}^k \Tr(U_r^T \pi_t U_r) \Sigma_r.\]
From the definitions, $\Tr(U_r^T \pi_t U_r)=c_r \Tr(\pi_t\Pi_r)=c_r d_t\1\{t
\preceq r\}$, which yields
\[\E[\MS_t]=\sum_{r \succeq t} c_r\Sigma_r=\sum_r \zeta(t,r)c_r\Sigma_r,\]
where $\zeta(t,z):=\1\{t \preceq r\}$ is the ``zeta function'' associated with
the partial order. We can then use M\"obius inversion to write $\Sigma_r$ in
terms of the expected mean squares,
\[c_r\Sigma_r=\sum_u \mu(r,u)\E[\MS_u],\]
where $\mu(r,u)$ is the associated ``M\"obius function''.
The MANOVA estimators are defined by substituting $\MS_u$ for
$\E[\MS_u]$, which yields
\begin{equation}\label{eq:betatu}
\hat{\Sigma}_t=Y^T B_tY,\qquad B_t=\sum_u \beta_{tu}\pi_u,
\qquad \beta_{tu}=\frac{\mu(t,u)}{c_td_u}.
\end{equation}
For our examples, we may solve for $\mu(t,u)$ directly;
more general discussions regarding the M\"obius inversion
may be found in \cite{speed}.

We apply Theorem \ref{thm:bulkdistribution} to each $\hat{\Sigma}_t$: To compute
$F=U^TB_tU$, we change to a new basis for $\R^{I_+}$. Let $V_r$ be an
$n \times d_r$ matrix whose columns form an orthonormal basis for
$\mathring{S}_r$. Let $\check{V}_r$ be the $n \times I_r$ matrix obtained by
combining the columns of each $V_{r'}$ for $r' \preceq r$. In view of
(\ref{eq:orthosum}), the columns of $\check{V}_r$ are an orthonormal basis for
$S_r$, and we have
\[U_r=\sqrt{c_r}\check{V}_r\check{W}_r^T\]
for some orthogonal matrix $\check{W}_r$ of size $I_r$. The block diagonal
matrix $Q=\diag(\check{W}_r)$ of size $I_+$ is also orthogonal, and
\[U=(\sqrt{I_1}U_1 \mid \cdots \mid \sqrt{I_k}U_k)=\sqrt{n}(\check{V}_1 \mid
\cdots \mid \check{V}_k)Q^T.\]
This yields $n^{-1/2}UQ=(\check{V}_1 \mid \cdots \mid \check{V}_k)$.
The matrix $V_r$ occurs in each $\check{V}_{r'}$ for which $r' \succeq r$. Let
$P$ be the permutation that collects all such occurrences of each $V_r$. Thus,
if we set $s(r)=|\{r':r' \succeq r\}|$ for $r \geq 1$ and $s(0)=k$, then we have
\[n^{-1/2}UQP=\big[\,\e_{s(0)}^T \otimes V_0 \mid \cdots \mid \e_{s(k)}^T
\otimes V_k\,\big].\]
(This is now in the decomposition $\R^{I_+}=\R^{s(0)d_0} \oplus \cdots
\oplus \R^{s(k)d_k}$.)
Now write $O=QP$. Recalling that $F=U^TB_tU$ with $B_t=\sum_u \beta_{tu}
V_u^TV_u$, and noting that $V_r^TV_u=\1\{r=u\}\Id_{d_u}$, we can write the
$(r,r')$ block of $O^TFO$, for $0 \leq r,r' \leq k$, in the form
\begin{align*}
(O^TFO)_{rr'}&=n(\e_{s(r)} \otimes V_r^T)B_t
(\e_{s(r')}^T \otimes V_{r'})\\
&=\1\{r=r'\}n\beta_{tr} \J_{s(r)} \otimes \Id_{d_r}.
\end{align*}
Thus we obtain the block diagonal representation
\[M:=O^TFO=\diag(n\beta_{tr}\J_{s(r)} \otimes
\Id_{d_r}).\]

We wish to compute the block trace of $(\Id+FD(a))^{-1}F$ in the
original decomposition
$\R^{I_+}=\R^{I_1} \oplus \cdots \oplus \R^{I_k}$. As the blocks
of $Q=\diag(\check{W}_r)$ are orthogonal of size $I_r$, $Q$
preserves this decomposition as well as $D(a)$, so we have
\[E(a):=O^TD(a)O=P^TD(a)P
=\diag(D_0(a) \otimes \Id_{d_0},\ldots,D_k(a) \otimes \Id_{d_k}),\]
where $D_r(a)=\diag(a_{r'}:r' \succeq r)$ is $s(r) \times s(r)$. Hence
\begin{align*}
S&:=O^T(\Id+FD(a))^{-1}FO\\
&=(\Id+ME(a))^{-1}M\\
&=\diag\left\{([\Id_{s(r)}+n\beta_{tr}\J_{s(r)}D_r(a)]^{-1}n\beta_{tr}
\J_{s(r)}) \otimes \Id_{d_r}\right\}_{r=0}^k.
\end{align*}
We apply the Woodbury formula via (\ref{eq:woodbury}) with
$R=\sqrt{n\beta_{tr}}\e_{s(r)}$ and $\Delta=D_r(a)$, to obtain
\begin{equation}\label{eq:gammatr}
S=\diag\{\gamma_{tr}\J_{s(r)} \otimes \Id_{d_r}\}_{r=0}^k,
\qquad \gamma_{tr}:=\frac{n\beta_{tr}}{1+n\beta_{tr}q_r},
\qquad q_r:=\e_{s(r)}^TD_r(a)\e_{s(r)}=\sum_{r' \succeq r} a_{r'}.
\end{equation}
Now undo the permutation $P$ and observe that the $(r,r)$ block corresponding to
the decomposition $\R^{I_+}=\R^{I_1} \oplus \cdots \oplus \R^{I_k}$ is given by
\[((P^T)^{-1}SP^{-1})_{rr}=\diag(\gamma_{tu}\Id_{d_u}:u \preceq
r).\]
Hence, we obtain the block trace
\[\Tr_r (\Id+FD(a))^{-1}F=\Tr_r [(P^T)^{-1}SP^{-1}]
=\sum_{u \preceq r} \gamma_{tu}d_u.\]
Recall from (\ref{eq:betatu}) that
$n\beta_{tu}d_u=I_t\mu(t,u)$. Then substituting for $\gamma_{tu}$ as defined in
(\ref{eq:gammatr}), the equation (\ref{eq:brecursion}) simplifies to
\begin{equation}\label{eq:generalb}
b_r(z)=-\frac{I_t}{I_r} \sum_{u \preceq r} \frac{\mu(t,u)}
{1+(I_t/d_u)\mu(t,u)q_u}.
\end{equation}

\subsubsection{Balanced nested classification}\label{subapp:nested}
The model (\ref{eq:balancednestedmodel}) may be
written in the form (\ref{eq:mixedmodel}) upon identifying $I_r=\prod_{s=1}^r
J_s$, $c_r=n/I_r$, $X\beta=\e_n\mu^T$, $U_r=\Id_{I_r} \otimes \e_{c_r}$, and
stacking the vectors $\alpha^{(r)}_*$ as rows of $\alpha_r \in \R^{I_r \times
p}$ and $\eps_*$ as rows of $\alpha_k \in \R^{n \times p}$.
The balanced model conditions are easily verified, with
(\ref{eq:orthosum}) following from the linear structure of the
inclusion lattice. Direct inversion yields
\[\mu(t,u)=\begin{cases}
1 & \text{if } u=t \\
-1 & \text{if } u=t+1\\
0 & \text{otherwise}.
\end{cases}\]
We also have $d_t=I_{t-1}(J_t-1)$, so the form of
the MANOVA estimator from (\ref{eq:betatu}) is
\[\hat{\Sigma}_t=Y^T\left(\frac{J_t}{n(J_t-1)}\pi_t-\frac{1}{n(J_{t+1}-1)}
\pi_{t+1}\right)Y\]
for $t=1,\ldots,k-1$, and the same without the $\pi_{t+1}$ term for $t=k$.
As $I_t/d_t=J_t/(J_t-1)$ and $I_t/d_{t+1}=1/(J_{t+1}-1)$,
(\ref{eq:generalb}) may be written as
\[b_r=\begin{cases} 0 & \text{if } r<t \\
-\frac{J_t-1}{J_t-1+J_tq_t} &\text{if } r=t\\
-\frac{1}{J_{t+1} \ldots J_r} \left(\frac{J_t-1}{J_t-1+J_tq_t}
-\frac{J_{t+1}-1}{J_{t+1}-1-q_{t+1}}\right) &\text{if } r>t.
\end{cases}\]
Noting $q_t=\sum_{r \geq t} a_r$, this implies that $a_1,\ldots,a_{r-1}$ are
irrelevant, and we obtain Corollary \ref{cor:nested}.

\subsubsection{Replicated crossed two-way classification}
The model (\ref{eq:crossedtwowaymodel}) may be written in the form
(\ref{eq:mixedmodel}) with $k=5$, upon identifying
\[I_1=I,\;I_2=IJ,\;I_3=IK,\;I_4=IJK,\;I_5=n,\]
$X\beta=\e_n \mu^T$, $U_1=\Id_I \otimes \e_{JKL}$,
$U_2=\Id_{IJ} \otimes \e_{KL}$, $U_3=\Id_I \otimes \e_J \otimes \Id_K \otimes
\e_L$, $U_4=\Id_{IJK} \otimes \e_L$, $U_5=\Id_n$, and
stacking the vectors $\alpha_*$, $\beta_*$, $\gamma_*$, $\delta_*$, $\eps_*$ as
the rows of matrices $\alpha_1,\ldots,\alpha_5$ respectively. The balanced model
conditions are easily verified, where (\ref{eq:orthosum}) uses the
observation that $\mathring{S}_2=S_2 \ominus S_1$ and $\mathring{S}_3=S_3
\ominus S_1$ are orthogonal.

From the lattice structure of Figure
\ref{fig:inclusionlattice} and direct evaluation, or by a general formula such
as \cite[p.\ 380]{speed}, we find that $M=(\mu(t,u))_{t,u=1,\ldots,5}$ has the
upper triangular form
\[M=\begin{pmatrix} 1 & -1 & -1 & 1 & 0 \\
& 1 & 0 & -1 & 0 \\
& & 1 & -1 & 0 \\
& & & 1 & -1\\
& & & & 1 \end{pmatrix}.\]
We also have the following values:
\begin{center}
\begin{tabular}{c|ccccc}
$t$ & 1 & 2 & 3 & 4 & 5\\
\hline
$d_t$ & $I-1$ & $I(J-1)$ & $I(K-1)$ & $I(J-1)(K-1)$ & $IJKL$ \\
$q_t$ & $\sum_1^5 a_i$ & $a_2+a_4+a_5$ & $a_3+a_4+a_5$ & $a_4+a_5$ & $a_5$
\end{tabular}
\end{center}
Then, for example, the MANOVA estimator $\hat{\Sigma}_2$ from
(\ref{eq:betatu}) is given by
\[\hat{\Sigma}_2
=Y^T\left(\frac{1}{I(J-1)KL}\pi_2-\frac{1}{I(J-1)(K-1)KL}\pi_4\right)Y,\]
and the forms of the other estimators follow similarly.

To explicitly write (\ref{eq:generalb}),
for $t>1$, let $\sigma(t)$ be the successor of $t$ in the partial order.
(We do not need $\sigma(t)$ for $t=5$.) Then $\mu(t,u)$ is only non-zero for
$u=t$ and $u=\sigma(t)$, so we have
\[b_r=\begin{cases} -\frac{1}{1+(I_t/d_t)q_t} & \text{if } r=t \\
-\frac{I_t}{I_r}\left(\frac{1}{1+(I_t/d_t)q_t}-\frac{1}{1-(I_t/d_{\sigma(t)})
q_{\sigma(t)}}\right) & \text{if } r \succ t\\
0 & \text{otherwise}.
\end{cases}\]
Specializing to $t=2$ yields Corollary \ref{cor:crossedtwoway}, and analogous
expressions may be derived for $t=3,4,5$.

For $t=1$, we have the following values for
$\check{\gamma}_u:=(1+(I_1/d_u)\mu(1,u)q_u)^{-1}$:
\begin{center}
\begin{tabular}{c|ccccc}
$u$ & 1 & 2 & 3 & 4\\
\hline
$\check{\gamma}_u$ & $\displaystyle \frac{I-1}{I-1+I\sum_1^5 a_i}$
& $\displaystyle \frac{J-1}{J-1-(a_2+a_4+a_5)}$
& $\displaystyle \frac{K-1}{K-1-(a_3+a_4+a_5)}$
& $\displaystyle \frac{(J-1)(K-1)}{(J-1)(K-1)+a_4+a_5}$
\end{tabular}
\end{center}
Then (\ref{eq:generalb}) simplifies to the equations
\begin{align*}
b_1&=-\check{\gamma}_1\\
b_2&=-J^{-1}(\check{\gamma}_1-\check{\gamma}_2)\\
b_3&=-K^{-1}(\check{\gamma}_1-\check{\gamma}_3)\\
b_4&=-(JK)^{-1}(\check{\gamma}_1-\check{\gamma}_2-\check{\gamma}_3
+\check{\gamma}_4)\\
b_5&=L^{-1}b_4.
\end{align*}

%% file: Appendix_B_fix.tex
We prove Theorem \ref{thm:momentapprox} and Corollary \ref{cor:stieltjesapprox}.
To ease subscript notation, throughout this section we denote by
$M[i,j]$ the $(i,j)$ entry of a matrix $M$.

Let $Q$ be a $*$-polynomial in
$(x_i)_{i \in \I_j,j \in \{1,\ldots,J\}}$ with coefficients in $\langle
P_1,\ldots,P_d \rangle$, and let $q$ denote the
corresponding $*$-polynomial with coefficients in $\langle p_1,\ldots,p_d
\rangle$. For Theorem \ref{thm:momentapprox},
we wish to show for any $r$, almost surely as $N \to \infty$,
\begin{equation}\label{eq:momentapprox}
\left|N_r^{-1}\Tr_r Q\left(H_i:i \in \I_j,j \in
\{1,\ldots,J\}\right)-
\tau_r\left(q\left(h_i:i \in \I_j,j \in \{1,\ldots,J\}\right)\right)
\right| \to 0.
\end{equation}

The high-level strategy of the proof is the same as
\cite[Theorem 1.6]{benaychgeorges}, and follows these steps:
\begin{enumerate}[1.]
\item By applying linearity of $\Tr$ and $\tau$, we may reduce to the case
$Q=\prod_{k=1}^K Q_k$, where each $Q_k$ is a simple-valued polynomial of a
single family $(H_i:i \in \I_{j_k})$.
\item By ``centering'' each $Q_k$ and inducting on $K$, it suffices to consider
the case where $j_1 \neq j_2,\,j_2 \neq j_3,\ldots,j_K \neq j_1$ and each
$Q_k$ satisfies $\Tr Q_k(H_i:i \in \I_{j_k})=0$.
\item The main technical ingredient is Lemma \ref{lemma:momentapprox} below,
which establishes the result for such $Q$. We use orthogonal invariance in law
of $(H_i:i \in \I_{j_k})$ to introduce independently random
block-orthogonal matrices, and
then condition on the $H_i$'s to reduce to a statement about
Haar-orthogonal and deterministic matrices.
\end{enumerate}
The last step above uses an explicit computation of the trace, together with
basic properties of the joint moments of Haar-orthogonal matrices. We follow an
approach inspired by \cite[Theorem 2.1]{hiaipetz}, but which (we believe) fills
in an omission in the proof and also extends the combinatorial argument to 
deal with rectangular matrices and the orthogonal (rather than unitary) case.

\begin{proof}[Proof of Theorem \ref{thm:momentapprox}]
To show (\ref{eq:momentapprox}),
by linearity of $\Tr$ and $\tau$, it suffices to consider the case where $Q$ is
a $*$-monomial, which we may always write as a product of
$Q_1,\ldots,Q_K$ where each $Q_k$ depends only on
the variables of a single family $\I_{j_k}$.
Writing $Q_k=(P_1+\ldots+P_d)Q_k(P_1+\ldots+P_d)$ and again applying linearity
of $\Tr$ and $\tau$, it suffices to consider the case where each 
$Q_k$ is simple-valued, i.e.\ $P_{r_k}Q_kP_{s_k}=Q_k$ for some
$r_k,s_k \in \{1,\ldots,d\}$. If $s_k \neq r_{k+1}$
for any $k$ (with the cyclic identification $r_{K+1} = r_1$), then
(\ref{eq:momentapprox}) is trivial as both quantities on the left are 0.
If $s_k=r_{k+1}$ for all $k$, then it suffices to consider $r=r_1$ and to
replace $N_r^{-1} \Tr_r$ by $N^{-1} \Tr$ and $\tau_r$ by $\tau$.
The result then follows from Lemma \ref{lemma:inductiveclaim} below.
\end{proof}

\begin{lemma}\label{lemma:inductiveclaim}
Under the assumptions of Theorem \ref{thm:momentapprox}, fix $K \geq 1$,
$j_1,\ldots,j_K \in \{1,\ldots,J\}$, and $r_1,\ldots,r_K \in \{1,\ldots,d\}$.
For each $k=1,\ldots,K$, let $Q_k$ be a $*$-polynomial
with coefficients in $\langle P_1,\ldots,P_d \rangle$ of the variables
$(x_i)_{i \in \I_{j_k}}$ of the single family $\I_{j_k}$,
such that $P_{r_k}Q_kP_{r_{k+1}}=Q_k$ (with the
identification $r_{K+1}:=r_1$). Let $q_1,\ldots,q_K$ denote
the corresponding $*$-polynomials with coefficients in $\langle p_1,\ldots,p_d
\rangle$. Then, almost surely as $N \to \infty$,
\begin{equation}\label{eq:inductiveclaim}
\left|\frac{1}{N}\Tr \prod_{k=1}^K Q_k\left(H_i:i \in \I_{j_k}\right)
-\tau\left(\prod_{k=1}^K q_k\left(h_i:i \in \I_{j_k}\right)
\right)\right| \to 0.
\end{equation}
\end{lemma}
\begin{proof}
We induct on $K$. For $K=1$, (\ref{eq:inductiveclaim}) holds by the
assumption that $(h_i)_{i \in \I_{j_1}}$ and $(H_i)_{i \in \I_{j_1}}$ are
asymptotically equal in $\D$-law a.s.

For $K \geq 2$, assume inductively that (\ref{eq:inductiveclaim})
holds for each value $1,\ldots,K-1$ in place of $K$. Let
\[t_k=\frac{1}{\tau(p_{r_k})}
\tau\left(q_k\left(h_i:i \in \I_{j_k}\right)\right),\]
and define the ``centered'' $*$-polynomials
\[D_k=Q_k-t_kP_{r_k},\qquad d_k=q_k-t_kp_{r_k}.\]
We clarify that $t_k \in \C$ is a fixed constant (evaluated at the $h_i$'s,
not at the arguments $x_i$'s of these $*$-polynomials),
and thus $D_k$ and $d_k$ are still
$*$-polynomials of $(x_i)_{i \in \I_{j_k}}$ with coefficients in $\langle
P_1,\ldots,P_d \rangle$ and $\langle p_1,\ldots,p_d \rangle$.
We have $t_k=0$ if $r_k \neq r_{k+1}$, because $q_k$ is simple.
Denoting by $\S_K$ the collection of all subsets of $\{k:r_k=r_{k+1}\}$
and applying a binomial expansion,
\[\frac{1}{N}\Tr \prod_{k=1}^K Q_k\left(H_i:i \in \I_{j_k}\right)
=\sum_{S \in \S_K} Q(S)\]
where
\[Q(S):=\prod_{k \in S} t_k \cdot
\frac{1}{N}\Tr \prod_{k \in \{1,\ldots,K\} \setminus S}
D_k\left(H_i:i \in \I_{j_k}\right).\]
Each $D_k$ still satisfies $P_{r_k}D_kP_{r_{k+1}}=D_k$.
Hence, for every $S \neq \emptyset$, applying the induction hypothesis,
\begin{equation}\label{eq:QSapprox}
\left|Q(S)-\prod_{k \in S} t_k \cdot
\tau\left(\prod_{k \in \{1,\ldots,K\} \setminus S}
d_k\left(h_i:i \in \I_{j_k}\right)\right)\right| \to 0.
\end{equation}
For $S=\emptyset$, if $j_k=j_{k+1}$ for some $k \in
\{1,\ldots,K\}$ (or $j_K=j_1$), then combining $D_kD_{k+1}$ into a single
polynomial (and applying cyclic invariance of $\Tr$ and $\tau$ if $j_K=j_1$),
the induction hypothesis still yields (\ref{eq:QSapprox}).

The remaining case is when $S=\emptyset$ and
$j_k \neq j_{k+1}$ for each $k=1,\ldots,K$. Note, by definition of $d_k$, that
\[\tau\left(p_rd_k\left(h_i:i \in \I_{j_k}\right)p_r\right)=0\]
for each $r$ and $k$, so by freeness of
$(h_i)_{i \in \I_1},\ldots,(h_i)_{i \in \I_k}$ with amalgamation over $\langle
p_1,\ldots,p_d \rangle$,
\[\tau\left(\prod_{k=1}^K d_k\left(h_i:i \in \I_{j_k}\right)
\right)=0.\]
Thus, it remains to show that $Q(\emptyset) \to 0$. 
Note first that
the definition of the free deterministic equivalent and the condition
$N_r/N>c$ imply, almost surely as $N \to \infty$,
\[\left|\frac{N}{N_{r_k}}-\frac{1}{\tau(p_{r_k})}\right| \to 0,\qquad
\left|\frac{1}{N}\Tr \left(Q_k\left(H_i:i \in \I_{j_k}\right)\right)
-\tau\left(q_k\left(h_i:i \in \I_{j_k}\right)\right)\right| \to 0.\]
Hence $|t_k-T_k| \to 0$ a.s.\ for
\[T_k=\frac{1}{N_{r_k}}\Tr Q_k\left(H_i:i \in \I_{j_k}\right).\]
Then it suffices to show
\[M(\emptyset):=\frac{1}{N} \Tr \prod_{k=1}^K M_k \to 0\]
for the matrices
\[M_k=Q_k\left(H_i:i \in \I_{j_k}\right)-T_kP_{r_k},\]
as we may replace in $Q(\emptyset)$ each $t_k$ by $T_k$ and bound the
remainders using the operator norm.

Finally, let us introduce random matrices
$(O_{j,r})_{j \in \N,r \in \{1,\ldots,d\}}$ that are
independent of each other and of the $H_i$'s, such that each $O_{j,r}$ is
orthogonal and Haar-distributed in $\R^{N_r \times N_r}$.
For each $j \in \N$, define the block diagonal matrix
$O_j=\diag(O_{j,1},\ldots,O_{j,d})$.
By orthogonal invariance in law
of $(H_i)_{i \in \I_{j_k}}$, we have the equality in law
\[M(\emptyset)\overset{L}{=}
\frac{1}{N}\Tr \prod_{k=1}^K O_{j_k}M_kO_{j_k}^{-1}.
\]
Write $\check{M}_k \in \R^{N_{r_k} \times N_{r_{k+1}}}$ as the non-zero block of
$M_k$. Then the above may be written as
\begin{equation}\label{eq:Memptyexpression}
M(\emptyset) \overset{L}{=}
\frac{1}{N}\Tr \prod_{k=1}^K O_{j_k,r_k}
\check{M}_kO_{j_k,r_{k+1}}^{-1}\Id_{N_{r_{k+1}}}.
\end{equation}
Conditional on the $H_i$'s, $\check{M}_k$ are deterministic matrices satisfying
$\|\check{M}_k\| \leq C$ for some constant $C>0$ and all large $N$ a.s., and if
$r_k=r_{k+1}$ then $\Tr \check{M}_k=\Tr M_k=0$ by definition of $T_k$.
Furthermore, recall that we are in the case $j_k \neq j_{k+1}$ for
each $k$.

The claim $M(\emptyset) \to 0$ follows from the following lemma:
\begin{lemma}\label{lemma:momentapprox}
Fix $d,K \geq 1$, $l_1,\ldots,l_K \in \N$,
$r_1,\ldots,r_K \in \{1,\ldots,d\}$, and $e_1,\ldots,e_K \in \{-1,1\}$. 
For $N_1,\ldots,N_d \geq 1$,
let $\{O_{l,r}\}_{l \in \N,r \in \{1,\ldots,d\}}$ be independent random matrices
such that
each $O_{l,r}$ is a Haar-distributed orthogonal matrix in $\R^{N_r \times N_r}$.
Let $D_1 \in \C^{N_{r_1} \times N_{r_2}},
D_2 \in \C^{N_{r_2} \times N_{r_3}},\ldots,D_K \in \C^{N_{r_K} \times
N_{r_1}}$ be deterministic matrices
such that, for each $k=1,\ldots,K$ (and cyclically identifying
$l_{K+1}:=l_1$, etc.),
if $(l_k,r_k,e_k)=(l_{k+1},r_{k+1},-e_{k+1})$, then $\Tr D_k=0$.

Let $N=N_1+\ldots+N_d$, and
suppose there exist constants $C,c>0$ such that, as $N \to \infty$,
$N_r/N>c$ for each $r=1,\ldots,d$ and $\|D_k\|<C$
for each $k=1,\ldots,K$. Then, almost surely,
\[N^{-1}\Tr \Big(O_{l_1,r_1}^{e_1}D_1O_{l_2,r_2}^{e_2}D_2\ldots
O_{l_K,r_K}^{e_K}D_K\Big) \to 0.\]
\end{lemma}
\noindent (We emphasize that the matrices $O_{l,r}$ and $D_k$ are
$N$-dependent, while $(l_k,r_k,e_k, k=1,\ldots,K)$ remain fixed as $N$
grows.)

Assuming this lemma for now, write the right side of (\ref{eq:Memptyexpression})
in the form
  \begin{equation*}
    N^{-1}\Tr \Big(O_{l_1,r_1}^{e_1}D_1O_{l_2,r_2}^{e_2}D_2\ldots
O_{l_{2K},r_{2K}}^{e_{2K}}D_{2K}\Big),
 \end{equation*}
by making the identifications
\begin{align*}
  (l_{2k-1},r_{2k-1},e_{2k-1},D_{2k-1})
  & \leftarrow  (j_k,r_k,1,\check{M}_k) \\
  (l_{2k},r_{2k},e_{2k},D_{2k})
  & \leftarrow  (j_k,r_{k+1},-1,\Id_{N_{r_{k+1}}}).
\end{align*}
Then Lemma \ref{lemma:momentapprox} implies $M(\emptyset) \to 0$
a.s.\ conditional on the $H_i$'s, and hence unconditionally as well.
Thus (\ref{eq:QSapprox}) holds for all $S \in \S_K$.

Finally, reversing the
binomial expansion,
\[\sum_{S \in \S_K} \prod_{k \in S} t_k \cdot
\tau\left(\prod_{k \in \{1,\ldots,K\} \setminus S}
d_k\left(h_i:i \in \I_{j_k}\right)\right)
=\tau\left(\prod_{k=1}^K q_k\left(h_i:i \in \I_{j_k}\right)\right).\]
This establishes (\ref{eq:inductiveclaim}), completing the induction.
\end{proof}

To conclude the proof of Theorem \ref{thm:momentapprox}, it remains
to establish the above Lemma \ref{lemma:momentapprox}. We require the following 
fact about joint moments of entries of Haar-orthogonal matrices:

\begin{lemma}\label{lemma:orthogonalproperties}
Let $O \in \R^{N \times N}$ be a random Haar-distributed
real orthogonal matrix, let $K \geq 1$ be any positive integer, and let
$i_1,j_1,\ldots,i_K,j_K \in \{1,\ldots,N\}$. Then:
\begin{enumerate}[(a)]
\item There exists a constant $C:=C_K>0$ such that
\[\E\big[\big|O[i_1,j_1]O[i_2,j_2]\ldots O[i_K,j_K]\big|\big]
\leq CN^{-K/2}.\]
\item If there exists $i \in \{1,\ldots,N\}$ such that $i_k=i$ for
an odd number of indices $k \in \{1,\ldots,K\}$ or $j_k=i$ for an odd number of
indices $k \in \{1,\ldots,K\}$, then $\E[O[i_1,j_1]\ldots O[i_K,j_K]]=0$.
\end{enumerate}
\end{lemma}
\begin{proof}
\cite[Eq.\ (21) and Theorem 3.13]{collinssnaidy}
imply $\E[O[i_1,j_1]^2\ldots O[i_K,j_K]^2] \leq CN^{-K}$ for a constant
$C:=C_K>0$. Part (a) then follows by Cauchy-Schwarz.
Part (b) follows from the fact that the distribution of $O$ is invariant to
multiplication of row $i$ or column $i$ by $-1$, hence if $i_k=i$ or $j_k=i$ for
an odd number of indices $k$, then
$\E[O[i_1,j_1]\ldots O[i_K,j_K]]=-\E[O[i_1,j_1]\ldots O[i_K,j_K]]$.
\end{proof}

\begin{proof}[Proof of Lemma \ref{lemma:momentapprox}]
Define $V_k=O_{l_k,r_k}^{e_k}$ (which is $O_{l_k,r_k}^T$ if $e_k=-1$).
Expanding the trace,
\begin{equation}\label{eq:traceexpansion}
\Tr\left[\prod_{k=1}^K V_kD_k\right]
=\sum_{\i,\j} V(\i,\j)D(\i,\j),
\end{equation}
where the summation is over all tuples
$(\i,\j):=(i_1,j_1,i_2,j_2,\ldots,i_K,j_K)$ satisfying
\begin{equation}\label{eq:ijconditions}
1 \leq i_k,j_k \leq N_{r_k}
\end{equation}
for each $k=1,\ldots,K$,
and where we have defined (with the identification $i_{K+1}:=i_1$)
\[V(\i,\j)=\prod_{k=1}^K V_k[i_k,j_k],\;\;\;\;
D(\i,\j)=\prod_{k=1}^K D_k[j_k,i_{k+1}].\]

Denote
\begin{equation}
  \label{eq:Edef}
\mathcal{E}=\E\left[\left|N^{-1}\Tr \left(\prod_{k=1}^K V_kD_k
\right)\right|^2\right]
=N^{-2}\sum_{\i,\j}\sum_{\i',\j'} D(\i,\j)\overline{D(\i',\j')}
\E[V(\i,\j)V(\i',\j')],
\end{equation}
where the second equality uses that
each $V_k$ is real and each $D_k$ is deterministic.
By the Borel-Cantelli lemma, it suffices to show $\mathcal{E} \leq CN^{-2}$ for
some constant $C:=C_K>0$.

Let $\mathcal{R}$ be the set of distinct pairs among $(l_k,r_k)$ for
$k=1,\ldots,K$, corresponding to the set of distinct matrices $O_{l,r}$ that
appear in (\ref{eq:traceexpansion}). By independence of the matrices $O_{l,r}$,
\begin{equation}\label{eq:EQ}
\E[V(\i,\j)V(\i',\j')]
=\prod_{(l,r) \in \mathcal{R}}
\E\left[\prod_{k:(l_k,r_k)=(l,r)} V_k[i_k,j_k]V_k[i_k',j_k']\right].
\end{equation}
Since $O_{l,r}$ is invariant in law under permutations of rows and columns, each
expectation on the right side above depends only on which indices
are equal, and not on the actual index values.
(For example, denoting $O:=O_{l,r}$,
\begin{equation}
    \label{eq:Qexample}
O[1,2] O^{-1}[2,3] O[1,4] O^{-1}[3,3]
  \stackrel{L}{=} O[8,7] O^{-1}[7,6] O[8,5] O^{-1}[6,6]     
\end{equation}
where the equality in law holds by permutation of both the rows and the
columns of $O$.)
We therefore analyse $\mathcal{E}$  by decomposing the sum in \eqref{eq:Edef} 
over the different relevant partitions of $(\i,\j,\i',\j')$ 
specifying which indices are equal.

More precisely, let
\[\I=(i_k,j_k,i_k',j_k':k=1,\ldots,K)\]
be the collection of all indices, with cardinality $|\I|=4K$.
For each $(l,r) \in \mathcal{R}$, let
\[\cI(l,r)=(i_k,j_k,i_k',j_k':k \text{ such that }l_k=l,r_k=r).\]
These sets $\I(l,r)$ form a fixed partition of $\I$.
For each $(l,r)$, denote by $\Q(l,r)$ any further partition
of the indices in $\cI(l,r)$, and let
\begin{equation}\label{eq:Q}
\Q=\bigsqcup_{(l,r) \in \mathcal{R}} \Q(l,r)
\end{equation}
be their combined partition of $\I$. 
Denoting by $Q_{l,r}=|\Q(l,r)|$
the number of elements of $\Q$ that partition $\I(l,r)$, we may identify
\[\Q \equiv \{(l,r,q):(l,r) \in \mathcal{R},\,q \in \{1,\ldots,Q_{l,r}\}\}.\]

We say that $(\i,\j,\i',\j')$
\emph{induces} $\Q$ if, for every two indices belonging to the same
set $\cI(l,r)$, they are equal in value if and only if they belong to
the same element of $\Q$.\footnote{For example, if $K=2$, in display
\eqref{eq:Qexample},
both $(i_1,j_1,i_2,j_2,i_1',j_1',i_2',j_2')=(1,2,2,3,1,4,3,3)$
and $(8,7,7,6,8,5,6,6)$ induce
\[ \Q(l,r) = \{ \{i_1,i_1'\}, \{j_1,i_2\}, \{j_2,i_2',j_2'\}, \{j_1'\}
  \} \quad \text{with} \quad Q_{l,r}=4. \].}
Then $\E[V(\i,\j)V(\i',\j')]$ is the same for all
$(\i,\j,\i',\j')$ that induce the same partition $\Q$. Thus we may define
$E(\Q)=\E[V(\i,\j)V(\i',\j')]$ for any such $(\i,\j,\i',\j')$ and write
\[\mathcal{E}=N^{-2}\sum_{\Q} E(\Q) \sum_{\i,\j,\i',\j'|\Q}
D(\i,\j)\overline{D(\i',\j')},\]
where the first sum is over all partitions $\Q$ of the form (\ref{eq:Q}),
and the second is over all
$(\i,\j,\i',\j')$ satisfying (\ref{eq:ijconditions}) and inducing
$\Q$.

Applying Lemma \ref{lemma:orthogonalproperties}(a) and the bound $N_r/N>c$
to (\ref{eq:EQ}), we have $|E(\Q)| \leq CN^{-K}$ for a constant $C:=C_K>0$ and
all partitions $\Q$. Thus
\begin{equation}\label{eq:Eexpression}
\mathcal{E} \leq CN^{-2-K} \sum_{\Q:E(\Q) \neq 0} |D(\Q)|
\end{equation}
where
\[D(\Q):=\sum_{\i,\j,\i',\j'|\Q} D(\i,\j)\overline{D(\i',\j')}
=\sum_{\i,\j,\i',\j'|\Q} \prod_{k=1}^K D_k[j_k,i_{k+1}]
\prod_{k=1}^K \overline{D_k}[j_k',i_{k+1}'].\]

For fixed $\Q$, we may rewrite $D(\Q)$ as follows: Denote $L=2K$, $M_k=D_k$, and
$M_{K+k}=\overline{D_k}$. 
Let $\mathfrak{q,q'}:\{1,\ldots,L\} \to \Q$ be the maps such that
$\mathfrak{q}(k),\mathfrak{q'}(k),\mathfrak{q}(K+k),\mathfrak{q'}(K+k)$ are the elements of $\Q$ containing
$j_k,i_{k+1},j_k',i_{k+1}'$, respectively. 
Then
\[D(\Q)=\sum_\a \prod_{\ell=1}^L 
      M_\ell[\a_{\mathfrak{q}(\ell)},\a_{\mathfrak{q'}(\ell)}],\]
where $\sum_\a$ denotes the summation over all maps $\a:\Q \to \N$ such that
$\a(l,r,q) \in \{1,\ldots,N_r\}$ for each $(l,r,q) \in \Q$
and $\a(l,r,q) \neq \a(l,r,q')$ whenever $q \neq q'$. (So $\a$ gives
the index values, which must be distinct for elements of $\Q$
corresponding to the same $(l,r) \in \mathcal{R}$.)

We may simplify this condition on $\a$ by considering
the following embedding: Let
\[\tilde{N}=\sum_{(l,r) \in \mathcal{R}} N_r,\]
and consider the corresponding
block decomposition of $\C^{\tilde{N}}$ with blocks
indexed by $\mathcal{R}$. (So the $(l,r)$ block has size $N_r$.) 
For each $\ell=1,\ldots,L$, if $\mathfrak{q}(\ell)=(l,r,q)$ 
and $\mathfrak{q'}(\ell)=(l',r',q')$, then note that $M_\ell$ is
of size $N_r \times N_{r'}$. Let
$\tilde{M}_\ell \in \C^{\tilde{N} \times \tilde{N}}$ be its embedding
whose $(l,r) \times (l',r')$ block
equals $M_\ell$ and whose remaining blocks equal 0. Then
\[D(\Q)=\sum_\a \prod_{\ell=1}^L 
 \tilde{M}_\ell[\alpha_{\mathfrak{q}(\ell)},\alpha_{\mathfrak{q'}(\ell)}],\]
where $\sum_\a$ now denotes the summation over all maps $\a:\Q \to
\{1,\ldots,\tilde{N}\}$ such that each $\alpha(l,r,q)$ belongs to the $(l,r)$
block of $\{1,\ldots,\tilde{N}\}$,
and the values $\alpha(l,r,q)$ are distinct across all $(l,r,q) \in \Q$.
Extending the range of summation of each $\alpha(l,r,q)$ to all of
$\{1,\ldots,\tilde{N}\}$ simply adds 0 by the definition of $\tilde{M}_\ell$,
so we finally obtain
\begin{equation}\label{eq:DQfinal}
D(\Q)=\sum_{\a_1,\ldots,\a_Q}^* 
 \prod_{\ell=1}^L \tilde{M}_\ell[\alpha_{\mathfrak{q}(\ell)},
\alpha_{\mathfrak{q'}(\ell)}]
\end{equation}
where $Q=|\Q|$ and the sum is over all tuples of $Q$
distinct indices in $\{1,\ldots,\tilde{N}\}$.

We must bound $|D(\Q)|$ for any $\Q$ such that $E(\Q)
\neq 0$.
 By Lemma \ref{lemma:orthogonalproperties}(b) and the expression
(\ref{eq:EQ}) for $E(\Q)$, if $E(\Q) \neq 0$,
then for each $(l,r) \in \mathcal{R}$
and each index value $i \in \{1,\ldots,N_r\}$, there must be
an even number of indices in $\cI(l,r)$ equal in value to $i$, i.e.\ each
element $S \in \Q$ must have even cardinality. Furthermore, if exactly two
indices in $\cI(l,r)$ equal $i$, then they must both be row indices or both be
column indices for $O_{l,r}$. In particular, if
$S \in \Q$ has cardinality $|S|=2$, and
if $S=\{j_k,i_{k+1}\}$ or $S=\{j_k',i_{k+1}'\}$,
then this implies
$(l_k,r_k,e_k)=(l_{k+1},r_{k+1},-e_{k+1})$. The condition of the lemma ensures
in this case that $\Tr D_k=0$, so also $\Tr \tilde{M}_k=\Tr \tilde{M}_{K+k}=0$.

We pause to formulate a lemma which provides the bound for
$|D(\Q)|$ that we need.
\begin{lemma}\label{lemma:deterministichelper}
Fix integers $L,Q \geq 1$ and a constant $B>0$.
Let $\ii,\jj:\{1,\ldots,L\} \to \{1,\ldots,Q\}$ be two fixed maps. Let
$M_1,\ldots,M_L \in \C^{N \times N}$ be such that $\|M_l\| \leq B$ for all $l$.
Call an index $q \in \{1,\ldots,Q\}$ ``good'' if both of the following hold:
\begin{itemize}
\item Exactly two of $\ii(1),\ldots,\ii(L),\jj(1),\ldots,\jj(L)$ are equal to $q$.
\item If $\ii(\ell)=\jj(\ell)=q$ for some $\ell$, then $\Tr M_\ell=0$.
\end{itemize}
Let $T$ be the number of good indices $q \in Q$.

Denote by $\sum_{\a_1,\ldots,\a_Q}^*$
the sum over all tuples of $Q$ indices
$\a_1,\ldots,\a_Q \in \{1,\ldots,N\}$ with all values distinct.
Then, for some constant $C:=C(L,Q,B)>0$,
\begin{equation}\label{eq:deterministichelper}
\left|\sum_{\a_1,\ldots,\a_Q}^* \prod_{\ell=1}^L
M_\ell[\a_{\ii(\ell)},\a_{\mathfrak{q}'(\ell)}]\right|
\leq CN^{Q-T/2}.
\end{equation}
\end{lemma}
  
Assuming this lemma for now, we can complete the proof of Lemma
\ref{lemma:momentapprox}.
We saw that any $S \in \Q$ of cardinality $|S|=2$ is good, for
if $S=\{\mathfrak{q}(\ell),\mathfrak{q}'(\ell)\}$, then either
$S=\{j_k,i_{k+1}\}$ or $S=\{j_k',i_{k+1}'\}$ and so
$\Tr \tilde{M}_\ell=0$.
Letting $T$ be the number of elements of $\Q$
with cardinality 2, we have $2T+4(Q-T) \leq 4K$.
But $T$ is also the number of good indices $q$, so Lemma
\eqref{eq:deterministichelper} implies
\begin{equation}
  \label{eq:keybound}
 |D(\Q)| \leq C\tilde{N}^{Q-T/2} \leq C\tilde{N}^K.   
\end{equation}
Noting that $\tilde{N}/N$ and the number of distinct
partitions $\Q$ are also both bounded by a $K$-dependent constant,
and combining with (\ref{eq:Eexpression}), we obtain $\mathcal{E} \leq CN^{-2}$
as desired, and hence Lemma \ref{lemma:momentapprox}.
\end{proof}

\begin{proof}[Proof of Lemma \ref{lemma:deterministichelper}]
Denote $[L]=\{1,\ldots,L\}$ and $[Q]=\{1,\ldots,Q\}$.
We will show the following claim by induction on $t$: For any $L,Q \geq 1$ and
$B>0$, if the number of good indices $T$ satisfies $T \geq t$, then there
exists a constant $C:=C(L,Q,B,t)>0$ for which
\begin{equation}\label{eq:induction}
\left|\sum_{\a_1,\ldots,\a_Q}^* \prod_{l=1}^L
M_l[\a_{\mathfrak{q}(l)},\a_{\mathfrak{q}'(l)}]\right|
\leq CN^{Q-t/2}.
\end{equation}
The desired result follows from this claim applied with $t=T$ and
$C=\max_{t=0}^Q C(L,Q,B,t)$.

For the base case $t=0$, the left side of (\ref{eq:induction}) is
bounded by $CN^Q$ for $C=B^L$, regardless of $T$, as each entry of $M_l$ is
bounded by $B$.

For the inductive step, let $t \geq 1$, suppose the number $T$ of good indices
satisfies $T \geq t$, and suppose
the inductive claim holds for $t-1,t-2,\ldots,0$.
We consider two cases
corresponding to the two possibilities for goodness of an index $q$:

{\bf Case 1:} There exists a good index $q$ and some $l \in [L]$ such
that $\mathfrak{q}(l)=\mathfrak{q}'(l)=q$ and $\Tr M_l=0$. For notational convenience, assume
without loss of generality that $q=Q$ and $l=L$.
Summing first over $\alpha_1,\ldots,\alpha_{Q-1}$ and 
then over $\alpha_Q$, and noting that no other $\mathfrak{q}(l)$ or
$\mathfrak{q}'(l)$ equals
$Q$ for $l \leq L-1$ because $Q$ is good,
the left side of (\ref{eq:induction}) may be written as
\[\mathrm{LS}:=\left|\sum_{\a_1,\ldots,\a_{Q-1}}^*
\left(\prod_{l=1}^{L-1} M_l[\alpha_{\mathfrak{q}(l)},\alpha_{\mathfrak{q}'(l)}]\right)
\mathop{\sum_{\a_Q=1}^N}_{\a_Q \notin \{\a_1,\ldots,\a_{Q-1}\}}
M_L[\a_Q,\a_Q]\right|.\]
Then applying $\Tr M_L=0$, if $Q=1$, then LS vanishes and there is nothing
further to do. If $Q>1$, we get
\begin{align*}
\mathrm{LS}&=\left|\sum_{\a_1,\ldots,\a_{Q-1}}^*
\left(\prod_{l=1}^{L-1} M_l[\a_{\mathfrak{q}(l)},\a_{\mathfrak{q}'(l)}]\right)
\sum_{\a_Q \in \{\a_1,\ldots,\a_{Q-1}\}} M_L[\a_Q,\a_Q]\right|\\
&\leq \sum_{k=1}^{Q-1}
\left|\sum_{\a_1,\ldots,\a_{Q-1}}^*
\left(\prod_{l=1}^{L-1} M_l[\a_{\mathfrak{q}(l)},\a_{\mathfrak{q}'(l)}]\right)
M_L[\a_k,\a_k]\right|.
\end{align*}
We may apply the induction hypothesis to each of the $Q-1$ terms of the above
sum: Define $\tilde{\mathfrak{q}},\tilde{\mathfrak{q}}':[L] \to [Q-1]$ by
$\tilde{\mathfrak{q}}(l)=\mathfrak{q}(l)$ and $\tilde{\mathfrak{q}}'(l)=\mathfrak{q}'(l)$ for $l \in [L-1]$ and
$\tilde{\mathfrak{q}}(L)=\tilde{\mathfrak{q}}'(L)=k$. Each $q \in [Q-1]$ that was good
for $i,j$ remains good for $\tilde{i},\tilde{j}$, except possibly $q=k$.
Thus the number of good indices for $\tilde{\mathfrak{q}},\tilde{\mathfrak{q}}'$
is at least
$\check{t}:=\max(t-2,0)$. The induction hypothesis implies
\[\mathrm{LS} \leq (Q-1)\cdot C(L,Q-1,B,\check{t})N^{Q-1-\check{t}/2}
\leq (Q-1)\cdot C(L,Q-1,B,\check{t})N^{Q-t/2}.\]

{\bf Case 2:} There exists a good index $q$ and distinct $l \neq l' \in [L]$
such that one of $\mathfrak{q}(l),\mathfrak{q}'(l)$ and one of
$\mathfrak{q}(l'),\mathfrak{q}'(l')$ equal $q$.
For notational convenience, assume without loss of generality that $q=Q$,
$l=L-1$, and $l'=L$. By possibly replacing
$M_{L-1}$ and/or $M_L$ by $M_{L-1}^T$ and/or $M_L^T$, we may further assume
$\mathfrak{q}'(L-1)=\mathfrak{q}(L)=Q$.

Summing first over $\a_1,\ldots,\a_{Q-1}$ and then over
$\a_Q$ as in Case 1, and noting that no $\mathfrak{q}(l)$ or $\mathfrak{q}'(l)$ equals $Q$ for $l \leq
L-2$ because $Q$ is good, the left side of (\ref{eq:induction}) may be
written as
\[\mathrm{LS}:=\left|\sum_{\a_1,\ldots,\a_{Q-1}}^*
\left(\prod_{l=1}^{L-2} M_l[\a_{\mathfrak{q}(l)},\a_{\mathfrak{q}'(l)}]\right)
\mathop{\sum_{\a_Q=1}^N}_{\a_Q \notin \{\a_1,\ldots,\a_{Q-1}\}}
M_{L-1}[\a_{\mathfrak{q}(L-1)},\a_Q] M_L[\a_Q,\a_{\mathfrak{q}'(L)}]\right|.\]
Define $M=M_{L-1}M_L$. Then $\|M\| \leq B^2$, and
\begin{align*}
\mathrm{LS}&=\Bigg|\sum_{\a_1,\ldots,\a_{Q-1}}^*
\left(\prod_{l=1}^{L-2} M_l[\a_{\mathfrak{q}(l)},\a_{\mathfrak{q}'(l)}]\right)
\Bigg(M[\a_{\mathfrak{q}(L-1)},\a_{\mathfrak{q}'(L)}]\\
&\hspace{2in}-\sum_{\a_Q \in \{\a_1,\ldots,\a_{Q-1}\}}
M_{L-1}[\a_{\mathfrak{q}(L-1)},\a_Q]M_L[\a_Q,\a_{\mathfrak{q}'(L)}]\Bigg)\Bigg|\\
&\leq \left|\sum_{\a_1,\ldots,\a_{Q-1}}^*
\left(\prod_{l=1}^{L-2} M_l[\a_{\mathfrak{q}(l)},\a_{\mathfrak{q}'(l)}]\right)
M[\a_{\mathfrak{q}(L-1)},\a_{\mathfrak{q}'(L)}]\right|\\
&\hspace{1in}+\sum_{k=1}^{Q-1}
\left|\sum_{\a_1,\ldots,\a_{Q-1}}^*
\left(\prod_{l=1}^{L-2} M_l[\a_{\mathfrak{q}(l)},\a_{\mathfrak{q}'(l)}]\right)
M_{L-1}[\a_{\mathfrak{q}(L-1)},\a_k]M_L[\a_k,\a_{\mathfrak{q}'(L)}]\right|.
\end{align*}
We may again apply the induction hypothesis to each term of the above sum:
For the first term, each original good index $q \in [Q-1]$
remains good, except possibly
$k:=\mathfrak{q}(L-1)=\mathfrak{q}'(L)$ if $k$ was originally good but now $\Tr M \neq 0$.
Hence for this first term there are still at least
$\check{t}:=\max(t-2,0)$ good indices.
The other $Q-1$ terms are present only if $Q>1$.
For each of these terms, each original good index $q \in [Q-1]$
remains good, except possibly $q=k$---hence there are also at least 
$\check{t}$ good indices. Then the induction hypothesis
yields, similarly to Case 1,
\[\mathrm{LS} \leq \Big(C(L-1,Q-1,B^2,\check{t})+
(Q-1)\cdot C(L,Q-1,B,\check{t}) \Big) N^{Q-t/2}.\]
This concludes the induction in both cases, upon setting
$C(L,Q,B,t)=C(L-1,Q-1,B^2,\check{t})+(Q-1) \cdot C(L,Q-1,B,\check{t})$.
\end{proof}

This concludes the proof of Theorem \ref{thm:momentapprox}. Finally, we prove
Corollary \ref{cor:stieltjesapprox} which establishes the approximation at the
level of Stieltjes transforms.

\begin{proof}[Proof of Corollary \ref{cor:stieltjesapprox}]
Under the given conditions, there exists a constant $C_0>0$ such that
$|\tau(w^l)| \leq C_0^l$ for all $N$ and $l \geq 0$, and also
$|N^{-1}\Tr W^l| \leq \|W\|^l \leq C_0^l$ a.s.\ for all
$l \geq 0$ and all sufficiently large $N$.
Fix $z \in \C^+$ with $|z|>C_0$. Then
$m_w(z)=-\sum_{l=0}^\infty z^{-(l+1)}\tau(w^l)$ and
$m_W(z)=-N^{-1} \Tr(z-W)^{-1}=-\sum_{l=0}^\infty z^{-(l+1)}N^{-1}\Tr
W^l$ define convergent series for all large $N$.
For any $\eps>0$, there exists $L$ such that
\[\left|\sum_{l=L+1}^\infty z^{-(l+1)}N^{-1}\Tr W^l\right|<\eps,\qquad
\left|\sum_{l=L+1}^\infty z^{-(l+1)}\tau(w^l)\right|<\eps\]
for all large $N$, while by Theorem \ref{thm:momentapprox},
as $N \to \infty$
\[\left|\sum_{l=0}^L z^{-(l+1)}N^{-1}\Tr W^l-z^{-(l+1)}\tau(w^l)\right| \to
0.\]
Hence $\limsup_{N \to \infty} |m_W(z)-m_w(z)| \leq 2\eps$ a.s.,
and the result follows by taking $\eps \to 0$.
\end{proof}

%% file: Appendix_C_fix.tex
We analyze the fixed-point equations
(\ref{eq:Warecursion}--\ref{eq:Wbrecursion}) and conclude the proof of the main
result, Theorem \ref{thm:Wdistribution}. The analysis follows arguments
similar to those in \cite{couilletetal} and \cite{dupuyloubaton}.

\begin{lemma}[\cite{caratheodorylandau}]\label{lemma:caratheodory}
Let $\Omega \subseteq \C$ be a connected open set, let $E \subseteq \Omega$ be
any set with an accumulation point in $\Omega$, let $a,b \in \C$ be any two
distinct fixed values, and let $\{f_n\}_{n=1}^\infty$ be a sequence of analytic
functions $f_n:\Omega \to \C$.
If $f_n(z) \notin \{a,b\}$ for all $z \in \Omega$ and $n \geq 1$, and if
$\lim_{n \to \infty} f_n(z)$ exists (and is finite) for each $z \in E$, then
$\{f_n\}_{n=1}^\infty$ converges uniformly on compact subsets of $\Omega$ to
an analytic function.
\end{lemma}
\begin{proof}
The result is originally due to \cite{caratheodorylandau}. It also follows
by the theory of normal families:
$\{f_n\}_{n=1}^\infty$ is a normal family by Montel's fundamental normality
test, see e.g.\ \cite[Section 2.7]{schiff}. Hence every subsequence has a
further subsequence that converges uniformly on compact sets to an analytic
function. All such analytic functions must coincide on $E$, hence they coincide
on all of $\Omega$ by uniqueness of analytic extensions, implying the desired
result.
\end{proof}

In the notation of Theorem \ref{thm:Wdistribution}, denote
$a=(a_1,\ldots,a_k)$, $b=(b_1,\ldots,b_k)$, 
\begin{align*}
f_r(z,b)&=-\frac{1}{n_r}
\Tr\left((z\Id_p+b \cdot H^*H)^{-1}H_r^*H_r\right),\\
g_r(a)&=-\frac{1}{n_r}\Tr_r \left([\Id_{n_+}+FD(a)]^{-1}F\right).
\end{align*}

\begin{lemma}\label{lemma:domainrange}
Under the conditions of Theorem \ref{thm:Wdistribution}:
\begin{enumerate}[(a)]
\item For all $z \in \C^+$ and $b \in (\overline{\C^+})^k$,
$z\Id_p+b \cdot H^*H$ is invertible,
$f_r(z,b) \in \C^+ \cup \{0\}$, and $m_0(z) \in \C^+$ for $m_0$
as defined by (\ref{eq:Wm0}).
\item For all $a \in (\C^+ \cup \{0\})^k$, $\Id_{n_+}+FD(a)$ is invertible and
$g_r(a) \in \overline{\C^+}$.
\end{enumerate}
\end{lemma}
\begin{proof}[Proof of Lemma \ref{lemma:domainrange}]
For any $v \in \C^p$,
\[\Im\left[v^*(z\Id_p+b \cdot H^*H)v\right]
=(\Im z)v^*v+\sum_s (\Im b_s)v^*H_s^*H_sv>0.\]
Hence $z\Id_p+b \cdot H^*H$ is invertible. Letting
$T=(z\Id_p+b \cdot H^*H)^{-1}$,
\begin{align*}
n_rf_r(z,b)=-\Tr TH_r^*H_r
&=-\Tr TH_r^*H_rT^*\left(z\Id_p+b \cdot H^*H\right)^*\\
&=-\overline{z}\Tr TH_r^*H_rT^*-\sum_{s=1}^k \overline{b_s}
\Tr TH_r^*H_rT^*H_s^*H_s.
\end{align*}
As $\Tr TRT^*S$ is real and nonnegative for any Hermitian positive-semidefinite
matrices $R$ and $S$, the above implies $\Im f_r(z,b) \geq 0$.
In fact, as $\Tr TH_r^*H_rT^*>0$ unless $H_r=0$, either
$\Im f_r(z,b)>0$ or $f_r(z,b)=0$. Similarly,
\[pm_0(z)=-\Tr T=-\overline{z}\Tr TT^*-\sum_{s=1}^k \overline{b_s}
\Tr TT^*H_s^*H_s,\]
and as $\Tr TT^*>0$, $\Im m_0(z)>0$. This establishes (a).

For (b), let us first show $\Id_{n_+}+FD(a)$ is
invertible. Note if $a_1=0$, then by the fact that a block matrix
\[\begin{pmatrix} A & B \\ 0 & C \end{pmatrix}\]
is invertible if and only if $A$ and $C$ are invertible, it suffices to show
invertibility of the lower-right $(n_2+\ldots+n_k) \times (n_2+\ldots+n_k)$
submatrix. Hence we may reduce to the case where $a_s \neq 0$, i.e.\ $a_s \in
\C^+$, for all $s$. Suppose $\rank(F)=m$ and let $F^\dagger$ denote the 
pseudo-inverse of $F$, so that $FF^\dagger$ is a projection matrix of rank $m$
onto the column span of $F$. $F^\dagger$ is Hermitian,
since $F$ is. Let $Q$ denote the projection orthogonal to $FF^\dagger$, of rank
$n_+-m$. Then
\[\Id_{n_+}+FD(a)=Q+F(F^\dagger+D(a)).\]
For each $s=1,\ldots,k$, let $P_s$ be the projection of rank $n_s$ such that
$D(a)=\sum_{s=1}^k a_sP_s$. Then for any $v \in \C^{n_+}$,
\[\Im[v^*(F^\dagger+D(a))v]
=\Im[v^*D(a)v]=\sum_s (\Im a_s)v^*P_sv>0,\]
as $v^*F^\dagger v$ and $v^*P_sv$ are real and $\Im a_s>0$ for each $s$.
Hence $F^\dagger+D(a)$ is invertible, so
$\Id_{n_+}+FD(a)$ is of full column rank and thus also invertible.

For the second claim, supposing momentarily that $F$ is
invertible and letting $M=(F^{-1}+D(a))^{-1}$,
\begin{align*}
n_rg_r(a)=-\Tr_r M
&=-\Tr_r \left(M\left(F^{-1}+\sum_{s=1}^k
a_sP_s\right)^*M^*\right)\\
&=-\Tr P_rMF^{-1}M^*-\sum_{s=1}^k\overline{a_s}\Tr P_rMP_sM^*.
\end{align*}
As $\Tr P_rMF^{-1}M^*$ is real and $\Tr P_rMP_sM^*$ is real and
nonnegative, this implies $\Im g_r(a) \geq 0$. By continuity in
$F$, this must hold also when $F$ is not invertible, establishing (b).
\end{proof}

\begin{lemma}\label{lemma:contractivemap}
Let $C,M>0$ and let $\mathcal{S}$ denote the space of $k$-tuples
$b=(b_1,\ldots,b_k)$ such that each $b_r$ is an analytic function
$b_r:\C^+ \to \overline{\C^+}$ and
$\sup_{z \in \C^+:\Im z>M} \|b(z)\| \leq C$. For sufficiently large $C$ and $M$
(depending on $p,n_r,m_r$ and the matrices $H_r$ and $F_{r,s}$ in
Theorem \ref{thm:Wdistribution}):
\begin{enumerate}[(a)]
\item $\rho:\mathcal{S} \times \mathcal{S} \to \R$ defined by
\[\rho(b,\tilde{b}):=\sup_{z \in \C^+:\Im z>M} \|b(z)-\tilde{b}(z)\|\]
is a complete metric on $\mathcal{S}$, and
\item Letting $g=(g_1,\ldots,g_k)$ and $f=(f_1,\ldots,f_k)$ where $g_r$ and
$f_r$ are as above, $b \mapsto g(f(z,b))$ defines a map from $\mathcal{S}$ to
itself, and there exists $c \in (0,1)$
such that for all $b,\tilde{b} \in \mathcal{S}$,
\[\rho(g(f(z,b)),g(f(z,\tilde{b}))) \leq c\rho(b,\tilde{b}).\]
\end{enumerate}
\end{lemma}
\begin{proof}
For part (a), $\rho$ is clearly nonnegative, symmetric, and satisfies the
triangle inequality. By definition of $\mathcal{S}$,
$\rho(b,\tilde{b})<\infty$ for all $b,\tilde{b} \in \S$. By uniqueness of
analytic extensions, $\rho(b,\tilde{b})=0 \Leftrightarrow b=\tilde{b}$, hence
$\rho$ is a metric.
If $\{b^{(l)}\}_{l=1}^\infty$ is a Cauchy sequence in $(\S,\rho)$,
then for each $z \in \C^+$ with $\Im z>M$,
$\{b^{(l)}(z)\}_{l=1}^\infty$ is Cauchy in $(\overline{\C^+})^k$
and hence converges to some
$b(z)=(b_1(z),\ldots,b_k(z)) \in (\overline{\C^+})^k$. Then
Lemma \ref{lemma:caratheodory} implies each $b_r(z)$ has an analytic extension
to all of $\C^+$, and $b_r^{(l)} \to b_r$ uniformly over compact subsets of
$\C^+$. This implies $b_r(z) \in \overline{\C^+}$ for all $z \in \C^+$ and
$\sup_{z \in \C^+:\Im z>M} \|b(z)\| \leq C$, so $b \in \S$. Furthermore
$\rho(b^{(l)},b) \to 0$, hence $(\S,\rho)$ is complete.

For part (b), clearly if $b=(b_1,\ldots,b_k)$ is a $k$-tuple of
analytic functions on $\C^+$, then $g(f(z,b))$ is as well. Now
consider $z \in \C^+$ with $\Im z>M$ and fixed values
$b \in (\overline{\C^+})^k$ with $\|b\| \leq C$, and define
\begin{equation}\label{eq:TR}
T=\left(z\Id_p+b \cdot H^*H\right)^{-1},\qquad
R=\left(\Id_{n_+}+FD(f(z,b))\right)^{-1},
\end{equation}
where invertibility of these quantities follows from Lemma
\ref{lemma:domainrange}. Since $H_s^*H_s$ is positive-semidefinite,
\cite[Lemma 8]{couilletetal} implies $\|T\| \leq (\Im z)^{-1}$. Then if
$C,M>0$ (depending on $p,n_r,m_r,H_r,F_{r,s}$) are sufficiently large,
we have $|f_r(z,b)| \leq C(\Im z)^{-1}$, $\|FD(f(z,b))\|<1/2$, $\|R\|<2$, and
$\|g(f(z,b))\| \leq C$. 
This establishes that for sufficiently large $C,M>0$,
if $b \in \mathcal{S}$, then $g(f(z,b)) \in \mathcal{S}$.

Next, consider also $\tilde{b} \in (\overline{\C^+})^k$ with $\|\tilde{b}\| \leq
C$, and define
$\tilde{T}$ and $\tilde{R}$ by (\ref{eq:TR}) with $\tilde{b}$ in place of $b$.
For each $s=1,\ldots,k$, let $P_s$ be the projection such that
$D(a)=\sum_{s=1}^k a_sP_s$. 
Then by the matrix identity $A^{-1}-(A+E)^{-1}=A^{-1}E(A+E)^{-1}$,
\begin{align*}
f_r(z,b)-f_r(z,\tilde{b})
&=\frac{1}{n_r}\Tr\left(\tilde{T}(T^{-1}-\tilde{T}^{-1})TH_r^*H_r\right)\\
&=\frac{1}{n_r}\sum_{s=1}^k(b_s-\tilde{b}_s)\Tr\left(\tilde{T}
H_s^*H_sTH_r^*H_r\right),\\
g_r(f(z,b))-g_r(f(z,\tilde{b}))&=\frac{1}{n_r} \Tr
P_r\tilde{R}(R^{-1}-\tilde{R}^{-1})RF\\
&=\frac{1}{n_r}\sum_{s=1}^k (f_s(z,b)-f_s(z,\tilde{b}))\Tr P_r\tilde{R}FP_sRF.
\end{align*}
Then $g(f(z,b))-g(f(z,\tilde{b}))=M^{(2)}M^{(1)}(b-\tilde{b})$
for the matrices $M^{(1)},M^{(2)} \in \C^{k \times k}$ having entries
\[M^{(1)}_{rs}=\frac{1}{n_r}
\Tr\left(\tilde{T}H_s^*H_sTH_r^*H_r\right),\;\;
M^{(2)}_{rs}=\frac{1}{n_r}\Tr P_r\tilde{R}FP_sRF.\]
For sufficiently large $C,M>0$, we have $\|T\| \leq (\Im z)^{-1}$,
$\|\tilde{T}\| \leq (\Im z)^{-1}$,
$\|M^{(1)}\| \leq C(\Im z)^{-2}$, $\|R\|<2$, $\|\tilde{R}\|<2$, and
$\|M^{(2)}\| \leq C$, hence $\|M^{(2)}M^{(1)}\| \leq C^2(\Im z)^{-2}
\leq C^2M^{-2}$. Increasing $M$ if necessary so that $C^2M^{-2}<1$, this yields
part (b).
\end{proof}

We conclude the proof of Theorem \ref{thm:Wdistribution} using these lemmas,
Corollary \ref{cor:stieltjesapprox}, and Lemma \ref{lemma:cauchycomputation}.
\begin{proof}[Proof of Theorem \ref{thm:Wdistribution}]
Let $C,M>0$ be ($p,n_r,m_r$-dependent values) such that the conclusions of
Lemma \ref{lemma:contractivemap} hold. Increasing $C$ if necessary, assume
$\|b^{(0)}\|<C$
where $b^{(0)}=(b_1^{(0)},\ldots,b_k^{(0)})$ are the initial values for the
iterative procedure of part (c).
Lemma \ref{lemma:contractivemap} and the Banach fixed point theorem imply
the existence of a unique point $b \in \S$ such that $g(f(z,b))=b$. Defining
$a=f(z,b)$, Lemma \ref{lemma:domainrange} implies $a \in (\C^+ \cup \{0\})^k$
for each $z \in \C^+$. Then $a_r,b_r$ satisfy (\ref{eq:Warecursion}) and
(\ref{eq:Wbrecursion}) for each $z \in \C^+$ by construction, which verifies
existence in part (a). For part (c), define the constant functions
$\tilde{b}^{(0)}_r(z) \equiv b_r^{(0)}$ over $z \in \C^+$.
Then $\tilde{b}^{(0)}:=(\tilde{b}_1^{(0)},\ldots,\tilde{b}_r^{(0)}) \in \S$.
Define iteratively $\tilde{b}^{(t+1)}=g(f(z,\tilde{b}^{(t)}))$. Then
Lemma \ref{lemma:contractivemap} implies
\[c\rho(b,\tilde{b}^{(t)}) \geq \rho(g(f(z,b)),g(f(z,\tilde{b}^{(t)})))
=\rho(b,\tilde{b}^{(t+1)}),\]
for $b$ the above fixed point and some $c \in (0,1)$.
Hence $\rho(b,\tilde{b}^{(t)}) \to 0$ as $t \to \infty$. This implies
by Lemma \ref{lemma:caratheodory} that $\tilde{b}^{(t)}(z) \to b(z)$ for all
$z \in \C^+$, which establishes part (c) upon noting that $\tilde{b}_r^{(t)}(z)$
is exactly the value $b_r^{(t)}$ of the iterative procedure applied at
$z$. Part (c) implies uniqueness in part (a),
since $(b_1^{(t)},\ldots,b_k^{(t)})$ would not converge to
$(b_1,\ldots,b_k)$
if this iterative procedure were initialized to a different fixed point.
For part (b), Lemma \ref{lemma:domainrange} verifies that
$m_0(z) \in \C^+$ for $z \in \C^+$. As $b_1(z),\ldots,b_k(z)$ are analytic,
$m_0(z)$ is also analytic. Furthermore, as $b \in \mathcal{S}$,
$b_1(z),\ldots,b_k(z)$ remain bounded as $\Im z \to \infty$, 
so $m_0(z) \sim -z^{-1}$ as $\Im z \to
\infty$. Then $m_0$ defines the Stieltjes transform of a probability measure
$\mu_0$ by \cite[Lemma 2]{geronimohill}.

It remains to verify that $\mu_0$ approximates
$\mu_W$. Let $f_{rs},g_r,h_r \in \A$ be the free deterministic equivalent
constructed by Lemma \ref{lemma:freeconstruction}, and let
$N=p+\sum_r m_r+\sum_r n_r$. Uniqueness of the
solution $a_r,b_r$ to (\ref{eq:Warecursion})
and (\ref{eq:Wbrecursion}) in the stated domains implies that
the analytic functions $a_1,\ldots,a_k,b_1,\ldots,b_k$ in Lemma
\ref{lemma:cauchycomputation} must coincide with this solution for $z \in
\mathbb{D}$. Then Lemma \ref{lemma:cauchycomputation} implies, for
$z \in \mathbb{D}$,
\[m_w(z):=\tau((w-z)^{-1})=\frac{p}{N}m_0(z)-\frac{N-p}{Nz}.\]
The conditions of Corollary \ref{cor:stieltjesapprox} are satisfied by
Lemma \ref{lemma:freeconstruction}, so Corollary \ref{cor:stieltjesapprox}
implies $m_{\tilde{W}}(z)-m_w(z) \to 0$ as $p,n_r,m_r \to \infty$,
pointwise a.s.\ over $\mathbb{D}$, where $\tilde{W} \in \C^{N \times N}$ is
the embedding of $W$ and $m_{\tilde{W}}$ is its empirical spectral measure.
As
\[m_{\tilde{W}}(z)=\frac{p}{N}m_W(z)-\frac{N-p}{Nz},\]
we have $m_W(z)-m_0(z) \to 0$ pointwise a.s.\ over $\mathbb{D}$.
As $m_W-m_0$ is uniformly bounded over $\{z \in \C^+:\Im z>\eps\}$ for any
$\eps>0$, Lemma \ref{lemma:caratheodory} implies $m_W(z)-m_0(z) \to 0$
pointwise a.s.\ for $z \in \C^+$.
Hence $\mu_W-\mu_0 \to 0$ vaguely a.s.\ (see, e.g.,
\cite[Theorem B.9]{baisilversteinbook}). By the conditions of the theorem and
\cite{yinetal}, $\|W\|$ is almost surely bounded by a constant for all large
$p,n_r,m_r$. Furthermore, by Lemma \ref{lemma:freeconstruction},
we have $\tau(w^l) \leq \|w\|^l \leq C^l$
for some constant $C>0$ and all $l \geq 0$, so $m_w$ and $m_0$ are Stieltjes
transforms of probability measures with bounded support. Then
the convergence $\mu_W-\mu_0 \to 0$ holds weakly a.s., concluding the proof
of the theorem.
\end{proof}

%% file: Appendix_D_fix.tex
We construct the spaces $(\A,\tau,p_1,\ldots,p_d)$
in Examples \ref{ex:semicircle}, \ref{ex:MP}, \ref{ex:deterministic}, and
point the reader to the relevant references that
establish Lemma \ref{lemma:freeconstruction}.

Recall that a von Neumann algebra $\A$ is a sub-$*$-algebra of the space of
bounded linear operators $B(H)$ acting on a Hilbert space $H$, such that $\A$
is $\sigma$-weakly closed and contains the identity.
The trace $\tau$ is positive, faithful, and normal if
$\tau(a^*a) \geq 0$ for all $a \in \A$, $\tau(a^*a)=0$ only if $a=0$,
and $\tau$ is $\sigma$-weakly continuous. (See I.9.1.2 and III.2.1.4 of
\cite{blackadar} for equivalent topological characterizations.) 
$\B$ is a von Neumann sub-algebra of $\A$ if it is algebraically and
$\sigma$-weakly closed.

\begin{lemma}\label{lemma:individualconstruction}
Rectangular probability spaces $(\A,\tau,p_1,\ldots,p_d)$ satisfying the
properties of Examples \ref{ex:semicircle}, \ref{ex:MP}, and
\ref{ex:deterministic} exist, such that in each example, $\A$ is a von Neumann
algebra and $\tau$ is a positive, normal, and faithful trace.
\end{lemma}
\begin{proof}
In Examples \ref{ex:semicircle} and \ref{ex:MP}, let $(\Omega,\mathbb{P})$ be a
(classical) probability space and let $\A$ be the von Neumann algebra of
$d \times d$ random matrices with entries in $L^\infty(\Omega,\mathbb{P})$, the
bounded complex-valued random variables on $\Omega$. ($\A$ acts on the Hilbert
space $H$ of length-$d$ random vectors with elements in
$L^2(\Omega,\mathbb{P})$, endowed with inner-product
$v,w \mapsto \E \langle v,w \rangle$.)
Defining $\tau(a)=N^{-1}\E[\sum_{r=1}^d N_ra_{rr}]$, $\tau$ 
is a positive and faithful trace. As $a \mapsto
\E[a_{rr}]$ is weakly continuous and hence $\sigma$-weakly continuous for each
$r=1,\ldots,d$, $\tau$ is normal. Letting $p_r \in \A$ be the (deterministic)
matrix with $(r,r)$ entry 1 and remaining entries 0,
$(\A,\tau,p_1,\ldots,p_d)$ is a rectangular probability space, and
$\tau(p_r)=N_r/N$ for each $r=1,\ldots,d$.
For Example \ref{ex:semicircle}, the element $g \in \A$ may be realized as the
random matrix with $(r,r)$ entry equal to $X$ and all other entries 0, where
$X \in L^\infty(\Omega,\mathbb{P})$ is a random variable with standard
semi-circle distribution on $[-2,2]$.
For Example \ref{ex:MP}, the element $g \in \A$ may be realized as the matrix
with $(r_1,r_2)$ entry equal to $X$ and all other entries 0,
where $X \in L^\infty(\Omega,\mathbb{P})$ is
the square root of a random variable having the Marcenko-Pastur distribution
(\ref{eq:standardMP}) with $\lambda=N_{r_2}/N_{r_1}$.

For Example \ref{ex:deterministic},
we may simply take $(\A,\tau,p_1,\ldots,p_d)$ to be the
rectangular probability space of deterministic $N \times N$ matrices from
Example \ref{ex:matrixspace}. ($\A$ is the space $B(H)$
for $H=\C^N$, and $\tau$ is clearly positive, faithful, and normal as $H$ is
finite-dimensional.) We may take the elements $b_1,\ldots,b_k \in \A$ to be 
the original matrices $B_1,\ldots,B_k$.
\end{proof}

The sub-$*$-algebras $\D$ in the three examples above are isomorphic. They are
also finite-dimensional, hence $\sigma$-weakly closed, so each is a von Neumann
sub-algebra of $\A$.

\begin{proof}[Proof of Lemma \ref{lemma:freeconstruction}]
For each $r=1,\ldots,k$, let $(\A^{(r)},\tau^{(r)},p_0,\ldots,p_{2k})$ be the
space constructed as in Lemma \ref{lemma:individualconstruction} corresponding
to Example \ref{ex:MP} and containing the element $g_r$, satisfying conditions
1, 2, and 4. Let $(\A^{(k+1)},\tau^{(k+1)},p_0,\ldots,p_{2k})$ and
$(\A^{(k+2)},\tau^{(k+2)},p_0,\ldots,p_{2k})$
be the spaces constructed as in Lemma \ref{lemma:individualconstruction}
corresponding to Example \ref{ex:deterministic} and containing the families
$\{h_r\}$ and $\{f_{rs}\}$, respectively, satisfying conditions 1, 2, and 3.
$\D=\langle p_0,\ldots,p_{2k} \rangle$ is a common (up to
isomorphism) $(2k+1)$-dimensional von Neumann sub-algebra of each $\A^{(r)}$,
and each $\tau^{(r)}$ restricts to the same trace on $\D$. Then the
construction of the finite von Neumann amalgamated free product of
$(\A^{(1)},\tau^{(1)}),\ldots,(\A^{(k+2)},\tau^{(k+2)})$ with amalgamation
over $\D$ \cite{voiculescusymmetries,popa} yields a von Neumann algebra
$\A$ with a positive, faithful, and normal trace $\tau$ such that:
\begin{itemize}
\item $\A$ contains (as an isomorphically embedded von Neumann sub-algebra) each
$\A^{(r)}$, where $\A^{(r)}$ contains the common sub-algebra $\D$.
\item Letting $\F:\A \to \D$ and $\F^{(r)}:\A^{(r)} \to \D$
denote the $\tau$-invariant and $\tau^{(r)}$-invariant conditional
expectations, $\F|_{\A^{(r)}} \equiv \F^{(r)}$.
\item $\tau=\tau^{(r)} \circ \F$ for any $r$, so in particular,
$\tau|_{\A^{(r)}}=\tau^{(r)}$.
\item The sub-algebras $\A^{(1)},\ldots,\A^{(k+2)}$ of $\A$ are free with
amalgamation over $\D$ in the $\D$-valued probability space
$(\A,\D,\F)$.
\end{itemize}
(For more details about the amalgamated free product construction, see the
Introduction of \cite{dykema} and also Section 3.8 of \cite{voiculescubook}.)
Since $\tau$ restricts to $\tau^{(r)}$ on each $\A^{(r)}$, conditions 1--4
continue to hold for the
elements $p_r,f_{rs},g_r,h_r$ in $\A$. The generated von Neumann algebra
$\langle D,g_r \rangle_{W^*}$ is contained in $\A^{(r)}$ and similarly for
$\langle D,h_1,\ldots,h_k \rangle_{W^*}$ and $\langle
D,f_{11},f_{12},\ldots,f_{kk} \rangle_{W^*}$, so $\D$-freeness of
these algebras is implied by the $\D$-freeness of the sub-algebras
$\A^{(r)}$. The elements $f_{rs},g_r,h_r$ have bounded norms in the original
algebras $\A^{(1)},\ldots,\A^{(k+2)}$ and hence also in the free product.
\end{proof}

%% file: Appendix_E.tex
We illustrate the proof of Theorem \ref{thm:Wdistribution} in the special setting
of Remark \ref{remark:MP}: 
$Y$ has $n$ i.i.d.\ rows distributed as $\Nor(0,\Sigma)$
and we consider the sample covariance matrix, so that $k=1$,
$B = n^{-1} \Id_n$,
\begin{equation*}
  W = n^{-1} Y^TY = \Sigma^{1/2}G^T G \Sigma^{1/2},
\end{equation*}
and $G$ is an $n \times p$ matrix with i.i.d. $\Nor(0,1/n)$ entries.

Let $O_l$ and $O_r$ be Haar distributed $p \times p$
orthogonal matrices, independent of each other and of $G$,
and let $H = O_l^T \Sigma^{1/2} O_r$ be a
randomized version of $\Sigma^{1/2}$. With slight abuse of notation,
we rewrite
\begin{equation*}
  W = H^TG^T G H,
\end{equation*}
as the spectrum of $W$ is unchanged by the replacement.

\medskip
\textit{Approximation.} \
The matrices $W$, $H$, and $G$ are embedded into larger block matrices, in the
following regions corresponding to the decomposition
$\C^N = \C^p \oplus \C^p \oplus \C^n$
\begin{equation*}
  \begin{bmatrix}
    W & H^* & \cdot \\
    H & \cdot    & G^* \\
    \cdot  & G   & \Id_n
  \end{bmatrix}
\end{equation*}
(we use conjugate transpose notation even though -- at this point --
all matrix entries are real).

More formally, let $N_0=N_1=p$, $N_2=n$ and $N = 2p+n$.
We define diagonal projection matrices $P_i$ having $\Id_{N_i}$ in the $i$th
diagonal block and zeros elsewhere.
Let $\tilde{H}$ and $\tilde{G}$ denote the embeddings of $H$ and $G$
into $N \times N$ matrices, as above.
Then $(\C^{N\times N}, N^{-1} \Tr, P_0,P_1,P_2)$ is a rectangular
probability space as in Example \ref{ex:matrixspace}, and
$\tilde{H}, \tilde{G}$ are independent simple matrices in $\C^{N \times
  N}$, with each being block-orthogonally invariant.
That is, for any orthogonal matrices $O_r \in \R^{N_r \times N_r}$,
$O_1^T H O_0$ has the same law as $H$ and $O_2^TGO_1$ has the same law
as $G$.

For the approximating free model, consider a rectangular probability
space $(\A,\tau,p_0,p_1,p_2)$
with sub-$*$-algebra $\mathcal{D} = \langle p_0,p_1,p_2 \rangle$ and
with deterministic elements $g,h \in \A$ satisfying the following conditions:
\begin{enumerate}[1.]
\item $\tau(p_0) = \tau(p_1) = p/N$, $\tau(p_2)=n/N$.
\item $g$ and $h$ are simple: \quad $p_2gp_1 = g, \quad p_1hp_0 = h$.
\item For each $l \geq 0$,
\begin{equation}
  \label{eq:hHeq}
  \tau_0((h^*h)^l) = p^{-1} \Tr((H^*H)^l).
\end{equation}
\item $g^*g$ has Marcenko-Pastur law: For each $l \geq 0$,
\begin{equation}\label{eq:gGeq}
\tau_1((g^*g)^l)=\int x^l \nu_\lambda(x)dx
\end{equation}
where $\nu_\lambda$ is as in (\ref{eq:standardMP}) with $\lambda=p/n$.
\item $\langle \D,g \rangle_{W^*}$ and $\langle \D,h\rangle_W^*$ are $\D$-free.
\end{enumerate}

Since $h$ is $(1,0)$-simple, (\ref{eq:hHeq}) is enough to specify the full
$\D$-law of $h$, and it implies that $h$ and $\tilde{H}$ are equal in $\D$-law.
Similarly, (\ref{eq:gGeq}) is enough to specify the full $\D$-law of $g$,
and $g$ and $\tilde{G}$ are asymptotically equal in $\D$-law as argued in
Example \ref{ex:MP}. Finally, by definition, $\tau(p_r) = N^{-1} \Tr(P_r)$.
Therefore $(\A,\tau,p_0,p_1,p_2)$ along with $g,h$ forms a free
deterministic equivalent for $(\C^{N\times N}, N^{-1} \Tr,
P_0,P_1,P_2)$ along with $\tilde{G}, \tilde{H}$.

For constants $C,c>0$, suppose that $n, p \to \infty$ in
such a way that $c < p/n < C$ and $\| \Sigma\| < C$. 
Theorem \ref{thm:momentapprox} asserts that the pairs $(\tilde{H},\tilde{G})$
and $(h,g)$ are \textit{jointly} asymptotically equal in $\D$-law a.s.
In particular
\begin{equation*}
  N_r^{-1}\Tr[Q(\tilde{G},\tilde{H})] - \tau_r[Q(g,h)]
     \stackrel{\rm a.s.}{\to} 0
\end{equation*}
for any $*$-polynomial $Q$.
Corollary \ref{cor:stieltjesapprox}
applies this to $Q(g,h) = (h^*g^*gh)^l$ for each
positive integer $l$, and arrives
at a conclusion about approximation of Stieltjes transforms of
$\tilde{W}$ and
$w = h^*g^*gh$, namely that for all large $z \in \C^+$,
\begin{equation*}
  m_{\tilde{W}}(z) - m_w(z) \stackrel{\rm a.s.}{\to} 0.
\end{equation*}
In terms of the non-embedded matrix $W$, denoting $m_0(z)=\tau_0((w-z)^{-1})$,
we deduce
\begin{equation*}
  m_W(z) - m_0(z) \stackrel{\rm a.s.}{\to} 0.
\end{equation*}

\medskip
\textit{Computation.} \
We develop equations for $m_0(z) = \tau_0((w-z)^{-1})$ in the approximating free
model, proving the special case of Lemma \ref{lemma:cauchycomputation}: For 
$C_0>0$ large, there exist analytic functions 
$a_1: \mathbb{D}(C_0) \to \C^+\cup\{0\}$ and
$b_1: \mathbb{D}(C_0) \to \C$ so that equations
(\ref{eq:MPab}--\ref{eq:MPm0}) of Remark \ref{remark:MP} hold.

A suitably specialized form of Proposition \ref{prop:freeness}
shows the role of $\cD$-freeness of $g$ and $h$:
Let $b = g^*g$. If $\mathcal{H}:=\langle \cD, h \rangle_{W^*}$ and
$\langle \cD,b \rangle_{W^*}$ are $\D$-free, then for all $l \geq 1$,
\begin{equation}\label{eq:ghfree}
  \kappa_l^\H(bh, \ldots, bh, b) 
     = \kappa_l^\cD(b\F^\cD (h), \ldots, b\F^\cD (h), b).
\end{equation}

\begin{remark}
  Classical cumulants of a random variable $X$ are derived from the
  log moment generating function
$    \log \E e^{bX} = \sum_{l \geq 1} \kappa_l(X) b^l/l!$.
In the notation of classical multivariate conditional cumulants 
(e.g. \cite{Brillinger1969,speed}), the conditional distribution of
$X$ given a $\sigma$-field 
$\H$ is described by
\begin{equation*}
  \log \E(e^{bX}|\H) = \sum_{l \geq 1} \kappa_l(bX, \ldots, bX \mid \H)/l!.
\end{equation*}
If classical variables $X$ and $Y$ are conditionally independent given
a $\sigma$-field $\cD$, and $\H$ is the $\sigma$-field generated by $\cD$
and $Y$, then
$\E(e^{bX}|\H) = \E(e^{bX}|\cD)$.
Proposition \ref{prop:freeness} may be seen as a non-commutative version of this
identity, written in terms of cumulants.
\end{remark}

Using (\ref{eq:ghfree}), we may express a possibly
complicated transform $\cR_w^\H$ in terms of a simpler one,
namely $\cR_{g^*g}^\cD$.  Indeed, the simpler version of Lemma
\ref{lemma:freecompute} needed here is

\begin{lemma}
Let $\langle \D,g \rangle_{W^*}$ and
$\H:=\langle \D,h\rangle_{W^*}$ be $\cD$-free, and let $w = h^*g^*gh$.
For $c \in \H$ with $\|c\|$ sufficiently small,
\begin{equation}
  \label{eq:HtoD}
  \cR_w^\H (c) = h^*h \, \tau_1(\cR_{g^*g}^\cD (p_1 \tau_1(hch^*))).
\end{equation}
\end{lemma}
\begin{proof}
We use expression \eqref{eq:Rtransform} for $\cR_w^\H (c)$ in terms
of cumulants. We have
\begin{align*}
  \kappa_l^\H (wc, \ldots, wc, w) 
    & = h^* \kappa_l^\H (g^*g \, hch^*, \ldots, g^*g \, hch^*, g^*g) h
  \\
    & = h^* \kappa_l^\cD (g^*g \, \F^\cD(hch^*), \ldots, g^*g \,
      \F^\cD(hch^*), g^*g) h.        
\end{align*}
Here the first equality uses properties
(\ref{eq:kappaid1}--\ref{eq:kappaid2}) of $\kappa^\H$, while 
the second equality relies on $\cD$-freeness of $g$ and $h$ through 
(\ref{eq:ghfree}).
Since $hch^*$ is $(1,1)$-simple, we have from \eqref{eq:FD} that
$\F^\cD (hch^*) = \sum p_r \tau_r(hch^*) = p_1 \tau_1(hch^*)$.
Summing over $l$ in the previous display, we obtain
\begin{equation*}
  \cR_w^\H (c) = h^* \cR_{g^*g}^\cD (p_1\tau_1(hch^*))h.
\end{equation*}
Since $h$ is $(1,0)$-simple, for any $a \in \D$ we have
$h^*ah=\sum_r h^*p_rh\tau_r(a)=h^*h\tau_1(a)$. Noting that $\cR^\D_{g^*g}$ is
$\D$-valued, we obtain \eqref{eq:HtoD}.
\end{proof}

To deduce (\ref{eq:MPab}--\ref{eq:MPm0}), note first that since
$\tau_0(\F^\D(a))=\tau_0(a)$, we have
\begin{equation}\label{eq:E0}
-m_0(z)=\tau_0((z-w)^{-1})=\tau_0 \circ \F^\D((z-w)^{-1})
=\tau_0(G_w^\D(z))=\tau_0 \circ \F^\D(G_w^\H(z)),
\end{equation}
the last step applying (\ref{eq:Gprojection}). For fixed
$z \in \mathbb{D}(C_0)$, define
\begin{equation}
  \label{eq:E1}
  \alpha = \tau_1(h G_w^\H (z)h^*), \qquad
  \beta  = \tau_1( \cR^{\cD}_{g^*g}(p_1 \alpha)).
\end{equation}
We can then rewrite the inversion formula (\ref{eq:GRrelation2})
using \eqref{eq:HtoD}, with $c = G_w^\H(z)$, as
\begin{equation}
  \label{eq:E3}
  G_w^\H(z) = (z-\cR_w^\H(G_w^\H(z)))^{-1}
            = (z-h^*h \beta)^{-1}.
\end{equation}
Lemma \ref{lemma:MPR} below computes $\cR_{g^*g}^\D(p_1\alpha)$ using the
$\cR$-transform of the standard Marcenko-Pastur law, yielding
\[\cR_{g^*g}^\D(p_1\alpha)=p_1(1-\lambda \alpha)^{-1}\]
for $\lambda=p/n$. Applying this and (\ref{eq:E3})
to (\ref{eq:E0}--\ref{eq:E1}), we obtain the equations
\begin{equation*}
  \alpha = \tau_1(h(z-h^*h \beta)^{-1} h^*),\quad
  \beta  = (1-\lambda \alpha)^{-1},\quad
-m_0(z)=\tau_0((z-h^*h\beta)^{-1}).
\end{equation*}
Now passing to a power series, then applying \eqref{eq:hHeq} and the
spectral calculus, we obtain
\begin{equation}\label{eq:m0series}
  -m_0(z) 
   = \sum_{l \geq 0}z^{-(l+1)}\tau_0((h^*h)^l) \beta^l
   = \sum_{l \geq 0}z^{-(l+1)}\frac{1}{p}\Tr((H^*H)^l) \beta^l
     = \frac{1}{p} \Tr (z \Id_p - \beta H^*H)^{-1}.
\end{equation}
Using the definition of $\tau_r$ and cyclic property of $\tau$,
a similar calculation shows that
\begin{equation}\label{eq:alphaseries}
  \alpha = \frac{\tau(p_0)}{\tau(p_1)}\tau_0((z-h^*h\beta)^{-1}h^*h)
=\frac{1}{n} \Tr [(z \Id_p - \beta H^*H)^{-1} H^*H].
\end{equation}
Setting $a_1 = -(p/n) \alpha = - \lambda \alpha$ and $b_1 = - \beta$
and recalling that $\Sigma = H^*H$, we recover (\ref{eq:MPab}--\ref{eq:MPm0}).

We check a few analytic details of the above argument: For $z \in
\mathbb{D}(C_0)$ and $C_0$ sufficiently large, $\alpha$ is defined by
the series expansion (\ref{eq:Cauchytransform}) and we have
\[\alpha=\tau_1\left(h\sum_{l=0}^\infty \F^\H(z^{-1}(wz^{-1})^l)h^*\right)
=\sum_{l=0}^\infty z^{-(l+1)}\tau_1(hw^lh^*).\]
For $C_0$ sufficiently large, boundedness of $\tau$ implies that
$\alpha$ is analytic in $z$, with $\alpha \sim z^{-1}\tau_1(hh^*)$ as $z \to
\infty$. Then either $h=0$ in which case $\alpha=0$ for all $z$, or positivity
and faithfulness of $\tau$ yields $\Im \alpha<0$ and $\Im a_1>0$ for all
$z \in \mathbb{D}(C_0)$. Furthermore, this implies $\beta=(1-\lambda
\alpha)^{-1}$ is also analytic in $z$ and bounded over $\mathbb{D}(C_0)$,
which justifies the use of formal series and spectral calculus in
(\ref{eq:m0series}) and (\ref{eq:alphaseries}) for large $C_0$. This establishes
Lemma \ref{lemma:cauchycomputation} in this special case.

\begin{lemma}\label{lemma:MPR}
For any $z \in \C$ with $|z|$ sufficiently small,
$\cR_{g^*g}^\cD(p_1z)=p_1(1-\lambda z)^{-1}$.
\end{lemma}
\begin{proof}
We first verify that
the Marcenko-Pastur law $\nu_\lambda$ given in (\ref{eq:standardMP}) has
$\cR$-transform $\cR(z)=(1-\lambda z)^{-1}$: Indeed, its
Stieltjes transform $m(z)$ satisfies the functional equation
  \begin{equation}\label{eq:FE}
    m(z) = (1-\lambda-\lambda z m(z) - z)^{-1}
  \end{equation}
  for each $z\in\C^+$ \cite[eq.(1.4)]{silv95}---this is the limiting
  version of (\ref{eq:mpequation}) for $\Sigma = \Id$ as $p/n\to \lambda$.
  The Cauchy transform $w=-m(z)$ has a
  functional inverse which we write as $z=K(w)$. Rewriting
  \eqref{eq:FE} in terms of $w$ and $K(w)$   yields
  \begin{equation*}
    1-\lambda+(\lambda w-1)K(w) = -w^{-1}.
  \end{equation*}
  The $\cR$-transform is then $\cR(w)=K(w)-1/w$, for
  example from (\ref{eq:GRrelation}).
Inserting $K(w)=R(w)+1/w$ into the previous
  display and rearranging yields $\cR(w)=(1-\lambda w)^{-1}$.

Now denote by $\cR_{g^*g}^{\C,1}$ the scalar $\cR$-transform of
$g^*g$ with respect to trace $\tau_1$. By (\ref{eq:gGeq}), the above implies
$\cR_{g^*g}^{\C,1}(\alpha)=(1-\lambda \alpha)^{-1}$. The lemma follows from
relating $\cR_{g^*g}^\D(p_1\alpha)$ to $\cR_{g^*g}^{\C,1}(\alpha)$.
In the full proof of Lemma \ref{lemma:cauchycomputation}, we do the analogous
step by relating $\cR_{g^*g}^\cD$ to $G_{g^*g}^\cD$ using (\ref{eq:GRrelation}),
projecting down to $G_{g^*g}^\C$ using (\ref{eq:Gprojection}), and relating 
this back to $\cR_{g^*g}^\C$. Here, we use a simpler direct argument:

Noting that
$\alpha \in \C \subset \cD$ and $g^*gp_1\alpha=\alpha g^*g$, we have by
(\ref{eq:Rtransform}) and (\ref{eq:kappaid1}--\ref{eq:kappaid2}) that
\begin{equation}\label{eq:RDpalpha}
\cR_{g^*g}^{\cD}(p_1\alpha)
=\sum_{l \geq 1} \alpha^{l-1}\kappa_l^\D(g^*g,\ldots,g^*g).
\end{equation}
Since $g^*g$ is $(1,1)$-simple, the $\cD$-valued moments of $g^*g$ are given by
$\F^\cD((g^*g)^l)=p_1\tau_1((g^*g)^l)$.
The cumulants are defined by the moment-cumulant relations
\[\kappa_l^\cD(g^*g,\ldots,g^*g)=\sum_{\pi \in \text{NC}(l)}
\mu(\pi,\{1,\ldots,l\})\prod_{S \in \pi} \F^\cD((g^*g)^{|S|}),\]
where $\text{NC}(l)$ is the lattice of non-crossing partitions on
$\{1,\ldots,l\}$ and $\mu$ is the Mobius inversion function on this lattice,
see e.g.\ \cite[Eq.\ (11.5)]{nicaspeicherlectures}. Then
\[\kappa_l^\cD(g^*g,\ldots,g^*g)=p_1
\sum_{\pi \in \text{NC}(l)}
\mu(\pi,\{1,\ldots,l\})\prod_{S \in \pi} \tau_1((g^*g)^{|S|})
=p_1\kappa_l^{\C,1}(g^*g,\ldots,g^*g),\]
where $\kappa_l^{\C,1}$ are the scalar-valued free cumulants for trace
$\tau_1$. Recalling (\ref{eq:RDpalpha}), we obtain
$\cR_{g^*g}^{\cD}(p_1\alpha)=p_1\cR_{g^*g}^{\C,1}(\alpha)$,
which concludes the proof.
\end{proof}